\documentclass[12pt]{amsart}
\usepackage{tikz}
\usetikzlibrary{calc, arrows, decorations.markings} \tikzset{>=latex}
\usepackage{amsmath,amssymb,amsfonts}
\usepackage{amscd, amssymb, latexsym, amsmath, amscd}
\usepackage[all]{xy}
\usepackage{pb-diagram}
\usepackage{enumerate}
\usepackage{anysize}
\marginsize{3cm}{3cm}{3cm}{3cm}
\input xy
\xyoption{all}
\usepackage{pb-diagram}
\usepackage[all]{xy}
\input xy
\xyoption{all}

\theoremstyle{plain}
\newtheorem{thm}{Theorem}

\newtheorem{lemma}{Lemma}

\newtheorem{corollary}{Corollary}

\newtheorem{proposition}{Proposition}

\newtheorem{conj}{Conjecture}

\theoremstyle{definition} \theoremstyle{definition}
\newtheorem{remark}{Remark}

\theoremstyle{remark}

\newcommand{\g}{\mathfrak{g}}
\newcommand{\A}{\mathbb{A}}
\newcommand{\Q}{\mathbb{Q}}
\newcommand{\LG}{{}^L{G}}
\newcommand{\LT}{{}^L{T}}
\newcommand{\LM}{{}^L{M}}
\newcommand{\kk}{\mathfrak{k}}

\newcommand{\p}{\mathfrak{p}}

\newcommand{\ad}{{\rm ad}}

\newcommand{\ttt}{\mathfrak{t}}

\newcommand{\U}{\mathcal{U}}
\newcommand{\Z}{\mathbb{Z}}
\newcommand{\G}{\mathbb{G}}

\newcommand{\R}{\mathbb{R}}
\newcommand{\Gm}{\mathbb{G}_m}

\newcommand{\C}{\mathbb{C}}
\newcommand{\Si}{\mathbb{S}}
\newcommand{\F}{\mathbb{F}}
\newcommand{\N}{\mathbb{N}}
\newcommand{\out}{\rm op}
\newcommand{\wG}{\widehat{G}}
\newcommand{\wM}{\widehat{M}}
\newcommand{\wT}{\widehat{T}}
\newcommand{\wS}{\widehat{S}}
\newcommand{\der}{\rm der}
\newcommand{\scon}{\rm sc}
\newcommand{\Fr}{\rm Fr}
\newcommand{\Nm}{\mathbb{N}m}

\newcommand{\Hom}{{\rm Hom}}
\newcommand{\Ind}{{\rm Ind}}

\def\M{{\rm M}}
\def\G{{\rm G}}
\def\Aut{{\rm Aut}}
\def\SL{{\rm SL}}
\def\Spin{{\rm Spin}}

\def\PSO{{\rm PSO}}
\def\PGSO{{\rm PGSO}}

\def\PGSp{{\rm PGSp}}
\def\Sp{{\rm Sp}}
\def\St{{\rm St}}

\def\U{{\rm U}}
\def\M{{\rm M}}

\def\GL{{\rm GL}}
\def\PGL{{\rm PGL}}

\def\Gal{{\rm Gal}}
\def\SO{{\rm SO}}
\def\OO{{\rm O}}

\def\ad{{\rm ad\, }}

\begin{document}

\title[A relative local Langlands correspondence]
{A `relative' local Langlands correspondence}
\author{Dipendra Prasad}
\address{Tata Institute of Fundamental 
Research, Homi Bhabha Road, Bombay - 400 005, INDIA.}
\email{dprasad@math.tifr.res.in}


\begin{abstract}
For $E/F$ quadratic extension of local fields and $G$ a reductive algebraic group over $F$, 
the paper formulates a conjecture classifying irreducible admissible representations of $G(E)$ which
carry a $G(F)$ invariant linear form, and the dimension of the space of these invariant forms,
 in terms of the Langlands parameter of the representation. 
The paper studies parameter spaces of Langlands parameters, and morphisms between them associated to 
morphisms of $L$-groups. The conjectural answer to the question on 
the space of $G(F)$-invariant linear forms is in terms of fibers of a 
particular finite map between parameter spaces. 
\end{abstract}

\maketitle
{\hfill \today}

\tableofcontents

\newpage

\section{Introduction}

`Relative' Langlands program, as this phrase has come to be used, is the analogue 
for $L^2(H\backslash G)$, say for spectral decomposition,  for suitably chosen subgroups $H$ of a reductive group 
$G$, of what the `usual' 
Langlands program is for $L^2(G)$; see for instance the work of Sakellaridis and Venkatesh \cite{sav} where 
the authors setup a large framework to deal with this question in considerable generality.   
The relative Langlands program  has both local and global aspects, and just as for the usual
Langlands program where trace formulae play a pivotal role, for the relative Langlands program, the relative trace 
formula of H. Jacquet introduced in \cite{JLai} ---itself inspired by the work [HLR] on Tate conjectures for Hilbert modular surfaces--- plays a pivotal role. 

For a reductive algebraic group $G$ over a local field $F$, 
and a subgroup $H$ of $G(F)$, a version of the local question  is to simply classify
irreducible admissible representations $\pi$ of $G(F)$ such that $\Hom_{H}[\pi, \C] \not = 0,$ 
and to understand these dimensions
(or, multiplicities) when nonzero; this is actually what we will be considering in this paper. A more specific
motivating question for this paper was  how does $\dim \Hom_{H}[\pi, \C] $ vary 
as we vary the representations $\pi$ of $G$ 
inside an $L$-packet of $G$; there is the further possibility to vary the pairs $(G,H)$ in their innerforms.   
Representations of $G(F)$ which carry nonzero $H$-invariant
linear form are often said to be distinguished by $H$; more generally, if $\omega: H \rightarrow \C^\times$ 
is a character, one talks of $\omega$-distinguished representations as those irreducible admissible
representations of $G(F)$ for which $\Hom_{H}[\pi, \omega] \not = 0.$ 

Among many examples of $H$, the most studied cases have been 
the symmetric spaces $H\backslash G$ where $H$ is the fixed points of an involution on $G$. There is a vast literature 
on this in the Archimedean case 
extending the work of Harish-Chandra for $\Delta(G)\backslash G \times G$; especially one knows when there are 
discrete summands in $L^2(H\backslash G)$, and even a complete knowledge of the Plancherel measure is known.

 There is an even 
simpler example where $H$ is the fixed points of a Galois involution, i.e., $G=G(E)$, $H=G(F)$ for $E/F$ a separable quadratic extension of local fields. This 
has been studied at great length especially by H. Jacquet for the pairs $(\GL_n(E),\GL_n(F))$ and $(\GL_n(E), \U_n(F))$
but  even in these cases, the results obtained are not in the most general cases
but essentially only for discrete series representations of $\GL_n(E)$. As has been known to workers 
in this field, especially for the pair $(\GL_n(E), \U_n(F))$, multiplicities which are one for discrete 
series representations can be higher for non-discrete series representations, and in spite of many attempts,
have continued to be a bit elusive.

The aim of this work is to make some beginnings for this Galois case of the relative Langlands program.
Thus, let $G$ be a reductive algebraic group over a local field $F$, and $E$ 
a quadratic extension of $F$. The principal aim of this paper is  to 
formulate a conjecture about representations of $G(E)$ distinguished by $G(F)$
in terms of the Langlands parameters of the representations of $G(E)$, or what is also called the 
Langlands-Vogan parametrization.
A special case of  this conjecture is for the  degenerate 
case of $E=F+F$, in which case the question amounts to understanding
 the contragredient of a representation of $G(F)$ in terms of Langlands-Vogan
parameters, a question that we will take up first.

As the reader will notice, the conjectured answer to the question: which 
irreducible admissible representations of $G(E)$ are distinguished by $G(F)$, and then the refined question on 
multiplicities 
$\dim \Hom_{H}[\pi, \C] $,  are given in terms of the algebraic geometry
of (functorial) maps between affine algebraic varieties---the parameter spaces of Langlands parameters---which appear in the following commutative diagram (for a homomorphism $\varphi: {}^LG_1 \longrightarrow {}^LG_2$ of $L$-groups):
 $$
\begin{CD}
X_1 = \Hom(W'_F, {}^LG_1(\C)  )
@>\Phi'>>  X_2 = \Hom(W'_F, {}^LG_2(\C))  \\ 
 @ V{\pi_1}VV  @VV{\pi_2}V  \\
\Sigma_F(G_1)= X_1/\!/\widehat{G}_1(\C)
@> \Phi >>  \Sigma_F(G_2)=X_2/\!/\widehat{G}_2(\C).
\end{CD}
$$

We remind the reader that the map of algebraic varieties $X_1/\!/\widehat{G}_1(\C) \stackrel{\Phi}\rightarrow X_2/\!/\widehat{G}_2(\C)$, `remembers' the fact that $X_1/\!/\widehat{G}_1(\C)$  
and $X_2/\!/\widehat{G}_2(\C)$ arose 
as quotients, and therefore for points of $X_1/\!/\widehat{G}_1(\C)$  corresponding to
 closed orbits of $\wG_1(\C)$ on $X_1$ (closed orbits $\widehat{G}_1(\C)\cdot x_1 \subset X_1$ correspond exactly to elements in 
$x \in X_1 = \Hom(W'_F, {}^LG_1(\C)  )$ which are called {\it admissible} in [Bo] or also $F$-admissible in the literature),
$\Phi$ induces a map from stabilizer in $\widehat{G}_1(\C)$ of such a point in $X_1$ to the stabilizer in $\widehat{G}_2(\C)$ of the 
corresponding point in $X_2$. In particular, there is a map on  the group of connected components of the stabilizers.
The mapping $\Phi$ together with these `extra' features corresponding to associated {\it stacks} 
will play a dominant role in our considerations in this paper in which the functorial map 
${}^LG_1(C)\rightarrow {}^LG_2(\C)$ corresponds to quadratic basechange for a group $G_1$ closely associated
to the group $G$ that we have denoted in this paper by $G^{\out}$ with $G^{\out}(E)=G(E)$, and the group $G_2 =R_{E/F}G = R_{E/F}G^{\out}$.

The proposed conjecture in this paper was first observed in the case of $\SL_2$ where 
the author noticed that given an irreducible admissible representation $\pi$ of $\SL_2(E)$, the dimension of the space of 
$\SL_2(F)$ invariant forms on $\pi$ is either 0 or a number $m(\pi)$ which depends not so much on $\pi$ but only 
on the $L$-packet to which $\pi$ belongs, and that $m(\pi)$ is the number of distinct lifts of the (equivalence classes) of the parameter 
$\sigma_\pi: W_E\rightarrow \PGL_2(\C)$ to the (equivalence classes) of the parameters $\tilde{\sigma}_\pi: W_F \rightarrow \PGL_2(\C)$. The 
conjecture proposed in this paper --- Conjecture \ref{conj3} in section 13 --- is considerably more complicated than this, but the observation above on $\SL_2$  was an important starting point for this work. The first refinement to counting the number of lifts was to count them with multiplicites which takes 
into account the ramification of the map on parameter spaces (the map $\Phi$ in the above diagram). 
Then, since parameters see not just one group, but all
groups which are pure innerforms of each other, presumably the number we get 
(by counting elements in the fiber together with multiplicity) does not relate  to 
$\dim \Hom_{G(F)}[\pi, \C]$ for one $G(F)$ but a sum of terms for 
all possible pure innerforms $G'(F)$ with $G'(E)=G(E)$. So we need to isolate 
terms responsible for a given $G(F)$. This was easy enough for $p$-adics using by now well-known theorems of Kottwitz on 
Galois cohomology of $p$-adic groups, but for reals it needed inputs on 
Galois cohomology of real groups from works of Adams and Borovoi, some of which is relatively new.

Although not a new point of view, cf. the book [ABV] in the real case, the paper emphasizes how geometry 
of Langlands parameters influences questions in representation theory, which here is about the dimension of the space
of invariant linear forms under a specific subgroup.    The word `relative' thus has 
similar connotation as its usage in Algebraic Geometry where it is important not only to study `spaces' but also morphisms between them.
  
 Results about $\SL_2(F) \subset 
\SL_2(E)$ involving nontrivial $L$-packets from [AP03] 
were an important guide in formulating the conjectural answer in this paper, 
besides many works of Jacquet on distinction of representations of
$\GL_n(E)$ by $\U_n(F)$ which were refined and extended in the work of Feigon-Lapid-Offen in [FLO].

We end the introduction by mentioning two obvious questions. First, as far as the author 
has been able to understand from the existing literature on unitary groups, there is no 
approach to parametrizing the space of invariant linear forms in terms of fibers of the basechange map. (One simplifying aspect 
of the pair $(\GL_n(E),\U_n(F))$ is that at least for discrete series representations of $\GL_n(E)$, the space of $\U_n(F)$-invariant forms is one, something which is not true in general for $(G(E),G(F))$, for example as [AP03] shows, not even for $\SL_2$. Our conjecture suggests, however, that there should be multiplicity one for $(G(E),G(F))$ for discrete series representations of $G(E)$ 
which are `stable', i.e., 
the centralizer of the parameter is just the center of the $L$-group.) 
The work [Pr1] in the 
finite fields, and the further work of Lusztig [Lu], suggests that a `character theoretic' proof might work in some generality as was the case with the work of Waldspurger proving the branching laws from $\SO_n$ to $\SO_{n-1}$.
Nonetheless, it would be 
good to have some geometric proofs say in the case of function fields.

Second, one would like to understand  what's the natural larger context in which one could view the questions on classification of
 irreducible admissible representations of $G(E)$ distinguished by $G(F)$ for $E/F$ quadratic, and then the multiplicities? Perhaps
the best place is for pairs $(G,H)$ where $H$ arises as fixed points of a `diagram' automorphism of $G$ of order 2. 
Jacquet in [J91] has suggested that 
representations of $\GL_n(F)$ distinguished by $\OO_n(F)$ arise from
the metaplectic correspondence of representations due to Kazhdan and Patterson 
between $\widetilde{\GL}_n(F)$ (a two fold cover of $\GL_n(F)$) and $\GL_n(F)$ (say for cuspidal automorphic representations of 
$\GL_n(\A)$). Our suggestion amounts to pleasanter properties for multiplicities for $n$ odd, which corresponds to fixed points of a 
diagram automorphism. The work of J. Hakim and J. Lansky for the pair $(\GL_n(F), \OO_n(F))$ in [Ha], [H-L] gives 
some confidence to this hope.

\section{Summary}
It may be helpful to  the reader if I summarized what's done in this paper.

In section 3,  I basically set the notation, and some of the concepts that we will
repeatedly use in the paper often without explicit mention.

In section 4, I make the conjecture on the Langlands-Vogan parameter of the 
contragredient representation. It is built on three ingredients. The first, the 
most obvious being that the parameter, must go to the `dual parameter', effected 
by the Chevalley involution on $\LG$; second, the character of the component group must
go to the dual character; finally, we have to take into account 
the effect on `base point':
since if  $\pi$  has a Whittaker model for a character $\psi: N \rightarrow \C^\times$, 
$\pi^\vee$  has a Whittaker model not for the character $\psi$, but for $\psi^{-1}$.

In section 5, I compare the proposal on the contragredient given in section 4 to the work of J. Adams for 
real reductive groups. The conjecture in section 4 is in terms of an involution on $G_0(\R)$, $G_0$ quasi-split over $\R$, 
 which comes from an element of $T_0(\R)$ which operates by $-1$ on all simple roots of a Borel subgroup $B_0$ over $\R$
containing $T_0$. J. Adams uses an involution on $G_0(\R)$ which operates by $-1$ on a fundamental Cartan subgroup.
The main part of this section identifies the two involutions by relating most split torus in $G_0(\R)$ with most compact
torus in $G_0(\R)$ by use of the Jacobson-Morozov theorem and the `Cayley transform' in $\SL_2(\R)$, 
and noting that the element 
$j = \left ( \begin{array}{cc} i & 0 \\ 0 & -i\end{array}\right ) \in \SL_2(\C)$ 
is the desired one both for me and for J. Adams.

In section 6, I discuss the question of distinction of representations of $G(E)$ by $G(F)$ for $E/F$ a quadratic extension 
of finite fields, in which I basically recall a theorem I proved in [Pr1], but which was not stated in this form there.

In section 7, using $-w_0$, where $w_0$ is the longest element in the Weyl group, 
I construct an involution  on the set of quasi-split reductive groups over any field $F$, denoted
$G \rightarrow G^{\out}$. Over reals, this involution  interchanges $\Si$ and $\R^\times$, and 
fixes the torus $R_{\C/\R} \Gm = \C^\times$.

In section 8, we define a character $\omega_G: G(F) \rightarrow \Z/2$ for any reductive group $G$ over any 
local field $F$, and a quadratic extension $E$ of $F$. This character factors through $G^{\ad}(F)$, and is the `same'
for all groups which are innerforms of each other. This character 
is a measure of the half sum of positive roots of $G$ being not a weight for $G$. 

In section 9, we analyze what are the possible lifts of a parameter for $G(E)$ to parameters for $G(F)$, and find
that these are described by some cohomology  spaces --- by some well-known simple generalities--- but 
which will play a basic role in the rest of the paper.

In section 10, we analyze the parameter spaces for Langlands parameters, mostly in the non-Archimedean case.

In section 11, using a theorem of Kottwitz, we find that to lifts of parameters of $G(E)$ to $G^{\out}(F)$, 
we can associate a set 
of pure innerforms of $G$. This is done by analyzing the action of character twists 
(characters of $G^{\out}(F)$ which become trivial on $G(E)$) on these lifted parameters
to $G^{\out}(F)$. This way  we associate as many pure innerforms of $G(F)$ (to a particular orbit of character twists 
of a parameter of $G^{\out}(F)$) as the number of elements in the orbit.

In section 12, gives a way to associate to lifts of a parameter  of $G(\C)$ to $G^{\out}(\R)$, 
a 
pure innerform of $G(\R)$. It hinges on  theorems of Adams and Borovoi 
calculating 
cohomology $H^1(\Z/2, G(\C))$ 
in terms of a Galois cohomology, when $\Z/2$ 
operates on $G(\C)$  by an involution. The cohomology groups 
$H^1(\Z/2, M(\C))$, for $M$ a Levi subgroup in $\wG(\C) $ 
appear in this section as they parametrize the  set of lifts of  a 
parameter of $G(\C)$ to $G^{\out}(\R)$.

In section 13, the conjecture on the dimension of the space of $G(F)$-invariant forms on irreducible admissible representations 
of $G(E)$ is made as Conjecture \ref{conj3} in this section. It is the {\it raison d'\^etre} for the whole work.

In section 14, we verify the conjecture for tori. The most subtle part in this section is to confirm that the condition 
on the character of the component group which appears in Conjecture \ref{conj3} is exactly what naturally comes up in this analysis.

In section 15, we discuss $\SL_2$. In particular, we prove that a representation of $\SL_2(E)$ distinguished by 
$\SL_2(F)$ must have a Whittaker model  by a character of $E$ trivial on $F$. We also analyze the condition 
on component groups. The factor $d_0(\tilde{\sigma})= \left |{\rm coker} \{\pi_0(Z(\tilde{\sigma})) \rightarrow 
\pi_0(Z(\sigma))^{\Gal{(E/F)}} \} \right|, $ in Conjecture \ref{conj3} has its origin in 
case II of the table at the end of the section 
(this is the case of base change of a principal series for $\GL_2(F)$ which has no selftwists over $F$, 
but has one after basechange to $E$).

In section 16, we relate Conjecture \ref{conj3} to known and expected results for $\GL_n$ and $\U_n$.

In section 17, we specialize to real groups, and state some simpler forms of Conjecture \ref{conj3} made in section 13, as well as some 
natural analogue of it for $L^2(G(\R)\backslash G(\C))$.

In section 18, we try to prove Conjecture \ref{conj3} for (certain) irreducible principal series representations of $G(\C)$ by using the 
(unique) closed orbit of $G(\R)$ on $B(\C)\backslash G(\C)$. A noteworthy feature of the analysis is the appearance
of the quadratic character $\omega_G: G(\R) \rightarrow \Z/2$.

In section 19, we try to prove Conjecture \ref{conj3} for (certain) 
irreducible principal series representations of $G(\C)$ by using open 
 orbit of $G(\R)$ on $B(\C)\backslash G(\C)$. In the process, we completely analyze open orbits 
of $G(\R)$ on $B(\C)\backslash G(\C)$, something which is well-known, but we give independent proofs.

In section 20, specific suggestions on $\dim \Hom_{\U(p,q)}[\pi, \C]$ are made for irreducible principal series
representations of $\GL_{p+q}(\C)$.

\section{Notation and other preliminaries} 
In this paper, $G$ will always be a connected reductive algebraic group 
over a local field $F$, which can be either Archimedean, or non-Archimedean. When speaking of connected real reductive 
groups, we mean $G(\R)$ for $G$ connected as an algebraic group; thus $G(\R)$ might well be disconnected such as
$\SO(p,q)(\R)$. We will use $\Si$ to denote the group  of norm 1 elements in $\C^\times$, and $\Si^n$, its $n$-th power.  

There will be another source of reductive groups in the paper, the dual 
group $\wG$ which will interchangeably be written as $\wG(\C)$, and the $L$-group ${}^L G= \wG \rtimes W_F$ where $W_F$ is the Weil group of the local field $F$.  We will use $W'_F$ for the Weil-Deligne group of $F$, which is 
$W_F$ itself if $F$ is Archimedean, and is $W_F \times \SL_2(\C)$ if $F$ is non-Archimedean. We will need to consider subgroups $Z_\varphi$ 
of $\wG(\C)$ which arise as centralizer in $\wG$ of Langlands parameters $\varphi: W'_F \rightarrow {}^LG$. These 
groups $Z_\varphi$ need not be connected, and we denote the group of their connected components as $\pi_0(Z_\varphi)$, 
a finite group. 

Center of a group $G$ will be denoted by $Z(G)$, or simply $Z$ if the context makes the group $G$ clear. The center
of an algebraic group $G$ defined over $F$ is an algebraic group defined over $F$.

We will often find it convenient to use cohomological language in this paper. For example, equivalence classes
of Langlands parameters will be considered as elements of $H^1(W_F,\wG)$, 
or of $H^1(W'_F,\wG)$. Because of the inflation-restriction exact sequence, for any finite extension $E$ of $F$,
$H^1(\Gal(E/F),\wG^{W_E})$ 
sits inside $H^1(W_F,\wG)$. Thus  for $E/F$ quadratic, we will have many occasions to
consider the group cohomology, $H^1(\Gal(E/F),\wG^{W_E})$ of the group $\Gal(E/F) = \Z/2$, with possibly
complicated action on $\wG^{W_E}$. 

There will be another reason to use cohomology in this paper, what's usually called Galois cohomology,
$H^1(\Gal(\bar{F}/F), G(\bar{F}))$ where $\bar{F}$ is a fixed algebraic closure of $F$;
these are often  written as $H^1(F,G)$. These groups appear in this paper 
as they classify forms of $G$, more precisely the pure innerforms of $G$ over $F$; for example,
$H^1(\Gal(\bar{F}/F), \SO_n(\bar{F}))$ classifies isomorphism classes of 
quadratic forms over $F$ of a given discriminant; it 
is possible that two  non-isomorphic quadratic forms give rise to the same isometry groups, in which case the 
isometry groups $\SO(V_1)$ and $\SO(V_2)$ will be treated as distinct objects.

It will be useful to us that the two usages of the cohomology in the last two paras do come together, say for a connected 
reductive group $G$ over $\R$. In this case, by a theorem due to J. Adams, for a reductive 
group $G$ over $\R$, $H^1(\Gal(\C/\R), G(\C)) \cong H^1(\Z/2, G(\C))$ 
where $\Z/2$ operates on $G(\C)$ by an algebraic automorphism which is the complexification of 
a Cartan involution on $G(\R)$. This will allow us to associate to  various lifts to $G^{\out}(\R)$ of a 
parameter for $G(\C)$ both a certain multiplicity as well as a pure innerform $G_\alpha(\R)$ of $G(\R)$ which will be related
to $\dim\Hom_{G_\alpha(\R)}[\pi, \C]$.

Explicit Galois twistings will often be used in this paper. Recall that two algebraic groups $G_1$ and $G_2$ 
over a field $F$ are said to be 
forms of each other if they become isomorphic over $\bar{F}$, the algebraic closure of $F$. We will be especially 
considering instances where the groups become isomorphic over a quadratic extension $E$ of $F$. 
Given an algebraic group $G$
over $F$, an automorphism $\phi$ of $G$ over $E$, a quadratic extension of $F$, with the property that 
$\phi \circ \phi^\sigma = 1$, define an algebraic group $G^\phi$ over $F$, said to be obtained from $G$ by twisting 
by the cocycle $\phi$,  having  the characteristic 
property $G^\phi(F) = \{g \in G(E)| \phi(g) = \sigma(g)\}$; we will  often use the notation
$g^\sigma$ for $\sigma(g)$.      

We will use this Galois twisting especially for maximal tori $T$ inside a reductive group $G$. 
Recall that if $N(T)$ is the normalizer of a maximal torus $T$, then essentially by definitions,
the set of tori of $G$ defined over $F$ 
is the set $(G/N(T))(F)$. Therefore, the set of tori of $G$ defined over $F$ up to $G(F)$-conjugacy
can be identified 
to the kernel of the natural map of pointed sets $H^1(\Gal(\bar{F}/F), N(T)) \rightarrow H^1(\Gal(\bar{F}/F),G)$. 
Two maximal tori $T_1$ and $T_2$ in $G$ are said to be 
stably conjugate if there is an isomorphism $\lambda: T_1 \rightarrow T_2$ defined over $F$ which is 
given by conjugation of an element in $G(\bar{F})$. By Proposition 6.1 of [Re],  the natural map
$H^1(\Gal(\bar{F}/F), N(T)) \rightarrow H^1(\Gal(\bar{F}/F),W)$, where $W=N(T)/T$ is the Weyl group of $T$,
identifies the set of stable conjugacy classes of maximal 
tori in $G$ with $H^1(\Gal(\bar{F}/F),W)$, when $G$ is quasi-split and
$T$ is a  maximally split maximal torus in $G$. (This is, in essence, a theorem of M.S. Raghunathan.)

In particular, if an automorphism $\phi$ 
of $T$ arises from an element $w$ of $W(E)$, where $W$ is the Weyl group $W= N(T)/Z(T)$, with the property 
that $w\cdot w^\sigma =1$ where $\sigma$ is the generator of $\Gal(E/F)$, a quadratic extension, 
then  
the twisted torus $T^\phi= T^w$ is a well-defined stable conjugacy class of a torus contained in $G$. 

The exact sequence of algebraic groups over $F$:
$$1 \rightarrow T \rightarrow N(T) \rightarrow W \rightarrow 1,$$
gives rise to an exact sequence of pointed sets:
$$ \cdots \rightarrow W(F) \stackrel{\delta}\rightarrow H^1(F, T) \rightarrow H^1(F,N(T)) \rightarrow \cdots;$$
in fact,  more is true. The group $W(F)$ operates on $H^1(F,T)$, an affine action, not one which preserves the group
structure on $H^1(F,T)$. (A group $G$ operating on an abelian group $A$ is said to operate by an affine action
if it operates on $A$ through a homomorphism into the group $ A \rtimes \Aut(A)$.) The fibers of the map from 
$H^1(F,T)$ to $H^1(F,N(T))$ are exactly the orbits of $W(F)$. To describe the action of $W(F)$ on $H^1(F,T)$ note that 
$W(F)$ naturally operates on $H^1(F,T)$ through its action on $T$ which we denote  as $w\cdot \alpha$ 
Then the action of $w\in W(F)$ on $\alpha \in H^1(F,T)$, denoted $\alpha^w$,  is  $\alpha^w = w^{-1}\cdot \alpha + 
\delta(\alpha)$. See Serre [Se], Proposition 40, I, \S5.6, for all this.  An example to keep in mind is 
the maximal compact torus $T=\Si^n$ in $\Sp_{2n}(\R)$, in which case $H^1(\R,T) = (\Z/2)^n$ acted upon by $W(\R)$ 
 which is $(\Z/2)^n \rtimes S_n$ by the natural affine action, so there is exactly one orbit. Contrast it with
the maximal compact torus $T=\Si^n$ in $\U_{p,n-p}(\R)$, in which case $H^1(\R,T) = (\Z/2)^n$ acted upon by $W(\R)$ 
 which is $S_n$ with its  natural action on $(\Z/2)^n$, so there are exactly $(n+1)$ orbits. 

Stably conjugate tori share many properties: they are isomorphic as algebraic groups, their 
Weyl groups are isomorphic as algebraic groups---in particular $W(F)$ is meaningful, 
and the image of $H^1(F,T) \rightarrow H^1(F,G)$ are the same
for stably conjugate tori: 
$$
\xymatrix{
T_1\ar@{->}[dd]_{\lambda}\ar@{->}[drr]^{i_1}&  & \\
& & G. \\
T_2 \ar@{->}[urr]_{i_2} &   &
}
$$

Let $G$ be 
a quasi-split reductive group over $F$,  and $(B,T)$ a pair 
consisting of a Borel subgroup $B$, and a maximal torus $T$ contained in $B$
and $S$ the maximal split torus in $T$. 
The Weyl group $W = N(T)/T$ has  $W(F) = N(T)(F)/T(F) = N(S)(\bar{F})/T(\bar{F})$.

One of the most important elements in the 
Weyl group for us will be $w_0$, which takes $(B,T)$ to $(B^{-},T)$ with $B^- \cap B = T$. The 
element $w_0 \in W$ is defined over $F$. We will denote by $-w_0$, the automorphism of $T$ which is
$t\rightarrow w_0(t^{-1})$. There is an automorphism $C_0$, called the Chevalley involution, 
of $G$ of order 1 or 2 defined over $F$ 
which fixes the pair $(B,T)$, and has the effect $-w_0$ on $T$, and is usually defined by fixing a `pinning', i.e.,
a set of nonzero elements $\{X_\alpha \}$ in the various simple root spaces $N_\alpha $ contained in the Lie algebra
$\g$ of $G$.   For a quadratic extension $E$ of $F$, sending the non-trivial element of
$\Gal(E/F)$ to $C_0$, and twisting the group $G$ by this involution, we get a quasi-split 
group over $F$ which
 is denoted in this paper as $G^{\out}$. If $G=T$ is a 
torus, then $T^{\out}$ is obtained by twisting $T$ by the automorphism $t\rightarrow t^{-1}$. 
The groups $T$ and $T^{\out}$ sit in the following exact sequence:
$$1 \rightarrow T \rightarrow R_{E/F}(T) \rightarrow T^{\out} \rightarrow 1.$$

We will also use the notation $A^{\out}$ for any commutative group scheme $A$ over a field $F$ twisted by the automorphism
$x\rightarrow -x$ on $A$ through a quadratic extension $E$ of $F$.
 
We will have many occasions to use a theorem of Kottwitz which calculates $H^1(\Gal(\bar{F}/F), G(\bar{F}))$
for $G$ connected reductive, and $F$ non-Archimedean. According to this theorem, $H^1(\Gal(\bar{F}/F), G(\bar{F}))$ 
is a finite abelian group which is isomorphic to the character group of $\pi_0(\wG^{W_F})$.

For any finite or more generally topological group $H$, we let $H^\vee$ denote the character group of $H$, i.e.,
the group of continuous homomorphisms from $H$ to $\C^\times$; $H^\vee$ is a topological group.

For an abelian algebraic group $A$ over a field $F$, we will let $A^\vee = \Hom(A,\Gm)$ be 
the Cartier dual of $A$. For diagonalizable groups $A$ such as tori or finite abelian groups over $F$, 
$A^\vee$ is a finitely generated $\Z$-module with an action of $\Gal(\bar{F}/F)$. 
For $A=T$, a torus, $A^\vee= X^\star(T)$, the character group of $T$. 

Thus the notation $H^\vee$ and $A^\vee$ in the last two paras do create a bit of ambiguity say for finite groups which naturally happen to be algebraic, but
in such situations, $H^\vee=A^\vee$, as abstract groups, but $A^\vee$ carries the extra information of being a 
$\Gal(\bar{F}/F)$-module. For instance, for a reductive group $G$ over a local field $F$, the identification 
$Z(G)^\vee = \pi_1(\wG)$ (of finitely generated $\Z$-modules) is as $W_F$-modules.

We will often use the Tate-Nakayama duality for a diagonalizable group $A$, according to which there is a perfect
pairing for $F$ a local field:
$$H^1(W_F, A) 
\times H^1( W_F,  A^\vee) \rightarrow H^2(W_F, \Gm),$$
with $H^2(W_F, \Gm) = \Q/\Z$ if $F$ is non-Archimedean, and $\Z/2$ if $F=\R$.

Representations of $G(F)$ will be smooth representations for $F$ non-Archimedean, and smooth Frechet representations 
of moderate growth for $F$ Archimedean.

For a representation $\pi$ of $G(F)$, $\pi^\vee$ denotes its contragredient.

For a reductive group $G$ 
over $\C$ (possible a finite group) 
operating on a variety $X$ over $\C$, we will use the standard notation
$X/\!/G$ to denote what's usually called the categorical quotient, which for affine varieties corresponds to ring of $G$ invariants.

\section{The contragredient}

Let $W_F$ be the Weil group of $F$, and $W'_F$ the Weil-Deligne group of $F$.
Let ${}^LG(\C) = \widehat{G}(\C) \rtimes W_F$, 
be the $L$-group of $G$ which comes equipped with a map onto
$W_F$. An admissible homomorphism $\varphi: W'_F\rightarrow {}^LG$
is called a Langlands parameter for $G$.  To an admissible homomorphism 
$\varphi$, is associated the group of connected components $\pi_0(Z_{\varphi} )
= Z_{\varphi}/ Z^0_{\varphi}$ where ${Z_{\varphi}}$ is the centralizer of
$\varphi$ in $\widehat{G}(\C)$, 
and ${Z^0_{\varphi}}$ is its connected component of 
identity.

According to the Langlands-Vogan parametrization, to an irreducible
admissible representation $\pi$ of $G(F)$, there corresponds a pair 
$(\varphi, \mu)$ 
consisting of an admissible homomorphism 
$\varphi: W'_F\rightarrow {}^LG$, and a representation $\mu$ of 
its group of connected components $\pi_0(Z_{\varphi})$. The pair $(\varphi, \mu)$ 
determines $\pi$ uniquely, and one knows which pairs arise in this way.
The Langlands-Vogan parametrization depends on fixing a base point, 
consisting of a pair $(\psi,N)$ where $\psi$ is a 
nondegenerate character on $N=N(F)$ which is a maximal unipotent subgroup in 
a quasi-split innerform of $G(F)$.

Before we come to the description of the Langlands parameter of $\pi^\vee$ in terms of that of $\pi$, we recall that for a simple algebraic group $G$ over 
$\C$, one has either:
\begin{enumerate}
\item  $G(\C)$ has no outer automorphism, or $G(\C)$ is $D_{2n}$ for some $n$; 
in these cases,
 all irreducible finite dimensional representations of $G(\C)$ are self-dual; or,
\item $G(\C)$ has an outer automorphism of order 2  which has the property
that it takes all irreducible
finite dimensional representations of $G(\C)$ to their contragredient.
\end{enumerate}

The above results about simple algebraic groups (over $\C$) can be extended to all
semisimple algebraic groups, and then further to all reductive algebraic groups, and 
give in this generality of reductive algebraic groups, an automorphism $\iota$ of $G(\C)$ 
of order 1 or 2 which is well defined as an element in $\Aut(G)(\C)/j(G(\C))$, where $j$ is the natural map 
$j: G(\C)\rightarrow \Aut(G)(\C)$,  which takes any finite dimensional
representation of $G(\C)$ to its dual. (For example, for a torus $T$, $\iota$ is just the inversion, 
$\iota: z\rightarrow z^{-1}, z \in T$.)
 
One can also define such an automorphism $\iota$ for a quasi-split 
group $G(F)$, for $F$ any field. For this we fix a Borel subgroup $B$ containing a maximal torus $T$, 
and $\psi: N \rightarrow F$ a non-degenerate character. The group $G$ has an automorphism $\iota_{B,T,N,\psi}$ defined over $F$ 
which takes the pair $(T,B)$ to itself, $\psi$ to $\psi^{-1}$, and its effect on simple roots of $T$ on 
$N(\bar{F})$ is that of the diagram automorphism $-w_0$ where $w_0$ is the longest element in the Weyl group
of $G$ over $\bar{F}$. Given the quadruple $(B,T,N,\psi)$, the automorphism $\iota_{B,T,N,\psi}$ is uniquely defined; 
since any two tori in $B$ are conjugate by 
$N(F)$, the dependence of the automorphism $\iota_{B,T,N,\psi}$ on $T$ is only by conjugation by elements of $N(F)$. 
However, since it is the group $G^{\ad}(F)$ (and not $G(F)$)  which acts transitively on pairs
$(B,\psi)$, the element $\iota_{B,T,N,\psi}$ 
is well-defined only up to $G^{\ad}(F)$, but if one has moreover
fixed a $(B,\psi)$, which we will when  we discuss Langlands-Vogan parametrization, $\iota_{B,T,N,\psi}$
 is well-defined up to $N(F)$,
and abbreviate  $\iota_{B,T,N,\psi}$   to $\iota$, the duality automorphism of $G(F)$; its dependence on $N(F)$ is of 
no consequence as we consider the effect of $\iota$ on representations of $G(F)$.

\begin{remark}
It is easy to see that the automorphism $\iota_{B,T,N,\psi}$ considered as an element of 
$\Aut(G)(F)/jG(F)$ is independent of all choices if the action of $-w_0$ on
$H^1(F,Z(G))$,  through the action of $-w_0$ on $Z(G)$, is trivial. 
\end{remark}

Now we come to a basic lemma regarding the contragredient.

\begin{lemma} Let $G(F)$ be a quasi-split reductive algebraic group over a local field $F$, with $B=TN$ a Borel subgroup of $G$,
and $\psi: N(F)\rightarrow \C^\times$ a non-degenerate character. Then if an irreducible admissible representation $\pi$ of $G(F)$
has a Whittaker model for the character $\psi$, then the contragredient $\pi^\vee$ has a Whittaker model for the character
$\psi^{-1}: N \rightarrow \C^\times$.
\end{lemma}

\begin{proof}Although this lemma must be a standard one, the author has not found a proof in the literature, so here is one.
Observe that if $\pi$ is unitary, then $\pi^\vee$ is nothing but the complex conjugate $\bar{\pi}$ of $\pi$, i.e., 
$\bar{\pi} = \pi \otimes_{\C} \C$ where $\C$ is considered as a $\C$ module using the complex conjugation action. By applying complex conjugation to a  linear form $\ell: 
\pi \rightarrow \C$ on which $N(F)$ operates by $\psi$, we see that $\bar{\ell}$ defines a $\bar{\psi} = \psi^{-1}$-linear form on $\bar{\pi} = \pi^\vee$, proving the lemma for unitary representations, in particular for tempered representations of any reductive group. The general case of the Lemma follows by realizing a general generic representation of $G(F)$ as an irreducible  principal 
series representation arising from a  generic tempered representation $\mu$ up to a twist on a Levi subgroup $M(F)$ of $G(F)$; that this
can be done is what's called the standard module conjecture which has been proved.
Then one notes that if $N_M$ is the unipotent radical of a Borel subgroup in $M$,
then there is a general recipe relating the character $\psi_M:N_M(F)\rightarrow \C^\times$ with respect to which $\mu$ is generic, and 
the character $\psi: N(F)\rightarrow \C^\times$ with respect to which the induced principal series ${\rm Ind}_{P(F)}^{G(F)}\mu$
is generic. This allows one to prove the lemma, 
but we omit the details. 
\end{proof}

We next recall that in [GGP], section 9, denoting $G^{\rm ad}$ the adjoint group of $G$ 
(assumed without loss of generality at this point to be quasi-split since we want to construct 
something for the $L$-group),  there is constructed a homomorphism from 
$  G^{\rm ad}(F)/G(F) \rightarrow 
 {\pi_0(Z_{\varphi})}^\vee$, denoted $g \mapsto \eta_g$ which we recall here. For doing this, let 
$\widehat{G}^{\scon}
$ be the universal cover of $\widehat{G}$:
$$1 \rightarrow 
 \pi_1(\widehat{G}) 
\rightarrow \widehat{G}^{\scon} 
\rightarrow \widehat{G} \rightarrow 1 . $$
By the definition of universal cover, any automorphism of $\widehat{G}$ lifts to $\widehat{G}^{\scon}$ uniquely, 
and thus if we are given an action of $W_F$ on $\widehat{G}$ (through a parameter $\varphi: W_F\rightarrow {}^L{G}$),
 it lifts uniquely to an action of $W_F$ on $\widehat{G}^{\scon}$,
preserving  $ \pi_1(\widehat{G})$. Treating the above as an exact sequence of $W_F$-modules, and taking 
$W_F$-cohomology, we get the boundary map:
$$ \widehat{G}^{W_F} = Z_\varphi \rightarrow H^1(W_F,  \pi_1(\widehat{G})),$$ 
which gives rise to the map $$ \pi_0(Z_\varphi) \rightarrow H^1(W_F,  \pi_1(\widehat{G})).$$
On the other hand by the Tate duality (using the identification $Z(G)^\vee = \pi_1(\wG)$), 
there is a perfect pairing:
$$H^1(W_F, Z(G)) 
\times H^1( W_F,  \pi_1(\widehat{G})) \rightarrow \Q/\Z.$$
This allows one to  think of elements in $H^1(W_F, Z(G)) $ as characters on $H^1(W_F,\pi_1(\widehat{G}))$, and therefore 
using the map $ \pi_0(Z_\varphi) \rightarrow H^1(W_F,  \pi_1(\widehat{G})),$ as characters on $\pi_0(Z_\varphi)$.
Finally, given the natural map $G^{\ad}(F)/G(F) \rightarrow H^1(F,Z(G))$, we have constructed a group homomorphism from
$G^{\ad}(F)/G(F)$ to characters on  $\pi_0(Z_\varphi)$.

  Let $g_0$ be the unique conjugacy class in $G^{\rm ad}(F)$ 
representing   an element in $T^{\rm ad}(F)$ (with $T^{\rm ad}$ a maximally split, maximal torus in $G^{\rm ad}(F)$) which acts by $-1$ on all simple root spaces of $T$ on $B$. Denote the corresponding $\eta_{g_0}$ by
$\eta_{-1}$, a character on $\pi_0(Z_{\varphi})$, which will be the trivial character 
for example if $g_0$ can be lifted to $G(F)$.

The automorphism $\iota$ defined using a fixed pinning 
$(\wG,\widehat{B}, \widehat{N},\{X_\alpha\})$   of $\wG$ 
belongs to the center of diagram automorphisms of $\widehat{G}$, 
and hence extends to an 
automorphism of  ${}^LG(\C) = \widehat{G}(\C) \rtimes W_F$, which will again be denoted by 
$\iota$, and will be called the Chevalley involution of $\LG$.

\begin{conj} \label{conj1} Let $G$ be a reductive algebraic group over a local field $F$ which is a pure innerform of a 
quasi-split group $G_0$ over $F$ which comes equipped with a fixed triple $(B_0,N_0,\psi_0)$ used to parametrize 
representations on all pure innerforms of $G_0$. 
For an irreducible admissible representation $\pi$ of $G(F)$
with Langlands-Vogan parameter $(\varphi,\mu)$, the Langlands-Vogan parameter 
of $\pi^\vee$ is $(\iota\circ \varphi, (\iota\circ\mu)^\vee \otimes \eta_{-1})$, where $\iota$ is the Chevalley involution 
of ${}^L{G}(\C)$. 

For $G_0$ quasi-split over $F$, let $\iota = i_{B_0,N_0,\psi_0}$ be an
automorphism of $G_0$ defined over $F$ corresponding to the diagram automorphism $-w_0$ and taking $\psi_0$ to $\psi_{0}^{-1}$. If  $\pi_0(Z_{\varphi})$ 
is an elementary abelian 2-group (or its irreducible representations are selfdual), the dual representation $\pi^\vee$ 
is obtained by using the automorphism $\iota$ of $G_0(F)$.

\end{conj}

\begin{remark} A consequence of the conjecture is that for pure innerforms $G(F)$ of a quasi-split 
reductive group $G_0(F)$, if the component groups are known to be abelian, such as for $F=\R$
for any reductive group, or classical groups defined using fields alone, 
then all irreducible admissible representations 
of $G(F)$ are self-dual if and only if $-1$ belongs to the Weyl group of $\wG$, and there is an automorphism
$\iota_{B_0,N_0,\psi_0}$ of $G_0$ 
which arises by conjugation of an element in $T_0(F)$, a maximal torus in $B_0$, and which takes $\psi_0$ to $\psi_0^{-1}$.
\end{remark}

\vspace{3mm}

\noindent{\bf Example: } We explicate  conjecture \ref{conj1} for $\SL_n(F)$ which in this case is an easy exercise. 
Let $\ell_a$ denote the natural action 
of $F^\times/F^{\times n}$ on $\SL_n(F)$ (corresponding to the conjugation action of 
$\GL_n(F)/F^\times \SL_n(F) $ on $\SL_n(F)$, ignoring inner conugation action of $\SL_n(F)$ on itself). Clearly, $(\ell_a \pi)^\vee \cong \ell_a \pi^\vee$ for any irreducible representation $\pi$ of $\SL_n(F)$. 
Now fix an irreducible representation $\pi$ 
of $\SL_n(F)$ which has a Whittaker model for a character $\psi:N \rightarrow \C^\times$ where $N$ is the group of upper triangular unipotent matrices in $\SL_n(F)$. Then by Lemma 1, $\pi^\vee$ 
has a Whittaker model for $\psi^{-1}$. Thus, if $\phi$ is the outer automorphism $\phi(g) = J {}^t\!g^{-1} J^{-1}$
where $J$ is the anti-diagonal matrix with alternating 1 and $-1$, 
 so that it takes $(B,N,\psi)$ to $(B,N,\psi^{-1})$, then $\phi(\pi) = \pi^\vee$, 
by a combination of the Gelfand-Kazhdan theorem regarding the contragredient of an irreducible admissible representation  of 
$\GL_n(F)$ and the uniqueness of Whittaker model in an $L$-packet 
(i.e., given $\psi: N \rightarrow \C^\times$, there is a unique member in an $L$-packet of $\SL_n(F)$ with this Whittaker model), since $\pi^\vee$ is the unique irreducible representation of $\SL_n(F)$ in its $L$-packet with Whittaker model by $\psi^{-1}$. Having understood the action of $\phi$ on $\pi$, how about on other members $\ell_a\pi$ of the $L$-packet of $\pi$? Observe that, $\phi \circ \ell_a = \ell_{a^{-1}} \circ \phi$ (up to an element of $\SL_n(F)$). 
Therefore, $\phi(\ell_a \pi)= \ell_{a^{-1}}(\phi \pi) = \ell_{a^{-1}}\pi^\vee$. The upshot is that although
$\phi$ takes $\pi$ to its dual $\pi^\vee$, it does not take $\ell_a \pi$ to its dual, but 
to the dual of $\ell_{a^{-1}} \pi$. Our conjectures say exactly this for $\SL_n(F)$, and in general too it says something similar except that we cannot prove it for other groups since unlike for $\SL_n(F)$ where Whittaker models can be used to describe all members in an $L$-packet of 
representations,  this is not the case in general. (The change that we see for $\SL_n(F)$ from $\ell_a$ to $\ell_{a^{-1}}$ is in general the dual representation on the component group).

\vspace{3mm}

\begin{remark}
Conjecture \ref{conj1} allows for the possibility,
for groups such as  $G_2, F_4$, or $E_8$ to have 
non-selfdual representations arising out of component groups (which can be 
$\Z/3,\Z/4$, and $\Z/5$ in these respective cases). 
Indeed $G_2, F_4$, and  $E_8$ are known to have non 
self-dual representations over $p$-adic fields 
arising from compact induction of cuspidal unipotent 
representations of corresponding
finite groups of Lie type.
\end{remark}

\begin{remark}
For a quaternion division algebra $D$ over a local field $F$, 
the projective (or, adjoint) symplectic group $\PGSp_{2n}(F)$ 
has a nontrivial pure innerform, call it 
$\PGSp_n(D)$.
It is known that every irreducible  representation of $\PGSp_{2n}(F)$ is selfdual (since by [MVW], any irreducible
representation of $\Sp_{2n}(F)$ is taken to its contragredient by conjugation action of  $-1 \in F^\times/F^{\times 2} = 
\PGSp_{2n}(F)/\Sp_{2n}(F).$) This is in conformity with our conjecture as $\PGSp_{2n}(F)$ is an adjoint group, 
with no diagram automorphisms; the $L$-group in this case is $\Spin_{2n+1}(\C)$, and the possible component groups are easily seen to be extensions of  the component groups for $\SO_{2n+1}(\C)$ (which are elementary abelian 2 groups) by $\Z/2$, but since we are looking at $\PGSp_{2n}(F)$, the character of the component group is supposed to be trivial on the $\Z/2$ 
coming from the center of $\Spin_{2n+1}(\C)$, therefore such representations of the component group are actually representations of an elementary abelian 2 group, in particular selfdual.

In the work \cite{LST}, the authors prove that there is no analogue of MVW theorem for $\Sp_n(D), n \geq 3$, or for $ \SO_n(D), n \geq 5$. Looking at their proof, 
it is clear that their argument also proves that not every irreducible  representation of $\PGSp_{n}(D)$, or of $\PGSO_n(D)$ 
is selfdual. How does this compare with our conjecture \ref{conj1}? The only way out is to have more complicated component groups for 
$\Spin_{2n+1}(\C)$, in particular having non-selfdual representations for the component groups. 
The author has not seen any literature
on component groups for Spin groups. Since $\Spin_5(\C) \cong \Sp_4(\C)$, the component groups are abelian in this case, 
therefore the component group in $\Spin_{2n+1}(\C)$ 
could have non-selfdual representations only for $n \geq 3$, a condition consistent with [LST]. 
\end{remark}

\section{Comparing with the work of Adams}

Jeff Adams in \cite {adams} has constructed the contragredient of any representation of $G(\R)$ in terms of 
what he calls the Chevalley involution $C$: an automorphism of $G(\R)$ of order 2 which leaves a fundamental
Cartan subgroup $H_f$ invariant and acts by $h\rightarrow h^{-1}$ on $H_f$. (A fundamental Cartan subgroup 
of $G(\R)$ is a Cartan subgroup of $G$ defined over $\R$  which is of minimal split rank; such Cartan subgroups are conjugate under $G(\R)$.) One of the main theorems of Adams is that $\pi^C \cong \pi^\vee$ for any irreducible representation $\pi$ of $G(\R)$. He uses this to prove that 
every irreducible representation of $G(\R)$ is selfdual if and only if  $-1$ belongs to the Weyl group of $H_f$ in $G(\R)$ (in particular $-1$ belongs to the Weyl group of $G(\C)$).

Assuming $G$ to be quasi-split, with $(B,T)$ a pair of a Borel subgroup and a maximal torus in it, 
 our recipe for constructing the contragredient is in terms of an automorphism of $G(\R)$ 
of the form $t_{-} \cdot \iota$ where $t_{-}$ in an element in $T^{\ad}(\R)$ which acts by $-1$ 
on all simple root spaces, and $\iota$ is a diagram automorphism of $G$ corresponding to $-w_0$. 
We must therefore check that our automorphism constructed out of considerations with the Whittaker model
are the same as 
that of Adams constructed using semisimple elements, and this is what we do in this section.

The following well-known lemma has the effect of saying that the contragredient which appears in the 
character of component group plays no role for real groups.

\begin{lemma} Let $G$ be a connected real reductive group with a Langlands parameter $\phi: W_{\R} \rightarrow {}^LG.$ 
Then, the group $\pi_0(Z_\phi)$ is an elementary abelian 2-group.
\end{lemma}
\begin{proof} From the structure of $W_\R = \C^\times \cdot \langle j \rangle$, with $j^2=-1$, 
the group $Z_\phi$ can be considered as the fixed points of the involution which is $\phi(j)$,
 on the centralizer of 
$\phi(\C^\times)$ in $\widehat{G}$. The group $\phi(\C^\times)$ being connected, 
abelian, and consisting of semisimple elements, its centralizer in $\widehat{G}$ is connected too. It 
suffices then to prove the following lemma. \end{proof}
\begin{lemma} Let $H$ be a connected reductive group over $\C$, and $j$ an involution on $H$ with $H^j$ its fixed points. 
Then the group of 
connected components  $\pi_0(H^j)$ is an elementary abelian 2-group.
\end{lemma}
\begin{proof}According to Elie Cartan, the involution $j$ gives rise to a real structure on $H$, i.e., a connected 
real reductive group $H_\R$ on which $j$ operates, and operates as a Cartan involution, i.e., $H_{\R}(\R)^j$ 
is a maximal compact subgroup of $H_\R(\R)$, and $H^j$ is the complexification of $H_{\R}(\R)^j$. 
Thus $\pi_0(H_\R(\R))= \pi_0(H_\R(\R)^j)= \pi_0(H^j).$ Finally, it suffices to note that the 
group of connected components of a connected real reductive group is an elementary abelian 2-group, cf. \cite{BT} corollaire 14.5, who attribute the result to Matsumoto. 
\end{proof}

It now suffices to note the following two propositions for ensuring that the conditions of Adams at the beginning of the section are the same that we require for all representations of $G(\R)$ to be selfdual.

\begin{proposition}\label{prop1}
Let $G$ be a connected real reductive group which is quasi-split with $-1$ belonging to the Weyl group of $G(\C)$.
Fix a Borel subgroup $B$ of $G$ 
defined over $\R$, as well as a maximal torus $T$ inside $B$. Let $H_f(\R)$ be a fundamental
torus in $G(\R)$. Then there is an element $t_0 \in T(\R)$ which operates on all simple roots of $B$ by $-1$ 
if and only if there is an element $t_{-1}$ in $G(\R)$ which normalizes $H_f(\R)$ and acts by $h \rightarrow h^{-1}$.
The element $t_0$, equivalently $t_{-1}$, exists in $G(\R)$ if and only if 
the natural map $\iota_\star:H^1(\Gal(\C/\R),\Z/2) \rightarrow H^1(\Gal(\C/\R), Z(\C))$ is trivial where $Z$ is the center of $G$, and 
$\iota: \Z/2\rightarrow Z(\C)$ is the restriction to the center of $\SL_2(\R)$ of 
a homomorphism $\phi: \SL_2(\R) \rightarrow G(\R)$ defined over $\R$ 
associated by the Jacobson-Morozov theorem to a regular unipotent element in $G$. Further, $t_0^2$ and $t_{-1}^2$ are independent of all choices, and 
$t_0^2=t_{-1}^2=\iota(-1).$
\end{proposition}
\begin{proof}  Using Jacobson-Morozov, we will fix a homomorphism $\phi: \SL_2(\R) \rightarrow G(\R)$ defined over $\R$ 
taking the upper triangular unipotent subgroup $U$ of $\SL_2(\R)$ to the Borel subgroup $B$ of $G$, 
and capturing a regular unipotent element of $B$ in the image.
Observe that a regular unipotent element in 
$B(\R)$ has a nonzero components in each of the simple root spaces. The element $t_0\in T(\R)$ which operates on all simple 
roots of $B$ by $-1$ is unique up to $Z(\C)$. The element $t_0$ must therefore be the image of 
$j = \left ( \begin{array}{cc} i & 0 \\ 0 & -i\end{array}\right ) \in \SL_2(\C)$ 
up to central elements. Clearly $\phi(j)$ is real if and only if $\phi(-1)=1$. If $\phi(-1) =1$, then we have an element  $\phi(j) \in T(\R)$ 
which operates on all simple roots of $B$ by $-1$. If $\phi(-1)=-1$, an element of $Z(\R)$, then if there is an element in
 $T(\R)$ which can act by $-1$ on all simple roots of $B(\R)$,  the only 
option for such an element would be $z\cdot \phi(j)$ for $z\in Z(\C)$. But for $z\cdot \phi(j)$  to be real, clearly $\bar{z}=-z$, i.e.,
the natural map $H^1(\Gal(\C/\R),\Z/2) \rightarrow H^1(\Gal(\C/\R), Z(\C))$ must be trivial.  

On the other hand, fix $\Si^1 = \left ( \begin{array}{cc} a & b \\ -b & a\end{array}\right )$  
inside $\SL_2(\R)$ (with $a^2+b^2=1$). We claim that the centralizer of $\phi(\Si^1)$ inside $G(\R)$ is a fundamental torus.
For this we first prove that the centralizer of $\phi(\Si^1)$ in $G$ is a torus. To prove this, it suffices to prove a similar statement for the 
centralizer in $G$ of the image of the diagonal torus in $\SL_2(\R)$ which is a standard and a simple result. Next, we prove that the centralizer of $\phi(\Si^1)$ inside $G(\R)$ 
is a fundamental torus. For this, we can assume that $\phi(\Si^1)$ lands inside a maximal compact subgroup $K$ of $G(\R)$. Since the 
centralizer of $\phi(\Si^1)$ inside $G(\R)$ is a torus, this is also the case about the centralizer of $\phi(\Si^1)$ inside $K$.
Thus the centralizer of $\phi(\Si^1)$ inside $K$ contains a maximal torus of $K$, but any maximal torus of $G(\R)$ containing a maximal
torus of $K$ is a fundamental torus of $G(\R)$. Denote the centralizer of $\phi(\Si^1)$ inside $G(\R)$ as $H_f(\R)$. 

An element $t_{-1}$ in $G(\C)$ which acts by $-1$ on $H_f(\R)$ must preserve $\phi(\Si^1)$. 
Note that $j = \left ( \begin{array}{cc} i & 0 \\ 0 & -i\end{array}\right ) \in \SL_2(\C)$ 
preserves $\Si^1$ and acts by $-1$ on it.  
Therefore, $\phi(j)$ preserves $\phi(\Si^1)$, and hence preserves its centralizer, i.e., $H_f$ too. Thus both $t_{-1}$ as well as
$\phi(j)$ preserve $H_f(\R)$, and act as $-1$ on $\phi(\Si^1)$. Thus, 
$t^{-1}_{-1}\phi(j)$ preserves $H_f$ and acts by identity on $\phi(\Si^1)$.
Thus  $t^{-1}_{-1}\phi(j)$ belongs $H_f$, therefore  $\phi(j)$ preserves $H_f$ and acts by $-1$ on it. If $\phi(-1) =1$, we once again have
an element $\phi(j) \in G(\R)$ acting on $H_f(\R)$ by $-1$. Assume then that $\phi(-1)=-1$. Then for $\phi(j)\cdot t_0$ to be real for 
$t_0 \in H_f(\C)$, clearly
we must have $\bar{t}_0=-t_0$, i.e.,  the natural map $H^1(\Gal(\C/\R),\Z/2) \rightarrow H^1(\Gal(\C/\R), H_f(\C))$ must be trivial.

To summarize the conclusion so far, if $\phi(-1)=1$, we can take $t_0 = t_{-1}=\phi(j)$. If $\phi(-1)$ is nontrivial, 
then a  multiple
of $\phi(j)$ by $Z(\C)$ can be taken as $t_0$ if and only if the 
natural map $H^1(\Gal(\C/\R),\Z/2) \rightarrow H^1(\Gal(\C/\R), Z(\C))$ is trivial.
A  multiple
of $\phi(j)$ by $H_f(\C)$ can be taken as $t_1$ if and only if the 
natural map $H^1(\Gal(\C/\R),\Z/2) \rightarrow H^1(\Gal(\C/\R), H_f(\C))$ is trivial.

The following lemma now completes the proof of the proposition.
\end{proof}

\begin{lemma}\label{lemma2}Let $G$ be a connected reductive group over $\R$ with $Z$ the center of $G$, and $H_f$ a fundamental maximal torus in $G$. Then the
natural map $H^1(\Gal(\C/\R), Z(\C)) \rightarrow H^1(\Gal(\C/\R), H_f(\C))$ is injective.
\end{lemma}
\begin{proof}It suffices to prove that the mapping $H_f(\R) \rightarrow (H_f/Z)(\R)$ is surjective. 
But as the next well-known lemma proves, a 
fundamental Cartan subgroup in an adjoint group is connected, so naturally the mapping $H_f(\R) \rightarrow (H_f/Z)(\R)$ is surjective. 
\end{proof}

\begin{lemma}\label{lemma3}For $G$ a connected adjoint group over $\R$, a fundamental torus in $G(\R)$ is connected.
\end{lemma}
\begin{proof}Since fundamental tori are shared among innerforms, it suffices to assume that $G$ is a quasi-split group over $\R$, so that we can use 
$\phi: \SL_2(\R) \rightarrow G(\R)$ defined over $\R$ 
associated by Jacobson-Morozov theorem to a regular unipotent element in $G$. A maximally split torus $T_s$ of $G$ is obtained as 
the centralizer of the image of the diagonal torus in $\SL_2(\R)$, whereas a fundamental torus $T_f$ in $G$ is obtained as the centralizer
of the image under $\phi$ of the standard $\Si^1$ in $\SL_2(\R)$.  The diagonal and the compact torus $\Si^1$ 
in $\SL_2(\R)$ are Galois twists of each other,
which implies that their centralizers $T_s$ and $T_f$ are also Galois twists of each other in an explicit way which allows 
one to conclude the lemma. We elaborate on this. 

Let $j = \frac{1}{\sqrt{2}}\left ( \begin{array}{cc} 1 & -i \\ -i & 1\end{array}\right ) 
\in \SL_2(\C)$. Then it can be seen that
$j^{-1} \bar{j} = k = \left ( \begin{array}{cc} 0 & i \\ i & 0\end{array}\right ).$  

The element $j^{-1}\bar{j}$ defines a 1-cocycle on $\Gal(\C/\R)$ with values in $N(T) \subset \SL_2(\C)$ which can be used to twist $T$ 
to get the compact form of $T$; this is the classical notion of Cayley transform. 
More precisely, the torus $jTj^{-1}$ is defined over $\R$, and is the torus $\Si^1 \subset \SL_2(\R)$ which is the twisted form $T_w$ of $T$ defined by $T_w(\R)=\{t \in T(\C)| \bar{t} = k^{-1}tk\}$. 
Analogously, $\phi(j) T_s \phi(j)^{-1}$ is defined over $\R$, and is the torus $T_f$ in $G(\R)$. At this point we note that since 
$G$ is quasi-split adjoint group, and $T_s$ is a maximally split torus, $T_s$ is a product $(\R^\times)^a \times (\C^\times)^b$ with $a$ being the
number of fixed points of the diagram automorphism.

Since conjugation by 
 $k=\left ( \begin{array}{cc} 0 & i \\ i & 0\end{array}\right )$  takes upper triangular
unipotent matrices of $\SL_2(\R)$,     to lower triangular matrices, $\phi(k)$ takes $B$ to $B^-$, the opposite Borel, i.e., $\phi(k)=w_0$ is the longest element
in the Weyl group of $G$, and $T_f$ is the twisted torus  $T_s^{w_0}$, i.e., $T_f(\R) = \{ t \in T_s(\C) | w_0(t) = \bar{t}\}$. Twisting 
by $w_0$ is the same as twisting by $-(-w_0)$. But $-w_0$ is a diagram automorphism. The torus $T_s$ itself is a twist by a diagram automorphism, say $w^0$
of a split torus $T$. The upshot is that the torus $T_f$ is the twists of $T$ by the automorphism $-(-w_0) w^0$ with 
$(-w_0)w^0$ a diagram automorphism.  Once again, $T^{-w_0w^0}$ is a product $(\R^\times)^a \times (\C^\times)^b$ with $a$ being the
number of fixed points of the diagram automorphism $-w_0w^0$. Twisting this torus by the further automorphism $t\rightarrow t^{-1}$,
changes it to  $(\Si)^a \times (\C^\times)^b$ which is connected, proving the lemma.
\end{proof}

\begin{remark}The first proof of an early version of the previous proposition was conveyed by J. Adams using Proposition 6.24 of [AV1]; the proof 
given here is different, and carries more precise information on the existence of $t_0$ and $t_{-1}$. 
\end{remark}

\begin{lemma} \label{lamma4}Let $S$ be a compact connected Lie group, with complexification $S(\C)$. 
Let ${\rm Aut}(S)$ be the automorphism group of $S$ with ${\rm Aut}_0(S)$, 
the connected component of identity which is a compact group. 
Then the real forms of 
$S$ are in bijective correspondence with conjugacy classes of elements of order 1 or 2 in ${\rm Aut}(S)$, i.e.,
with $H^1(\Z/2, {\rm Aut}(S))$ 
where $\Z/2$ acts trivially on ${\rm Aut}(S)$. Further, two real forms
of $S$ are in the same inner class if and only if the corresponding elements in 
    $H^1(\Z/2, {\rm Aut}(S))$ have the same image in $H^1(\Z/2, {\rm Aut}(S)/{\rm Aut}_0(S))$. 
The set of element of  
    $H^1(\Z/2, {\rm Aut}(S))$ with image $\phi$  in $H^1(\Z/2, {\rm Aut}(S)/{\rm Aut}_0(S))$ is in bijective
correspondence with $H^1(\Z/2, T_o/Z)/W^\phi$ where $T_o$ is a maximal torus of $S$, invariant under $\phi$ with $T_o^\phi$ of largest possible dimension;  $Z (\subset T_o \subset S)$ is the center of $S$, and 
where $\Z/2$ operates on $T_o$ by $\phi$, and $W^\phi$ is the fixed points
of $\phi$ on the Weyl group of $T_o$. 
\end{lemma}
\begin{proof} The first part of the Lemma is of course the well-known Cartan classification for real groups, see 
[Se], III \S4, Theorem 6 
for a modern account. We leave the proof of the second assertion to the reader; see Definition 6.10 of [AV1].

Suppose $s$ is any involution in $\Aut(S)$ with image $\phi$  in $\Aut(S)/\Aut_0(S)$ which we assume 
 fixes a pinning $(B_o,T_o,\{X_\alpha\})$. 
The torus $T_o$ is then a maximal torus with $T_o^\phi$ of maximal possible dimension.

Thus, the automorphism $s\cdot \phi^{-1}$ must be conjugation by an element of $T_o/Z$, i.e., $s = t \phi$ for $t \in T_o$.     It can be seen that the element $t \in T_o/Z$ must have $t \phi(t)=1$ (in Aut$(S)$, i.e., $t\phi(t)$ should be central) for $s = t\phi$ to be an involution. Conjugating $s$ by $t_1 \in T_o/Z$ 
has the effect of changing $t$ to $t_1 t\phi(t_1)^{-1}$. Thus in expressing the involution $s = t \phi$, the conjugacy class of $s$ depends only on the element $t \in H^1(\phi, T_o/Z)$. Normalizer of $T_o$ is then used to get the final assertion in the lemma.   \end{proof}

\begin{proposition}Let $G$ be a connected real reductive group.
 Then there exists an element in $G(\R)$ 
which acts on a fundamental Cartan subgroup of $G$ by $-1$, 
if and only if  there exists an element in $G'(\R)$ which acts on 
a fundamental Cartan subgroup of $G'(\R)$ by $-1$ where $G'(\R)$ is any innerform of $G(\R)$.
\end{proposition}
 
\begin{proof}
We will use the compact real form $S$ of $G$ to parametrize all real forms of $G$. Fix 
a maximal torus $T$ in $S$. By Cartan, all the real forms of $S$ are parametrized by involutive automorphisms
of $S$ up to conjugacy, which we assume fix the maximal torus $T$. Let $S(\C)$ be the group of complex points of $S$ together with the Galois involution $g \rightarrow \bar{g}$ with fixed points $S$. 
Given an involutive automorphism $\phi$ of $S$, define the Galois involution of the corresponding real form of $G$, denoted $G_\phi$,  by $\phi_\sigma: g\rightarrow \phi(\bar{g})$.   For the real group $G_\phi(\R)$ which is the fixed points of 
$\phi_\sigma$, $S^\phi$ is a maximal compact subgroup, and $T^\phi$ a maximal torus in $S^\phi$. 

Let the centralizer of $T^\phi$ in $G_\phi(\R)$ which is nothing but $G_\phi(\R) \cap T(\C)$   
be denoted by $T_{f, \phi}$, 
a fundamental torus in $G_\phi(\R)$, and which is defined to be 
$$T_{f, \phi} = \{ t \in T(\C) | \phi(\bar{t}) = t \}.$$ 

Let $n_\phi$ be an element of
$G_\phi(\R)$ which acts by $-1$ on  $T_{f, \phi}$, 
and hence also on $T(\C)$.

Let $\psi$ be another involution of $S$,
preserving $T$, defining an innerform $G_\psi$ of $G_\phi$, hence we assume that $\psi(g) = \phi(sgs^{-1})$ for some
$s \in T$ with $s \cdot \phi(s) = 1$. Note that,
$$T_{f, \psi} = \{ t \in T(\C) | \psi(\bar{t}) = t \}= \{ t \in T(\C) | \phi(s\bar{t}s^{-1}) = t \}= 
\{ t \in T(\C) | \phi(\bar{t}) = t \} = T_{f,\phi} .$$  

It suffices to check that the element $n_\psi= s^{-1}n_\phi$  belongs to  $G_\psi(\R)$ 
since it clearly acts by $-1$ on  $T_{f, \psi} = T_{f,\phi}$.

To say that $n_\psi \in G_{\psi}(\R)$ 
means, $\psi(\bar{n}_\psi) = n_\psi. $ Since, $\psi(g) = \phi(sgs^{-1})$, 
$$\psi(\bar{n}_\psi) 
= \phi(s\bar{n}_\psi s^{-1}) = \phi(s \bar{s}^{-1} \bar{n}_\phi s^{-1}).$$ But $s = \bar{s}$ being in $T \subset 
S$.  Thus, $$\psi(\bar{n}_\psi) = \phi(\bar{n}_\phi s^{-1})= n_\phi \phi(s^{-1}) = n_\phi s = s^{-1} n_\phi = n_\psi.$$
\end{proof}

The proofs here allow to prove the theorem of Adams on the existence of Chevalley involution. Thus we give another 
proof of the following result of Adams in [Ad].

\begin{proposition}
Let $H_f(\R)$ be a fundamental maximal torus in a real reductive group $G(\R)$. Then there exists an automorphism
$C$ of $G(\R)$ which acts as $t\rightarrow t^{-1}$ on $H_f$. Further, if $G(\R)$ is quasi-split with 
a Borel subgroup $B$ of $G(\R)$, and $T$ as a maximal torus in $B$, then we can choose $H_f$ so that $C$ is the automorphism
$t_{-}\cdot \iota$ where $t_{-}$ is an element of $(T/Z)(\R)$ acting as $-1$ on all simple roots in $B$, and 
$\iota$ is a diagram automorphism of $G(\R)$ corresponding to $-w_0$. 
\end{proposition} 
\begin{proof}First we prove the result assuming $G(\R)$ is quasi-split for which we use the Jacobson-Morozov
theorem to construct $\phi: \SL_2(\R) \rightarrow G(\R)$ associated to a regular unipotent element 
inside a Borel subgroup $B$ of $G(\R)$ with $T$ as its maximal torus. 
We can assume that $\phi$ commutes with diagram automorphisms of $G(\R)$,
in particular with the diagram automorphism corresponding to $-w_0$; call this automorphism $\iota$ of $G(\R)$ 
commuting with $\phi(\SL_2(\R))$, and consider the semi-direct product $G(\R)\rtimes \langle \iota \rangle$. 
As we saw in the proof of Proposition \ref{prop1}, for $\Si^1$, the standard maximal
torus in $\SL_2(\R)$, the centralizer of $\phi(\Si^1)$ is a fundamental torus $H_f$  in $G(\R)$, normalized by $\phi(j)$.  
Since $\phi(\Si^1)$ commutes with $\iota$, $\iota$ leave $H_f$ invariant. Since $\iota$ operates on $T$ as $-w_0$,
and there is an element in $G(\R)$ which operates as $w_0$, thus there is an element in 
$G(\R)\rtimes \langle \iota \rangle$ 
which acts by $-1$ on $T$. Since all tori are conjugate in $G(\C)$, there is an element in
$G(\C)\rtimes \langle \iota \rangle$ which acts by $-1$ on $H_f$. Call this element $t_0\cdot \iota$. 
Since 
$\iota$ commutes with $\phi(\Si^1)$, and $t_0\cdot \iota$ acts by $-1$ on $H_f$, it follows that $t_0$ acts by $-1$
on $\phi(\Si^1)$; same for $\phi(j)$. Therefore, $\phi(j)t_0^{-1}$ commutes with $\phi(\Si^1)$, 
so belongs to $H_f(\C)$. Write $\phi(j) = s^{-1}t_0$ with $s \in H_f(\C)$. Then, $t_0\cdot \iota = 
s \phi(j) \cdot \iota$ 
acts by $-1$ on $H_f(\C)$, therefore since $s \in H_f(\C)$, so does    
$\phi(j) \cdot \iota$ which is an automorphism of $G(\R)$ since the inner conjugation action of 
$\phi(j)$ is real, proving existence of the Chevalley involution for quasi-split groups.

For general $G(\R)$, we appeal to the result of Borovoi that $H^1(\Gal(\C/\R), H_f) \rightarrow 
H^1(\Gal(\C/\R), G)$ is surjective, and as a result, any innerform of $G(\R)$ is obtained by twisting by a cocycle
with values in $H_f(\C)$. It can be easily seen that if $G_c$ is obtained by twisting $G(\R)$ by an element
$c \in H_f(\C)$ with $ c\bar{c}=1$, then if $C$ is a Chevalley involution on $G(\R)$, 
$cCc^{-1}$ is defined over $\R$, and is 
a Chevalley involution on $G_c(\R)$ acting as $-1$ on $H_f$ which continues to be a fundamental torus.
\end{proof}

\section{Finite fields}

Before we start with the study of distinction property in the case of local fields, it may be good to recall what can be proved for finite fields, something which I did in \cite {Pr01}. For finite fields, the basic theorem is proved in the generality of connected groups without reductivity hypothesis. Since the result 
was not stated in its ideal form, I do it here. 

Recall that for a connected algebraic group $G$ over a finite field $\F_q$, by Lang's theorem, every element $x$ of $G(\F_q)$ can be written as $x = y^{-1}y^{[q]}$ 
for $y \in G(\bar{\F}_q)$ where $y\rightarrow y^{[q]}$ is the Frobenius map on
$G(\bar{\F}_q)$; the choice of $y$ in expressing $x = y^{-1}y^{[q]}$ is unique up to left translation by $G(\F_q)$. 

Define the Shintani transform ${\rm Sh}: G(\F_q) \rightarrow G(\F_q)$, by
$$\begin{array}{cccc}
{\rm Sh}: & G(\F_q) &\longrightarrow & G(\F_q)\\
 & y^{-1}y^{[q]} & \longrightarrow &y^{[q]}y^{-1}.
\end{array}$$

The Shintani transform $x \rightarrow {\rm Sh}(x)$ defines a well-defined map on 
the set of $G(\F_q)$ conjugacy classes in $G(\F_q)$ to the set of $G(\F_q)$ conjugacy classes in $G(\F_q)$. Since any semisimple element of $G(\F_q)$ belongs to a torus, the Shintani transform is the identity map on semisimple elements. 

\begin{thm}
Let $G$ be a connected algebraic group over a finite field $\F_q$, and $\pi$ an irreducible $\C$-representation of $G(\F_{q^2})$. Assume
that the character of $\pi$ takes the same value at $x$ as at ${\rm Sh}(x)$ for $x\in G(\F_{q^2})$. Then the representation $\pi$ has a 
$G(\F_q)$ fixed vector if and only if $\pi^\sigma \cong \pi^\vee$. If $\pi^\sigma \cong \pi^\vee$, then $\pi$ has a one dimensional space of fixed vectors
under $G(\F_q)$, and the representation $\pi \otimes \pi^\sigma$ which is canonically a representation of $G(\F_{q^2}) \rtimes \Z/2$ 
has a $G(\F_{q}) \rtimes \Z/2$ fixed vector.
\end{thm}

\begin{corollary}
Let $G$ be a connected reductive group over a finite field $\F_q$, and $\pi$ an irreducible Deligne-Lusztig 
representation of $G(\F_{q^2})$ induced from a torus.  Then the representation $\pi$ has a 
$G(\F_q)$ fixed vector if and only if $\pi^\sigma \cong \pi^\vee$. If $\pi^\sigma \cong \pi^\vee$, then $\pi$ has a one dimensional space of fixed vectors
under $G(\F_q)$, and the representation $\pi \otimes \pi^\sigma$ which is canonically a representation of $G(\F_{q^2}) \rtimes \Z/2$ (where $\Z/2$ acts by permuting the co-ordinates: $v_1 \otimes v_2 \rightarrow v_2 \otimes v_1$)
has a $G(\F_{q}) \rtimes \Z/2$ fixed vector.
\end{corollary}

\begin{proof}The proof of the corollary follows from a theorem of Digne-Michel that 
the character of a Deligne-Lusztig representation of $G(\F_{q^2})$ takes the same value at $x$ as at ${\rm Sh}(x)$ for $x\in G(\F_{q^2})$. 
\end{proof}

\section{Construction of the group $G^{\out}$}
Let $E/F$ be quadratic extension of local fields, and $G$ a quasi-split reductive algebraic 
group over $F$. 
In this section we construct another quasi-split group $G^{\out}$ 
such that the groups $G$ and $G^{\out}$ become isomorphic over $E$. This group $G^{\out}$ plays 
an important role in the study of representations of $G(E)$  distinguished by $G(F)$ (as these
representations arise as base change of representations of $G^{\out}(F)$).

Recall that given a reductive algebraic group $G$ over the algebraic closure $\bar{F}$ of $F$, 
quasi-split forms of $G$  over $F$ are described by homomorphisms of Gal$(\bar{F}/F)$ 
to  Out$(G(\bar{F})) 
= {\rm Aut}(G(\bar{F})) /  
{\rm Inn}(G(\bar{F}))$   where ${\rm Inn}(G(\bar{F}))$   denotes the group of inner automorphisms.

As recalled earlier, every reductive group $G$ over $\bar{F}$ has an automorphism $\iota$ of order $\leq 2$ 
which takes every irreducible algebraic representation
of $G$ to its contragredient. The automorphism $\iota$ defines a well-defined element of 
 Out$(G(\bar{F}))$, and belongs to the center of this group. This automorphism $\iota$ is called the Chevalley involution, or also the
{\it opposition involution}.

Let $G^{\out}$ be the quasi-split group over $F$ obtained by twisting 
$G$ by the Chevalley involution, i.e., the 
homomorphisms from Gal$(\bar{F}/F)$ 
 to Out$(G(\bar{F})) $ associated to  $G^{\out}$ is the homomorphisms from Gal$(\bar{F}/F)$  
to Out$(G(\bar{F})) $
associated to $G$ times the 
homomorphisms of Gal$(\bar{F}/F)$ to Out$(G(\bar{F})) $ which sends the nontrivial element
in Gal($E/F$) to the 
Chevalley involution $\iota$. 

\vspace{5mm}

\noindent{\bf Example :} (a) For $G = \GL_n$, $G^{\out} = \U_n$, and for $G = \U_n$, $G^{\out} = \GL_n$.

(b)  For a torus $T$ over $F$, it can be seen that  $T^{\out} = (R_{E/F}T)/T$ where $R_{E/F}$ is the Weil restriction of scalars.

\begin{remark} It can be seen that for $\C/\R$, if the group $G$ is an innerform of a split group, 
then $G^{\out}$ is the innerform of the compact group, and vice versa;  in particular, if the split 
and compact form are innerforms of each other (i.e., if $-1$ belongs to the Weyl group of $G$), 
$G^{\out}$ is the quasi-split innerform of $G$.
\end{remark}

\begin{remark} The (quasi-split) group $G^{\out}$ associated to a quadratic extension $E/F$, and a 
reductive group $G$ over $F$, has been around for a long time. See for example, the work \cite{Na} of Nadya Gurevich. 
This group has also been noticed in the case of real groups, see \cite{H-P}. 
\end{remark}

\section{A character of $G(F)$ of order 2} 
We now return to the context of $G(F) \subset G(E)$. 
In an earlier paper of the author \cite{Pr02}, 
there is the
construction of a character $\omega_G: G(F) \rightarrow \Z/2$ 
associated to any quadratic extension
$E$ of $F$ which plays an important role in questions about distinction. 
This character is in a certain sense `dual' to the canonical element in the center of a  
reductive group of order one or two which determines if a finite dimensional selfdual algebraic representation of 
the group is orthogonal or symplectic. It is dual also in the sense of `arrows', this one being  
$\omega_G: G(F) \rightarrow \Z/2$, the other one being, $\Z/2 \rightarrow \widehat{G}$. 

The character $\omega_G$ 
is functorial under maps of reductive algebraic groups with finite kernel and cokernel, and $\omega_{G_1 \times G_2} = \omega_{G_1} + \omega_{G_2}$. 
We will review the construction
of $\omega_G$ here.  Let $G^{\rm ad}$ denote  the adjoint group of $G$, i.e., $G$ modulo center, 
and $G^{\rm sc}$ the simply connected
cover of $G^{\rm ad}$. Let $Z_{\sc}$ be the center of $G^{\rm sc}$. Then we have an exact sequence of groups,
$$1 \rightarrow Z_{\sc}(F) \rightarrow G^{\rm sc}(F) \rightarrow G^{\rm ad}(F) \rightarrow 
H^1( {\rm Gal}(\bar{F}/F), Z_{\scon})
\rightarrow \cdots$$
The character $\omega_G$ factors through a character of $G^{\rm ad}(F)$ via the natural map,
 $G(F) \rightarrow G^{\rm ad}(F)$, so we need to construct one for $G^{\rm ad}(F)$, which arises from 
this exact sequence from a character of $H^1( {\rm Gal}(\bar{F}/F), Z_{\scon})$, which by
Tate-Nakayama duality amounts to constructing an element of $H^1({\rm Gal}(\bar{F}/F), {Z}^\vee_{\scon})$,
where ${Z}^\vee_{\scon}$ is the Cartier dual of $Z_{\scon}$. 

Let $\widehat{G}'$ be the dual group of $G'= G^{\rm ad}$. It is clear that one can choose
a regular  unipotent in $\widehat{G}'$ such that the corresponding Jacobson-Morozov homomorphism
of $\SL_2(\C)$ into $\widehat{G}'$ is invariant under ({\it pinned}) outer automorphisms of $\widehat{G}'$. The center
of $\SL_2(\C)$ under this Jacobson-Morozov homomorphism goes  to the center of $\widehat{G}'$ which is nothing but 
${Z}^\vee_{\scon}$, inducing a map $H^1({\rm Gal}(\bar{F}/F), \Z/2) \rightarrow 
H^1({\rm Gal}(\bar{F}/F), {Z}^\vee_{\scon})$. 
Since $H^1({\rm Gal}(\bar{F}/F), \Z/2)$ parametrizes
quadratic etale algebras over $F$, thus associated to the quadratic extension $E/F$, we get 
an element of $H^1({\rm Gal}(\bar{F}/F), {Z}^\vee_{\scon})$, 
therefore a character of $ H^1({\rm Gal}(\bar{F}/F), Z_{\scon})$, 
and finally a character
$\omega_G: G(F) \rightarrow \Z/2$ associated to any quadratic extension $E$ of $F$. We remind the reader that
although $\omega_G$ very much depends on the quadratic extension $E/F$, to lighten the notation, we have
not explicitly used it in $\omega_G$.

\vspace{4mm}

\noindent{\bf Example :} (a) For $G = \GL_n$, $\omega_G = \omega_{E/F} \circ \det$ for $n$ even and trivial for $n$ odd.

(b) For $G = \U_n$, defined using a hermitian form over $E$, $\omega_G $ is trivial for all $n$.  

\begin{remark}Since the centre of a reductive group is the same for all innerforms, the character 
$\omega_G:G(F) \rightarrow \Z/2$ is `independent' of the inner class of $G$.
\end{remark}
\vspace{4mm}

The following proposition points to  the close relationship between the character $\omega_G$ and 
{\it half the sum of positive roots} being not a weight, a well-known source of difficulty in character theory 
of both real and $p$-adic groups. We omit the proof of the proposition.

\begin{proposition}\label{prop2}
Let $G$ be a quasi-split adjoint group over a local field $F$ together with a maximal torus 
$T$ contained in a Borel subgroup $B$ of $G$ over $F$. Let $G^{\scon}$ be the simply connected cover of $G$ with $Z^{\scon}$ 
the kernel of the map $G^{\scon} \rightarrow G$, and $T^{\scon}$ the maximal torus in $G^{\scon}$ lying over $T$. 
Let $\delta_G$  be {\it half the sum of positive roots} of $T^{\scon}$ 
inside $B$. It is known that $\delta_G \in X^\star(T^{\scon}) \cap \frac{1}{2}X^\star(T)$. 
Therefore $ \delta_G|_{Z^{\scon}}$ gives a homomorphism $\delta_G: Z^{\scon} \rightarrow \Z/2,$ 
i.e., an element of 
order $\leq 2$ of $\Hom[Z^{\scon}, \Q/\Z]$. A quadratic extension $E$ of $F$ then 
gives a homomorphism from ${\rm Gal}(\bar{F}/F)$ to 
$ \Hom[Z^{\scon}, \Q/\Z]$, thus we get an element of    
$H^1({\rm Gal}(\bar{F}/F), \Hom[Z^{sc},\Q/\Z])$.  Tate duality
then gives  a character of $ H^1({\rm Gal}(\bar{F}/F), Z_{sc})$, hence 
a character of $G(F)$. The character of $G(F)$ so defined is nothing but the character $\omega_G: G(F) \rightarrow \Z/2$ defined above associated to any quadratic extension $E$ of $F$.
\end{proposition}

We now come to an another interpretation  of the character $\omega_G$.
For this, we recall some remarks of Langlands from Example 10.5 of [Bo], assuming that $G$ is a semisimple adjoint 
group with discrete series. Let $\widehat{T}(\C)$ be a `standard' maximal torus in $\widehat{G}(\C)$ which is used to define $\widehat{G}(\C)$.
A continuous homomorphism $\varphi: \C^\times \rightarrow \widehat{T}(\C)$ 
is of the form $z\rightarrow z^\mu \cdot \bar{z}^\nu$, for a pair $(\mu,\nu) $ in 
$X_\star(\widehat{T}) 
\otimes \C$ with $\mu -\nu \in   X_\star(\widehat{T})$. The homomorphism 
$\varphi = z^{\mu}\cdot \bar{z}^\nu: \C^\times \rightarrow \widehat{T}(\C)$ does not lie in a proper parabolic of $\widehat{G}(\C)$ if and only if 
$\mu =-\nu$, and it extends to
a parameter $\varphi': W_\R\rightarrow {}^LG$ if and only if 
$$\mu \in \delta_G + X_\star(\widehat{T}),$$
where $\delta_G$ is half the sum of positive roots of $G$, thus an element of 
$X^\star(T)\otimes \Q = X_\star(\wT) \otimes \Q$.  Let $\wG^{\ad}$ be the adjoint group associated to  $\wG$ with $\wT^{\ad}$
the image of $\wT$  making     a commutative diagram: 
$$\xymatrixrowsep{1in}
\xymatrixcolsep{1in}
\xymatrix{\widehat{T}\ar@{^{(}->} 
[r] \ar@{->}[d] & \ar[d]\widehat{G} \\
 \widehat{T}^{\ad}\ar@{^{(}->} [r]  & \widehat{G}^{\ad}.
}
$$
Then $X_\star(\wT) \subset X_\star(\wT^{\ad})$, and the coset space  $X_\star(\wT^{\ad})/X_\star(\wT) = X^\star(T^{\scon})
/X^\star(T)$ can be identified to 
the character group  $\Hom[Z^{\scon}, \Q/\Z]$ by restricting characters of $T^{\scon}$ to $Z^{\scon}$.
Now the condition on $\mu$, i.e., $\mu \in \delta_G + X_\star(\widehat{T}),$ can be reinterpreted to say that
$\mu = \delta_G: Z^{\scon} \rightarrow \Z/2.$

\vspace{4mm}
We conclude our discussion in this section with the following proposition which will be very useful when we 
come to real groups. 

\begin{proposition} \label{prop3}Suppose $G^{\out}$ is a quasi-split group over $\R$ with compact maximal torus $T_c(\R)$, 
with $G(\R)$ the corresponding split group with maximal split torus $T(\R)$  which sit together 
in the exact sequence:
$$1 \rightarrow T(\R) \rightarrow T(\C) \rightarrow T_c(\R)\rightarrow 1.$$
Suppose  $\mu$ is a character of the finite cover $T^s_c(\R)\rightarrow T_c(\R)$ with kernel 
$Z^{\scon}$  with $\mu|_{Z^{\scon}} = \delta_G$. Then the restriction of $\mu$ to $T(\C)$ using the triangle,
$$
\xymatrix{
  & \ar@{->}[d]T^s_c(\R) \\
T(\C) \ar@{-->}[ur]\ar[r] & T_c(\R),  
}
$$
when restricted to $T(\R)$ (a split torus) gives rise to the restriction of the character
$\omega_G$  of $G(\R)$ to its maximal split torus $T$.
\end{proposition}
\begin{remark}
The quadratic character $\omega_G: G(F) \rightarrow \Z/2$ 
was introduced in the paper \cite {Pr02}, where I   formulated the conjecture that for $E$ a quadratic extension of $F$, 
a $p$-adic field, the Steinberg
 representation of $G(E)$ is $\omega_G$-distinguished, and not $\omega$-distinguished for any other character 
$\omega \not = \omega_G$ of $G(F)$.
The Steinberg representation was  singled out partly because its character is stable, and therefore the theorem in 
my paper \cite {Pr01} for finite fields (recalled as theorem 1 in this paper), was most 
likely to hold for such representations, except for this twist to the story of having to use the character $\omega_G$.
 
The conjecture on the distinction property of the Steinberg are now proven  
in several cases by Broussous and Courtes, see for instance \cite{brf}, \cite{cou}, 
and other forthcoming papers of these authors.

\end{remark}
\section{Cyclic basechange}
Let $ E/F$ be a cyclic extension of local fields of degree $d$, giving rise to an exact sequence:
$$1 \rightarrow W_E \rightarrow W_F \rightarrow {\rm Gal}(E/F) 
\rightarrow 1.$$

Given a Langlands parameter $\phi: W_E\longrightarrow {}^LG$, 
its extensions to $W_F$ in the following diagram
$$
\xymatrix{
W_E \ar[r]  \ar@{_(->}[d]  & {}^LG \\
W_F \ar@{-->}[ur] & {}  
}
$$
is what is of critical interest to us in this paper for $E/F$ quadratic.

Recall that a Langlands parameter $\phi: W_E\longrightarrow {}^LG$, 
can equivalently be considered as an
element of $H^1(E,\widehat{G})$ 
where we give $\widehat{G}$ its natural structure as a $W_E$-group 
(through which ${}^LG = \wG \rtimes W_E$ is constructed).  Further, a given Langlands parameter 
$\phi_0: W_E\longrightarrow {}^LG$, gives $\widehat{G}$ another structure as a $W_E$-group (through the action of $w \in W_E$ 
on $g \in \widehat{G}$ by 
$\phi_0(w)g^w\phi_0(w)^{-1}$ where $g^w$ is the original action of $w$ on $g \in \widehat{G}$).  For this
action of $W_E$ on $\widehat{G}$, denote
$H^1(E,\widehat{G})$ 
by $H^1_{\phi_0}(E,\widehat{G})$, 
to differentiate it from the original $H^1(E,\widehat{G})$ where  
$\widehat{G}$ is given its natural structure as a $W_E$-group.

Then there is a natural identification between  $H^1_{\phi_0}(E,\widehat{G})$ and $H^1(E,\widehat{G})$ 
given at the level of cocycles by $$ \phi \in H^1_{\phi_0}(E,\widehat{G}) \longrightarrow \Phi = \phi \cdot \phi_0 
\in H^1(E,\widehat{G}),$$ as can be easily checked.
 
Thus $H^1_{\phi_0}(E,\widehat{G})$  still represents equivalence classes 
of Langlands parameters, but with the advantage that the   Langlands parameter $\phi_0: W_E\longrightarrow {}^LG$
now corresponds to the identity element of $H^1_{\phi_0}(E,\widehat{G})$. Observe also that for this
action of $W_E$ on $\widehat{G}$,   $\widehat{G}^{W_E}$ is nothing but centralizer $Z(\phi_0)$ of the parameter 
$\phi_0: W_E\longrightarrow {}^LG$.

Thus the following standard Lemma, cf. Serre \cite{Se} I \S5.8(a),  answers exactly the question we are 
considering in this paper, viz. possible lifts of a Langlands parameter from $W_E$ to $W_F$.

\begin{lemma}\label{fibers}The following is an exact sequence of pointed sets (with $j$ an injective map of sets):
$$1 \rightarrow H^1({\rm Gal}(E/F), \widehat{G}^{W_E}
) \stackrel{j}\rightarrow H^1_{\phi_0}(F, \widehat{G}) 
\rightarrow H^1_{\phi_0}(E,\widehat{G}) ^{{\rm Gal}(E/F) }.$$
\end{lemma}

\begin{corollary}
If $\phi$ is a `stable discrete' parameter for $G(E)$, 
i.e., $Z_{\wG}(\phi) = Z(\wG)^{W_E},$ then it follows that the 
possible lifts of $\phi$ to a parameter of $G(F)$ (if non-empty) is parametrized by 
  $H^1({\rm Gal}(E/F), Z(\wG)^{W_E})$, 
and if $\phi_0$ is one such lift as an element of 
$H^1(W_F, \widehat{G})$, 
then all the other lifts are 
$\varphi \cdot \phi_0$ for $\varphi \in H^1({\rm Gal}(E/F), Z(\wG)^{W_E}) \subset  H^1(W_F, Z(\wG))$, 
for the natural pairing: $H^1(W_F, Z(\wG)) \times H^1(W_F, \widehat{G}) \rightarrow H^1(W_F, \widehat{G})$, 
and these are distinct. 
Our conjecture \ref{conj3} will then say that a representation of $G(E)$ with a stable discrete parameter, 
if distinguished by $G(F)$, is distinguished
(with multiplicity 1) by {\rm all} pure innerforms of $G$ over $F$ which trivialize over $E$.   
\end{corollary}

\noindent{\bf Example 1:}
If $G$ is $\GL_n$ over $F$, and $\sigma$ is a parameter of a representation of $\GL_n(F)$, write
$$\sigma|_{W_E}= \sum_i n_i \sigma_i,$$
where $\sigma_i$ are irreducible representations of $W_E$.

Let $Z_{\M(n,\C)}(\sigma|_{W_E})$ be the centralizer of $\sigma|_{W_E}$ in the ambient algebra $\M(n,\C)$.
Clearly $Z_{\M(n,\C)}(\sigma|_{W_E}) = \sum_i \M(n_i,\C).$ 
The action of $\Gal(E/F) = \Z/2$ 
preserves the algebra  $Z_{\M(n,\C)}(\sigma|_{W_E}) = \sum_i \M(n_i,\C),$ and acts as an automorphism on it. Thus it either preserves 
a particular $\M(n_i,\C)$ corresponding to those representations $\sigma_i$ of $W_E$ which extend to $W_F$,
or permutes two $\M(n_i,\C)$ and $\M(n_j,\C)$   corresponding to those {\it distinct} 
representations $\sigma_i, \sigma_j$ of 
$W_E$ which are Galois conjugate. 

If we now let $Z(\sigma|_{W_E})$ be the centralizer of $\sigma|_{W_E}$ in the ambient $\GL(n,\C)$, we get the same conclusion,
i.e., $\Gal(E/F) = \Z/2$ either preserves a component, or permutes two of them. (We took this extra care
to first calculate the centralizer in $\M(n,\C)$ because it is the automorphism group of a semisimple algebra which has
the obvious structure, and not that of the group.)

In the calculation of $H^1({\rm Gal}(E/F), Z(\sigma|_{W_E})) 
= H^1(\Z/2, \prod_i \GL(n_i,\C))$, the factors
which are permuted by $\Gal(E/F)=\Z/2$ contribute nothing to the cohomology by Shapiro's lemma. The following lemma
summarizes the calculation of relevant cohomologies when the involution fixes a factor.

\begin{lemma} \label{inv}The involutive automorphisms of $\GL(n,\C)$ (up to conjugacy by $\GL(n,\C)$) are 
\begin{enumerate} 
\item inner-conjugation action by $I_{(p,q)}$ with $p+q=n$,
where $I_{(p,q)}$ represents the diagonal matrix with $p$ many 1, and $q$ many $-1$.  
\item $g\longrightarrow A {}^tg^{-1} A^{-1}$ where $A$ is a symmetric nonsingular matrix.
\item $g\longrightarrow A {}^tg^{-1} A^{-1}$ where $A$ is a skew-symmetric nonsingular matrix.
\end{enumerate}
In these respective cases,
\begin{enumerate} 
\item $H^1(\Z/2, \GL_n(\C))$ is a set with $(n+1)$ elements.
\item $H^1(\Z/2, \GL_n(\C)) = \langle 1 \rangle.$ 
\item $H^1(\Z/2, \GL_n(\C)) = \langle 1 \rangle.$ 
\end{enumerate}
\end{lemma}

In  Example $(1)$, only the first case of Lemma \ref{inv} arises as we will presently show, 
giving us a set of $(n_i+1)$ elements corresponding
to each $W_E$ isotypic component $n_i\sigma_i$ corresponding to those irreducible 
representations $\sigma_i$ of $W_E$ which extend to $W_F$.

To understand the action of $\Gal(E/F) = \Z/2$ on   $Z(\sigma|_{W_E}) = \prod_i \GL(n_i,\C),$ we try to understand
the effect of this action on the center of the centralizer, i.e., on   $Z(Z(\sigma|_{W_E})) = 
Z(\prod_i \GL(n_i,\C)) = \prod_i\C^\times.$ 
For this, note that the center of the ambient $\GL_n(\C)$, i.e.,
$\C^\times$,  has a natural embedding into $Z(Z(\sigma|_{W_E})) = Z(\prod_i \GL(n_i,\C)) = \prod_i\C^\times,$ 
which when projected to individual $\C^\times$ in  $\prod_i\C^\times,$ is just the identity map 
of $\C^\times$ to $\C^\times$. 

Since the action of $\Gal(E/F)$ on the center of the ambient $\GL_n(\C)$ is trivial, it follows that the action
of $\Gal(E/F)$ on 
$Z(Z(\sigma|_{W_E})) = Z(\prod_i \GL(n_i,\C)) = \prod_i\C^\times$ is also trivial. Looking at the possible 
involutive automorphisms in  lemma \ref{inv}, we thus see that the only possible 
involutions on a $\GL(n_i,\C)$ are conjugation by an element of order $\leq 2$.

\noindent{\bf Example 2:}
If $G$ is $\U_n$ over $F$, and $\sigma$ is a parameter of a representation of $\U_n(F)$, write
$$\sigma|_{W_E}= \sum_i n_i \sigma_i,$$
where $\sigma_i$ are irreducible representations of $W_E$.

Clearly $Z(\sigma|_{W_E}) = \prod_i \GL(n_i,\C).$ 
The action of $\Gal(E/F) = \Z/2$ on   $Z(\sigma|_{W_E}) = \prod_i \GL(n_i,\C),$ 
either preserves a particular $\GL(n_i,\C)$ corresponding to those representations $\sigma_i$ of $W_E$ for which 
$\sigma_i^\vee \cong \sigma_i^\sigma$,
or permutes two $\GL(n_i,\C)$ and $\GL(n_j,\C)$   corresponding to those {\it distinct} 
representations $\sigma_i, \sigma_j$ with
$\sigma_j^\vee \cong \sigma_i^\sigma$.

In the calculation of $H^1({\rm Gal}(E/F), Z(\sigma|_{W_E})) = H^1(\Z/2, \prod_i \GL(n_i,\C))$, the factors
which are permuted by $\Gal(E/F)=\Z/2$ contribute nothing to the cohomology by Shapiro's lemma as in Example $(1)$.
But in this case, as we will presently show, the $\GL(n_i,\C)$ which are left invariant by $\Gal(E/F)$ carry involutions 
which are as in case 2 or 3 of Lemma \ref{inv}, so we have,
$$H^1({\rm Gal}(E/F), Z(\sigma|_{W_E})) = \langle 1 \rangle.
$$
 
To understand the action of $\Gal(E/F) = \Z/2$ on   $Z(\sigma|_{W_E}) = \prod_i \GL(n_i,\C),$ we cannot 
proceed exactly as
in the previous example because the action of $\Gal(E/F)$ involves transpose-inverse. 
Note that the action of $\Gal(E/F)$ on the center of the ambient $\GL_n(\C)$ is this time 
$z\rightarrow z^{-1}$. 
It follows that the action
of $\Gal(E/F)$ on 
$Z(Z(\sigma|_{W_E})) = Z(\prod_i \GL(n_i,\C)) = \prod_i\C^\times$ is $(z_i)\rightarrow (z_i^{-1})$. 
Now, we argue --- without giving details --- that an involution of $\G= \prod_i \GL(n_i,\C),$ which is $x\rightarrow x^{-1}$ on the center 
of $\G$ must be a product of involutions on individual $\GL(n_i, \C)$, thus 
 involving only  cases 2 and 3 of lemma \ref{inv}.

The following well-known, but perhaps never quite proved, proposition follows as a corollary to Lemma \ref{inv}.

\begin{proposition} \label{prop5}
A Langlands parameter $\phi: W_F\longrightarrow \GL_n(\C) \rtimes \Gal(E/F)$, associated to $G=\U_n(F)$ is determined by its restriction 
to $W_E$, i.e., by $\phi|_{W_E} \longrightarrow \GL(n,\C)$.
\end{proposition}
\begin{remark}
Since the centralizer of the parameter $\phi: W_F \longrightarrow {}^LG$ in $\widehat{G}$, is the
 $\Gal(E/F)$ fixed points of $\widehat{G}^{W_E}$, it follows that in the example 2 of unitary groups, the non-trivial component
group is contributed only by  case 2 of Lemma \ref{inv}, which contributes a $\Z/2$ because of $\OO(n,\C)$, thus recovering 
a well-known result about component groups of Langlands parameters for unitary groups.
\end{remark}

\noindent{\bf Example 3:} We work out the possibilities for $\widehat{G}^{W_E}$ 
and $H^1({\rm Gal}(E/F), \widehat{G}^{W_E})$ 
when $G = \SL_2$, $\widehat{G}= \PGL_2(\C)$, assuming that the representation  $\pi$  of 
$\SL_2(E)$ arises as basechange of a representation from $\SL_2(F)$, so $\widehat{G}^{W_E}$ 
comes equipped with an action of ${\rm Gal}(E/F)=\Z/2$.

For a principal series representation $\pi$ of $\SL_2(E)$, 
$\widehat{G}^{W_E}$ 
is either $\C^\times$, or is $\OO_2(\C)$. 
On $\C^\times$, the action can only be trivial or  
$z\rightarrow z^{-1}$, and  $H^1({\rm Gal}(E/F), \C^\times)$ is $\Z/2$ or trivial in these 
two cases.

If $\widehat{G}^{W_E}$ is $\OO_2(\C)$, then the action of $\Z/2$  on $\OO_2(\C)$ can either be 
\begin{enumerate}
\item trivial, or

\item inner conjugation of an element of order 4 in $\SO_2(\C) = \C^\times$, or 

\item inner conjugation of a reflection, i.e., any element of $\OO_2(\C) \backslash \SO_2(\C)$.
\end{enumerate}

When the action of $\Z/2$ on $\OO_2(\C)$ is trivial, 
$\widehat{G}^{W_F} = \OO_2(\C)$, 
hence 
$\pi_0(\widehat{G}^{W_F}) = \Z/2$,
 and  $H^1(\Z/2,\OO_2(\C))$ 
is the set of conjugacy classes of involutions in $\OO_2(\C)$, 
i.e., a set consisting of 3 elements: the trivial element, the element of order 2 in $\SO_2(\C)$, and any element of $\OO_2(\C) \backslash \SO_2(\C)$ which are all of order 2, and define a single conjugacy class.

When the action of $\Z/2$ 
on $\OO_2(\C)$ 
is the inner conjugation of an element of order 4 in $\SO_2(\C) = \C^\times$, 
 then $\widehat{G}^{W_F} = \C^\times$, 
hence  $\pi_0(\widehat{G}^{W_F}) = \langle 1 \rangle$, and 
$H^1(\Z/2,\OO_2(\C)) = \langle 1 \rangle$ as can be easily seen.  

When the action of $\Z/2$ 
on $\OO_2(\C)$ is the inner conjugation of a reflection, 
$\widehat{G}^{W_F} = \Z/2 + \Z/2$, hence 
$\pi_0(\widehat{G}^{W_F}) = \Z/2 + \Z/2$. 
We now calculate 
 $H^1(\Z/2,\OO_2(\C))$.  
Note that $\R^+ \cdot \Si^1 = \C^\times= \SO_2(\C)$. Since
$\R^+$ is a uniquely divisible group,  $H^1(\Z/2,\OO_2(\C)) = H^1(\Z/2,\Si^1\rtimes\Z/2)=
H^1({\rm Gal}(\C/\R),\OO_2(\C))$, 
the last cohomology being Galois cohomology of the real 
group $\OO_2$ defined by the quadratic form $x^2 + y^2 = 1$ where Galois group operates by
$z\rightarrow \bar{z}$ on $\C^\times \subset \OO_2(\C)$ which is the same action as 
$z\rightarrow z^{-1}$ on $\Si^1$. Now $H^1({\rm Gal}(\C/\R),\OO_2(\C))$ is the set of isomorphism classes of non-degenerate quadratic forms in 2 variables over $\R$, so it is a set 
with 3 elements, corresponding to quadratic forms of signatures $(2,0), (1,1), (0,2)$.

For a discrete series representation $\pi$ of $\SL_2(E)$, 
$\widehat{G}^{W_E}$ is either trivial, or $\Z/2, \Z/2\oplus \Z/2$. 
On $\Z/2$, the action of $\Z/2$ can only be trivial, whereas 
on $\Z/2\oplus \Z/2$ it can be either trivial, or through a unipotent in $\GL_2(\Z/2)$.  
Clearly, $H^1({\rm Gal}(E/F), \Z/2) = \Z/2$, 
whereas for the action of $\Z/2$ through a unipotent
element, it can be seen that $H^1({\rm Gal}(E/F), \Z/2 + \Z/2) = 0$. 

\begin{remark}It may be noted that the action of ${\rm Gal}(E/F)=\Z/2$ on  
$\widehat{G}^{W_E}$ depends on a particular extension of a parameter of $W_E$ to $W_F$,
but the cohomology $H^1({\rm Gal}(E/F), \widehat{G}^{W_E})$ 
does not. We see this happening 
for principal series representations of $\SL_2(E)$ with non-trivial component group (so 
$\widehat{G}^{W_E} = \OO_2(\C)$) 
which can arise from
basechange of both principal series (with non-trivial component group so 
$\widehat{G}^{W_F} = \OO_2(\C)$
corresponding to 
the trivial action of $\Z/2$ on $\widehat{G}^{W_E} = \OO_2(\C)$) 
and discrete series corresponding to the inner conjugation of a reflection in $\OO_2(\C)$ 
(so with $\widehat{G}^{W_F} = \Z/2+\Z/2$), and for both possibilities,
 $H^1({\rm Gal}(E/F), \widehat{G}^{W_E})$ 
is a set with 3 elements.
\end{remark}

\begin{remark} Later we will explicitly discuss the fibers of basechange for 
principal series representations of $\SL_2(E)$, and curiously, we will find that {\it all} 
cases which have been discussed from an abstract point of view in example 3 above do arise. (Of course, 
the only thing not obvious from the discussion in example 3 is that all actions of $\Z/2$ on
$\widehat{G}^{W_E}$ arise from representation theory of $\SL_2(E)$.)  
\end{remark}
\section{Parameter spaces for Langlands parameters} \label{parameter}
In this section we discuss parameter spaces associated to Langlands parameters. Till the very last remark of this section,
we will deal exclusively with non-Archimedean local fields. As usual, let $W'_F= W_F \times \SL_2(\C)$ be the Weil-Deligne group of a non-Archimedean local field $F$, where $W_F$ is the Weil group which is a locally compact group, with unique
maximal compact subgroup $I_F$, the inertia group, which is normal in $W_F$ with $W_F/I_F \cong \Z$. One chooses an element $\Fr \in W_F$ to write $W_F$ as $I_F \rtimes \langle \Fr \rangle $ with image of $\Fr $ in $W_F/I_F$  a generator of the cyclic group $\Z$. 

Let $\widetilde{\Sigma}_F(G)$ 
denote the set of homomorphisms $\phi: W'_F\rightarrow \LG$ which has all the requirements for an admissible
homomorphism except that we will not demand that a Frobenius element of $W_F$ goes to a semisimple element of $\LG$. 
(This has been done because the set of semisimple elements of a complex reductive group does not have such a nice algebraic structure.) We will see that 
 $\widetilde{\Sigma}_F(G)$ is an affine algebraic variety with a 
 natural action of the connected reductive group $\widehat{G}$ (acting by conjugation on parameters). 
Let $\Sigma_F( G) = \widetilde{\Sigma}_F(G)/\!/\widehat{G}$ in the sense of invariant theory, i.e., $\Sigma_F(G)$ is the 
affine algebraic variety over $\C$ whose ring of functions is the ring of $\widehat{G}$-invariant functions on 
$\widetilde{\Sigma}_F(G)$. We call the variety $\Sigma_F(G)$ as the parameter space of Langlands parameters for the group $G$ over $F$.

It is easy to see that for a finite group $D$, and for any complex reductive group $H(\C)$ with finitely many connected components, 
the set of homomorphisms $\phi: D \rightarrow H(\C)$ 
is a finite disjoint union 
of homogeneous spaces for $H(\C)$ (for the action of $H(\C)$ on the set of homomorphisms 
$\phi: D \rightarrow H(\C)$ by inner conjugation), or up to $H(\C)$-conjugacy,  a finite set of points. These homogeneous 
spaces have the form $Z(\phi)\backslash H(\C)$ where $Z(\phi)$ is the centralizer 
of $\phi: D \rightarrow H(\C)$ inside $H(\C)$. Since these centralizers are reductive groups, 
it follows that the set of homomorphisms $\phi: D \rightarrow H(\C)$ 
has a natural structure of  a smooth affine variety with finitely many connected components.

Since homomorphisms from $W_F$ to ${}^LG(\C)$ 
have finite image when restricted to $I_F$, 
it follows that the set of homomorphisms from $W_F$ to  ${}^LG(\C)$ 
is a  disjoint union of varieties parametrized by their restrictions 
to $I_F$. 
Given a homomorphism
$\phi_0: I_F \rightarrow {}^LG(\C)$, 
extensions of $\phi_0$ to $W_F$
 is determined by where $\Fr$ goes in ${}^LG(\C)$
which must be an element of ${}^LG(\C)$ which normalizes $\phi_0(I_F)$, and in fact must operate on   
$\phi_0: I_F \rightarrow {}^LG(\C)$, through the action of $\Fr$ on $I_F$. If $Z(\phi_0)$ denotes the centralizer
of $\phi_0$ in $\widehat{G}$ and $(t_0,\Fr)$ is the image of $\Fr \in I_F \rtimes \Fr = W_F$ in $\widehat{G} \rtimes W_F$ 
for a particular extension of $\phi_0: I_F \rightarrow {}^LG(\C)$ to $W_F$, 
then clearly the various choices for the image of $\Fr$ in possible extensions 
of $\phi_0$ from $I_F$ to $W_F$ are of the form $(st_0,\Fr)$ for $s \in Z_{\widehat{G}}(\phi_0)$. 
Thus all the possible extensions of $\phi_0$ to $W_F$ form a coset space $Z_{\widehat{G} }(\phi_0)\cdot t_0$ 
of $Z_{\widehat{G} }(\phi_0)$ in $\widehat{G}$; the inner conjugation action of the element $(t_0, \Fr) \in \widehat{G} \rtimes \Fr$ preserves
$Z_{\widehat{G} }(\phi_0)$.
One knows
that $Z_{\widehat{G} }(\phi_0)$ is a reductive group, and modifying $t_0$ by an element of $Z_{\widehat{G} }(\phi_0)$, we can assume that $(t_0,\Fr)$ acts on
$Z_{\widehat{G} }
(\phi_0)$ by an automorphism of 
finite order.

Using the restriction map $\Hom[W_F, {}^LG(\C)] \rightarrow \Hom[I_F, {}^LG(\C)]$, we find that the set 
$\Hom[W_F, {}^LG(\C)]$
is an affine variety which is fibered over the affine variety $\Hom[I_F, {}^LG(\C)]$ by affine varieties.

Note that the set of $Z_{\widehat{G} }(\phi_0)$ conjugacy classes in 
the coset space $Z_{\widehat{G} }(\phi_0)\cdot t_0$ is nothing but the $t_0$-{\it twisted} conjugacy classes in 
$Z_{\widehat{G} }(\phi_0)$: two elements $z_1$ and $z_2$ in $Z_{\widehat{G} }(\phi_0)$ are said to be $t_0$-twisted conjugate if $z_1 = zz_2({}^{t_0} z)^{-1}$ where 
for $z \in Z_{\widehat{G} }(\phi_0)$, ${}^{t_0}z = t_0zt_0^{-1}$. Denote this equivalence relation on $Z_{\widehat{G} }(\phi_0)$ by $\sim_{t_0}$, and the quotient of
$Z_{\widehat{G} }(\phi_0)$  by this equivalence relation by $Z_{\widehat{G} }(\phi_0)/\!/\sim_{t_0}$.  
Just like the space of conjugacy classes in a reductive group, the space $Z_{\widehat{G} }(\phi_0)/\!/\sim_{t_0}$ 
is a smooth affine variety. (Assuming that the automorphism fixes a pinning $(B,T, \{X_\alpha \})$ in 
$Z_{\wG}(\phi_0)$), 
as we can by Steinberg after conjugating the automorphism by an inner 
automorphism---which does not change the variety of twisted conjugacy classes, 
the twisted conjugacy classes can be identified to the conjugacy classes in the fixed points of the automorphism.)

It follows that given 
$\phi_0: I_F \rightarrow {}^LG(\C)$, the set of 
conjugacy classes of extensions of $\phi_0$ to $\phi: W_F \rightarrow {}^LG(\C)$  
 is parametrized by 
$Z_{\widehat{G} }(\phi_0)/\!/\sim_{t_0}$, 
and the set of homomorphisms $\phi: W_F \rightarrow {}^LG(\C)$ whose restriction
to $I_F$ is a conjugate of $\phi_0: I_F \rightarrow {}^LG(\C)$, 
 is parametrized by  $(Z_{\widehat{G} }(\phi_0)\backslash \widehat{G}) \times (Z_{\widehat{G} }(\phi_0)/\!/\sim_{t_0})$.

We summarize our  discussion in the following proposition.

\begin{proposition} \label{prop44}The variety $\Sigma_F(G)$ 
is a countable disjoint union of smooth irreducible 
affine varieties. Each connected component of $\Sigma_F(G)$ is a connected
component of  $Z_{\widehat{G} }(\phi_0)/\!/\sim_{t_0}$, where $Z_{\widehat{G} }(\phi_0)$ is a reductive group which is the centralizer 
of $\phi_0: I_F \rightarrow {}^LG(\C)$ in $\widehat{G}$, and $\sim_{t_0}$ is twisted conjugacy 
for an automorphism  $t_0$ of $Z_{\widehat{G} }(\phi_0)$ of finite order.
\end{proposition}

It is clear that a morphism of $L$-groups, $\Phi: \LG\rightarrow \LG'$, 
induces a morphism of affine varieties,
$\widetilde{\Sigma}_F(G)(\Phi): \widetilde{\Sigma}_F(G) \rightarrow \widetilde{\Sigma}_F(G')$ which descends to a morphism
$\Sigma_F(G)(\Phi): \Sigma_F(G) \rightarrow \Sigma_F (G')$ of algebraic varieties.

\begin{lemma}\label{lemma4} 
The morphism $\Sigma_F(G)(\Phi): \Sigma_F(G) \rightarrow \Sigma_F(G')$ of affine algebraic varieties 
associated to a morphism of $L$-groups, $\Phi: \LG\rightarrow \LG'$ which we assume has finite kernel,
is proper in the sense that the image under $\Sigma_F(G)(\Phi)$ of a connected component
of $\Sigma_F(G)$ is a closed subvariety of $\Sigma_F(G')$, and that for any connected component of $\Sigma_F(G')$, there are
only finitely many connected component of $\Sigma_F(G)$ which map into it.
\end{lemma}
\begin{proof} As the set of connected components in $\Sigma_F(G)$ is, up to finite ambiguity, classified 
by homomorphisms of $I_F\rightarrow {}^LG$, and the map 
$\Phi: \LG\rightarrow \LG'$ which we assume has finite kernel, it is clear that  
for any connected component of $\Sigma_F(G')$, there are
only finitely many connected component of $\Sigma_F(G)$ which map into it.

Using the `local' models of $\Sigma_F(G)$ as $(Z(\phi_0)/\!/\sim_{t_0})$,
to prove the assertion on properness, it suffices to observe that the map 
$$Z(\phi_0)/\!/\sim_{t_0}  \longrightarrow  Z(\phi'_0)/\!/\sim_{t'_0},$$
is finite onto its image, which is easy enough to check.
\end{proof}
We use the lemma to define the notion of `fiber multiplicity' of the morphism $\Sigma_F(G)(\Phi): 
\Sigma_F(G) \rightarrow \Sigma_F( G')$ of affine algebraic varieties 
associated to a morphism of $L$-groups, $\Phi: \LG\rightarrow \LG'$ at any point $\mu \in \Sigma_F(G')$.  
Take the connected component of $\Sigma_F(G')$ passing through $\mu$. By the Lemma, only finitely many 
irreducible components of  $\Sigma_F(G)$ will map into this connected component with the image containing $\mu$. 
On each such irreducible component 
of $\Sigma_F(G)$, the proper morphism   $\Sigma_F(G)(\Phi)$ 
has a degree onto its image. The sum of these degrees 
(over irreducible components of  $\Sigma_F(G)$ whose image contains $\mu$) is defined   
to be the degree of $\Sigma_F(G)(\Phi)$ at $\mu$.

As an example, we may be in a situation with three connected components $X_1, X_2,X_3$ mapping to one connected component $Y$,
through mappings $\phi_1,\phi_2,\phi_3$ on the three connected components, 
with $ y=\phi_1(x_1)=\phi_2(x_2)=\phi_3(x_3)$ in $Y$. To calculate the degree of $\Phi$ at the point $y \in Y$ which has 3 
inverse images in $X$ located in distinct irreducible components $X_1,X_2,X_3$ (of possibly different dimensions), we consider the restriction of 
$\Phi$ to $X_i$ denoted $\phi_i$ but considered as a finite map from $X_i$ to $\phi_i(X_i) \subset Y$, 
and calculate the 
degree at $x_i$ for the maps $\phi_i: X_i\rightarrow \phi_i(X_i)$, and then we add up the degrees at $x_i$ to arrive at the degree at $y$. 

\vspace{4mm}

\noindent{\bf Example (unramified parameters):} 
Assume that $F$ is a non-Archimedean local field and
$G$ is an unramified group, i.e. $G$ is quasi-split and
splits after an unramified extension of $F$. A parameter $\phi: W'_F\rightarrow \LG$ is said to be unramified if 
it is trivial on $I_F \times \SL_2(\C)$ where $I_F$ is an inertia subgroup of $W_F$. Clearly, the set of unramified
parameters form an irreducible component of $\Sigma_F(G)$, and the ring of algebraic functions on this
connected component is nothing but the algebraic functions on $\widehat{G}$ which are $\LG$ invariant which by the Satake isomorphism is 
the Hecke algebra of $G(F)$ with respect to a hyperspecial maximal compact subgroup of $G(F)$.  
 
If $\Phi: \LG\rightarrow \LG'$ is an injective homomorphism of $L$-groups, 
then the morphism $\Sigma_F(G)(\Phi): \Sigma_F(G) \rightarrow \Sigma_F(G')$ 
of parameter spaces can map only the unramified
component of $\Sigma_F(G)$ into the unramified component of $\Sigma_F(G')$, and the degree of the map (onto its image) 
being a measure of how many distinct conjugacy classes in $\widehat{G}$ (under $\LG$) become the same in $\widehat{G}'$ 
(under $\LG'$). 

In the context of this paper, it may be useful to consider the case of quadratic base change for an unramified 
extension $E/F$ for a semisimple 
group $G$ which is quasi-split and split by the same extension $E/F$. Let $S$ be a maximal split torus in $G$, and $T =Z(S)\supset S$ 
a maximal torus of $G$  over $F$ with $W$ 
the absolute Weyl group, and $W^F = N(S)(F)/T(F)$, the Weyl group of $(G,S)$ over $F$ which is the $F$-rational points of the (absolute) Weyl group $W = N(T)(\bar{F})/T(\bar{F})$, 
on which $\Gal(\bar{F}/F)$ operates via $\Gal(E/F)$ by an automorphism $\phi$ of order 2. 

\begin{lemma}\label{lemma5}The degree of the base change map $BC: \Sigma_F(G) \rightarrow \Sigma_E(G)$ restricted to the unramified components of 
$\Sigma_F(G)$ and  $\Sigma_E(G)$ is $\left | H^1(\langle \phi \rangle
, \wT) \right|$ which is known to be 1 if $G$ is simply connected  or adjoint.
\end{lemma}
\begin{proof} Let $\wT \subset \wG$ be the standard embedding. We have ${}^LT = \wT \rtimes \langle \phi \rangle$, 
${}^LG = \wG \rtimes \langle \phi \rangle$. Further, if $N$ is the normalizer of $\wT$ in $\wG$, then $N \rtimes \langle \phi \rangle$, 
operates on $\wT$. Let $N^\phi$ be the inverse image in $N$ of $W^\phi$ which is the fixed points of $\phi$ on $W$. According to
[Bo], equivalence classes of unramified parameters of $G(F)$ are in bijective correspondence with 
$[\wT \rtimes  \phi ]/N^\phi$
which is isomorphic to $\wS/W^\phi$ under the natural map ${\hat{i}}: \wT \rightarrow \wS$, which is dual to the
inclusion $i: S \rightarrow T$, and which is given by dividing $\wT$ by $K = \{\phi(t) t^{-1}| t \in T(\C)\}$.

Clearly, the basechange map sends an element $(t,\phi) \in {}^LT $, to $(t,\phi)^2 = (t\cdot \phi(t), 1)$. 

Let $K$ be the subgroup $K= \{\phi(t) \cdot t^{-1} | t \in \wT \}$ of $\wT$ which is contained in the subgroup $\{t \in \wT| \phi(t)=t^{-1}\}$.
Then we have sequence of maps,
\[{\wT \longrightarrow \wT/K \cong \wS   {\longrightarrow \wT/   \{t \in \wT| \phi(t)=t^{-1}\} \stackrel{t\cdot \phi(t)} \longrightarrow}   \wT.} \]
Since the last arrow $\wT/   \{t \in \wT| \phi(t)=t^{-1}\} \stackrel{t\cdot \phi(t)} \longrightarrow \wT$ is clearly injective, 
degree of the composite 
map ${\wS \longrightarrow \wT/   \{t \in \wT| \phi(t)=t^{-1}\} \stackrel{t\cdot \phi(t)} \longrightarrow \wT}$ onto its image
is the degree of the map $\wS \longrightarrow \wT/   \{t \in \wT| \phi(t)=t^{-1}\} $ which is nothing but the index
of $K$ inside $\{t \in \wT| \phi(t)=t^{-1}\} $ which is clearly $\left | H^1(\langle \phi \rangle
, \wT) \right|$. It remains to observe that in the mapping $\wS \longrightarrow \wT$, 
two elements of $\wS$ which are $W$-conjugate in $\wT$ 
are $W^\phi$ conjugate in $\wS$, thus the degree of the mapping $\wS \longrightarrow \wT$ onto its image is the same as the degree 
of the mapping  $\wS/W^\phi \longrightarrow \wT/W$ onto its image. 
\end{proof}

\vspace{4mm}
\noindent{\bf Example ($\SL_2(F)$):} The set of parameters for $\SL_2(F)$ for which the $L$-group is $\PGL_2(\C)$
is a disjoint union of points corresponding to discrete series representations of $\SL_2(F)$, and certain affine spaces corresponding to principal series. 
A character $\chi: F^\times \rightarrow \C^\times$ gives rise to the parameter
 $\sigma_\chi: W_F \longrightarrow \PGL_2(\C)$ which is $(\chi, 1)$ sitting diagonally in $\PGL_2(\C)$
with $\sigma_\chi \cong \sigma_{\chi^{-1}}$.
If we write $F^\times = {\mathcal O}_F^\times \times \varpi^\Z$, and denote $\chi$ restricted to ${\mathcal O}_F^\times$ as $\chi_0$, 
then if $\chi_0^2 \not = 1$, the irreducible component of the parameter space passing through such a character can be 
identified  to $\C^\times$ via $\chi \rightarrow \chi(\varpi)$. If $\chi_0^2=1$, then $\chi(\varpi)$ and $\chi(\varpi)^{-1}$
gives rise to equivalent parameters. Thus the corresponding irreducible component of the parameter space is $\C^\times/\{z\sim  z^{-1} \}$ which is isomorphic to $\C$ via the map from $\C^\times$ to $\C$ given by $z\rightarrow z+z^{-1}$.

\vspace{8mm}

\noindent{\bf Parameter spaces for representations of $G(F)$:} 
Recall that according to Bernstein, there is a natural surjective map with finite fibers
from the set $\Pi(G)$  of irreducible admissible representations of $G(F)$ to a  countable disjoint union of
affine algebraic varieties (over $\C$) which we call the Bernstein variety ${\mathcal B}(G(F))$,  
each connected component of which is  indexed by pairs 
$(M,\pi)$ consisting of a Levi subgroup $M$ of $G$ together with a supercuspidal representation $\pi$ of $M(F)$, the pair 
$(M,\pi)$ taken up to $G(F)$-conjugacy; two pairs  $(M,\pi)$ and $(M,\pi \otimes \chi)$ are
declared to be in the 
same Bernstein component if $\chi$ is an unramified character of $M(F)$.   The algebra
of functions on this connected component is described as follows. Given $(M,\pi)$, let $T_\pi$ be the set of unramified 
twists $\pi \otimes \chi$ of $\pi$. Clearly, $T_\pi$ is the complex torus consisting of 
unramified characters of $M(F)$ modulo the finite group of unramified characters $\chi$ of $M(F)$ such that $\pi \otimes \chi \cong \pi$. Let $W_\pi$ be the subgroup of $W(M)= N_{G(F)}(M)/M$ consisting of those elements $w$ such that 
$w(\pi) \cong \pi \otimes \chi_w$ for some  unramified character $\chi_w$ of $M(F)$. The group $W_\pi$ operates on the torus $T_\pi$, and if $\C[T_\pi]$ is the algebra of regular functions on the torus $T_\pi$, then the algebra of functions on the
Bernstein component associated to $(M,\pi)$ is $\C[T_\pi]^{W_\pi}$.

Assuming that a Langlands correspondence is defined for supercuspidal representations of 
the group $G(F)$ as well as all it Levi subgroups, there is a natural way to extend it to a 
 morphism of algebraic varieties from the Bernstein variety 
${\mathcal B}(G(F))$ to  $\Sigma_F(G)$ by sending $(M,\pi \otimes \chi)$ to the Langlands parameter
of $\pi \otimes \chi$, a supercuspidal representation of $M(F)$, with values in ${}^LM$, which naturally sits inside
${}^LG$.

The commutativity of the following diagram  is usually called the {\it Haines' conjecture}.

$$
\xymatrix @C=17pc @R=4pc{
{}&
\ar@{->}[ld]_<(.4)[@!15]{{\rm Only~} W_F {\rm~~ part ~~of~~
the ~~Langlands~~~parameter}}
{\mathcal B}(G(F)) \\
 {\Sigma}_F(G)  & \ar[l]^{{\rm only~~} W_F{\rm~~part ~~ of ~ the~Langlands~~ parameter}} \ar[u]_[@!90]{\rm
Bernstein}\Pi(G).}
$$

The map from ${\mathcal B}(G(F))$ to  $\Sigma_F(G)$ 
(parameterizing only the equivalence classes of 
homomorphisms of the Weil group into the $L$-group) is an isomorphism of algebraic varieties for $G=\GL_n(F)$; for $G=\SL_n(F)$, the map from ${\mathcal B}(G(F))$ to  $\Sigma_F(G)$ is an isomorphism from each connected 
component onto the connected component of the image. However, for  general groups, 
the map seems more complicated in nature.

\begin{remark} One of the subtleties  about Langlands parameters is the difference between $W_F$ and 
$W'_F = W_F\times \SL_2(\C)
$. There is another description of representations of $W'_F = W_F\times \SL_2(\C) $ as that of $'{W}_F= \C \rtimes W_F$
where $W_F$ operates on $\C$ via the unramified character of $W_F$ taking a 
uniformizer to multiplication by $q$, the cardinality of the residue field. The two descriptions are the same (for complex representations!) because of Jacobson-Morozov theorem which identifies isomorphism classes of representations of 
$W'_F$ with that of ${}'W_F$. However,  the set of homomorphisms of  $W'_F$ and that of ${}'W_F$ 
into ${}^LG$ are quite different, and so are the quotient varieties of the set of such homomorphisms by $\wG$. For example,
in the case of $G(F)=\PGL_2(F)$, with $\wG=\SL_2(\C)$, 
 $\Sigma_G(F)$ that we defined has an isolated point corresponding to the 
2 dimensional irreducible representation of  $\SL_2(\C)$ (hence of $W'_F$), 
whereas for  ${}'W_F= \C \rtimes W_F$, the corresponding 
object (on which $W_F$ operates by $(\nu^{1/2},\nu^{-1/2})$, the diagonal matrix in $\SL_2(\C)$, 
and $t \in \C$ operates by the upper triangular unipotent matrix
\[\left ( \begin{array}{cc} 1 & tX \\ 0 & 1\end{array}\right )\] where $X\in \C$ is a fixed complex number) 
is the quotient of $\C$ by $\C^\times$, a non-Hausdorff space which reflects more 
correctly the topology on the unitary dual of $G(F)$. I have in this paper used $W'_F = W_F\times \SL_2(\C) $ instead of ${}'W_F= \C \rtimes W_F$, only to avoid having to deal with quotients by non-reductive groups, 
and thus suggestions in the paper where the difference between 
$W_F$ and $W'_F$ shows up must be taken with a pinch of salt.
\end{remark}

\begin{remark}We finally say a few words about parameter spaces in the Archimedean case. Just like non-Archimedean 
local fields, $W_\R$ and $W_\C$ are locally compact topological groups with unique maximal compact subgroups, which are
$\Si\rtimes \Z/2$ and $\Si$ in the two cases, with quotient isomorphic to $\R$. In fact, unlike the non-Archimedean case,
where $W_F $ is only a semi-direct product $I_F \rtimes \Z$, in the Archimedean case, the Weil group is a direct product
of its maximal compact subgroup and $\R$. Since homomorphisms  of compact groups inside a 
complex Lie group form a discrete set up to conjugation, the analysis done earlier to show that the set of Langlands 
parameters (without any admissibility hypothesis) forms a countable union of affine varieties holds good here too 
after we have noted that the set of homomorphisms $\phi: \R \rightarrow G(\C)$ can be identified to the Lie algebra
$\g$ of $G(\C)$, as they are all of the form $\phi_X:t\rightarrow \exp(tX)$ for $X \in \g$. The earlier analysis also
shows that these set of parameters up to conjugation by $\wG$, i.e., the categorical quotient of the parameter 
space by $\wG$,  form a disjoint union of smooth affine varieties. We will 
however not discuss explicitly geometry of parameter spaces in the Archimedean case any further.  Curiously, the book 
[ABV] does not seem to have geometry of the parameter spaces --- at least in the way we have defined --- explicitly discussed.
\end{remark}

\section{Character twists}
Given a parameter $\varphi$ for $G^{\out}(F)$ which we consider as
 an element of $ H^1(W_F,\widehat{G}^{\out}),$ our aim in this section is to construct 
certain pure innerforms of $G$ via self-twists of the parameter $\varphi$. A theorem of Kottwitz identifying 
$\Hom( H^0(W_F, Z(\widehat{G})),\Q/\Z)= \pi_0(Z(\wG)^{W_F})^\vee $ with $ H^1(W_F, G)$ plays an important role in this section, and thus in this section we consider only non-Archimedean local fields.  Pure innerforms of real groups, i.e., $H^1({\rm Gal}(\C/\R), G(\C))$ have more complicated structure
which we take up in the next section.

Let $Z(\widehat{G})$  
(resp. $Z(\widehat{G}^{\out})$) 
be the center of $\widehat{G}$ (resp. $\widehat{G}^{\out}$). These come equipped 
with an action of $W_F$. Further, as $W_F$-modules,  $Z(\widehat{G})^{\out} = Z(\widehat{G}^{\out})$.

We have a natural bilinear form:
$$Z(\widehat{G}) \times \widehat{G} \rightarrow  \widehat{G},$$  
giving rise to the bilinear pairing:
$$H^1(W_F, Z(\widehat{G})) 
\times H^1(W_F, \widehat{G}) \rightarrow  H^1(W_F,\widehat{G}).$$
  This pairing is a reflection on the parameter side of the notion of twisting representations of $G(F)$ by characters (i.e., one dimensional representations) of $G(F)$. Indeed, Langlands has proved, cf. [Bo] Desiderata 10.2, that  
an element of $H^1(W_F, Z(\widehat{G}))$ 
gives rise to a character of $G(F)$, something which is exactly the Langlands correspondence 
for tori if $G$ is a torus, or something almost as straightforward if $Z(\widehat{G})$ is a torus, but in general it needs more work 
to construct a character of $G(F)$ corresponding to an element of $H^1(W_F, Z(\widehat{G}))$,
which Langlands did using his notion of $z$-extensions.

\begin{lemma}\label{lemma6}If $\Omega_G(E) = H^1(\Gal(E/F), Z( \widehat{G}^{\out} )^{W_E}) 
$ is the kernel of the natural restriction map 
$H^1(W_F, Z( \widehat{G}^{\out} )) 
\rightarrow H^1(W_E, Z( \widehat{G}^{\out})) $, then $\Hom(\Omega_G(E),\Q/\Z)$ 
sits in the exact sequence,
$$0 \rightarrow  \Hom(\Omega_G(E),\Q/\Z) \rightarrow  \Hom( H^0(W_F, Z(\widehat{G})),\Q/\Z)
  \rightarrow \Hom(H^0(W_E,Z(\widehat{G})),\Q/\Z).$$   
\end{lemma}
\begin{proof}
Consider the exact sequence of $W_F$-modules,
$$1 \rightarrow  Z(\widehat{G}^{\out}) \rightarrow  {\rm Ind}_{W_E}^{W_F} Z(\widehat{G}) \rightarrow Z(\widehat{G}) \rightarrow 1.$$   
The corresponding long exact sequence of cohomology groups as $W_F$-modules gives:
$$\cdots \rightarrow  H^0(W_E, Z(\widehat{G})) \rightarrow  H^0(W_F, Z(\widehat{G})) \rightarrow H^1(W_F, Z(\widehat{G}^{\out}) )
\rightarrow H^1(W_E, Z(\widehat{G}^{\out})) \cdots
$$   
Therefore, by the definition of $\Omega_G(E)$, we have:
$$\cdots \rightarrow  H^0(W_E, Z(\widehat{G})) \rightarrow  H^0(W_F, Z(\widehat{G})) \rightarrow \Omega_G(E) \rightarrow 0.
$$   
Using the functor $\Hom(-,\Q/\Z)$, we get the desired exact sequence:
$$0 \rightarrow  \Hom(\Omega_G(E),\Q/\Z) \rightarrow  \Hom( H^0(W_F, Z(\widehat{G})),\Q/\Z)  \rightarrow \Hom(H^0(W_E,Z(\widehat{G})),\Q/\Z).$$   
\end{proof}

\begin{corollary}Interpreting $\Hom( H^0(W_F, Z(\widehat{G})),\Q/\Z)$ by Kottwitz as $H^1(W_F, G)$, 
i.e., the group of pure innerforms of $G$ over $F$, we find that the character group of $\Omega_G(E)$ 
is isomorphic to $ H^1(\Gal(E/F), G(E))$, 
the finite abelian group of pure innerforms of $G$ over $F$ which become trivial over $E$. 
Thus, we get a perfect pairing:
$$ H^1(\Gal(E/F), Z( \widehat{G}^{\out} )^{W_E}) 
\times  H^1(\Gal(E/F), G(E)) \rightarrow \Q/\Z.$$
\end{corollary}

Now given a parameter $\varphi$ for $G^{\out}$ which we consider as
 an element of $ H^1(W_F,\wG^{\out}),$ 
consider the subgroup $\Omega_G(\varphi, E)$ of $H^1(\Gal(E/F), Z( \widehat{G}^{\out} )^{W_E})$ 
which is the stabilizer  of $\varphi \in H^1(W_F,\widehat{G}^{\out})$ under the bilinear pairing:
 $$H^1(W_F, Z(\widehat{G}^{\out})) 
\times H^1(W_F, \widehat{G}^{\out}) \rightarrow  H^1(W_F,\widehat{G}^{\out}).$$
Now using the perfect pairing in the previous corollary:
$$ H^1(\Gal(E/F), Z( \widehat{G}^{\out} )^{W_E}) \times  H^1(\Gal(E/F), G(E)) 
\rightarrow \Q/\Z,$$
consider the annihilator ---call it $A_G(\varphi, E)$--- in $  H^1(\Gal(E/F), G(E)) $ 
of $\Omega_G(\varphi, E) \subset H^1(\Gal(E/F), Z( \widehat{G}^{\out} )^{W_E}) $.
The subgroup $A_G(\varphi, E)$ of $ H^1(\Gal(E/F), G(E))$ defines certain pure innerforms of $G$ over $F$ which trivialize over $E$. We also have a perfect pairing,
$$H^1(\Gal(E/F), Z( \widehat{G}^{\out} )^{W_E}) /\Omega_G(\varphi, E) \times A_G(\varphi, E) \rightarrow \Q/\Z,$$
meaning that in the orbit of character twists of $\varphi$ (which go to a particular parameter under the basechange to $E$) 
there are exactly as many parameters as there are pure innerforms of $G$ over $F$ which trivialize after basechange to $E$. 
 We record this as a proposition which will play a role in our conjecture below.

\begin{proposition} \label{prop4}Given a Langlands parameter $\varphi$ for $G^{\out}(F)$ considered as
 an element of $ H^1(W_F,\wG^{\out}),$ the orbit $\chi \cdot \varphi$ of character twists 
of $\varphi$ by elements $\chi$ of $H^1(W_F, Z(\wG^{\out}))$ with $\chi|_{W_E}=1$  
is in bijective correspondence with certain pure innerforms of $G$ over $F$ which become trivial over $E$, which form a group of pure innerforms of $G$ denoted above by $A_G(\varphi,E)$, in particular contains the trivial element, i.e., the quasi-split innerform of $G$ over $F$ which is the basepoint.
\end{proposition} 

\section{Cohomology versus Galois cohomology of real groups}

In this  section, given a parameter $\varphi$ for $G^{\out}(\R)$, we attach a  pure innerform of $G$ over $\R$.  
This will be achieved by results of Borovoi, and J. Adams which we now recall.

The following result of Borovoi, cf. [Bor],
 generalizes a result from  [Se], III \S4, Theorem 6, which is about Galois cohomology of compact 
connected Lie groups, to all connected reductive groups over $\R$. See also Proposition 4.5 of [Ad2].

\begin{thm}Let $G$ be a connected reductive group over $\R$, and $T$ a maximal torus in $G$, with $W(\R)$
the real points of $N(T)/T$. Then there is a natural action of $W(\R)$  on  $H^1(\Gal(\C/\R), T(\C))$ 
such that the fibers of the natural map  $H^1(\Gal(\C/\R), T(\C)) \rightarrow H^1(\Gal(\C/\R), G(\C))$ 
are the orbits of $W(\R)$, i.e., there is an injection of sets
$H^1(\Gal(\C/\R), T(\C))/W(\R) \subset  H^1(\Gal(\C/\R), G(\C)).$  This injection is an isomorphism of sets if $T$ is 
fundamental, i.e., contains a maximal compact torus of $G(\R)$.
\end{thm}

A real structure on $G(\C)$ gives rise to two commuting involutions on $G(\C)$, one the complex conjugation (with fixed points $G(\R)$), and the other 
a Cartan involution of $G(\R)$ extended to $G(\C)$ as an algebraic automorphism. 
One can reverse  the role of the two involutions, and begin with an (algebraic) involution on 
$G(\C)$ which is a complex reductive group, and construct uniquely (up to equivalences) a real structure on $G$. Since this will be useful to us, we review it. Let $\theta$ be an involution on $\G(\C)$ with $K(\C) = G(\C)^\theta$, and $K$ a fixed
maximal compact subgroup of $K(\C)$. Let
the Lie algebra $\g$ of $G(\C)$ be decomposed into the two eigenspaces of $\theta$ as $\g = \kk + \p$. The action 
of the compact group $K$ on $\p$ lands inside $\SO(\p)$ defined using the Killing form of $\g$ restricted to $\p$.
By the classification of maximal compact subgroup of $\SO_n(\C)$, it follows that there is a real subspace $\p_\R$ of $\p$ 
invariant under $K$ and on which the killing form is positive definite. If the Lie algebra of $K$ is $\kk_\R$, then 
$\g_\R=\kk_\R+ \p_\R$ defines a Lie algebra over $\R$ with $\g_\R \otimes \C = \g$. The corresponding subgroup 
$G(\R)$ of $G(\C)$ is a real structure on $G(\C)$ with $K$ as a maximal compact subgroup of $G(\R)$ (note that the Killing form of $\g$ restricted to $\kk_\R$ is negative definite, and is positive definite on $\p_\R$). 

A useful corollary of the above is that given an involution $\theta$ on $G(\C)$ giving rise to a real structure $G(\R)$, 
there is  a maximally compact torus $T(\R)$ of $G(\R)$ for which $T(\C)$ is a minimally $\theta$-split torus of $G(\C)$, 
i.e., $T(\C)$ is $\theta$ stable, and the part of $T(\C)$ on which $\theta$ operates as $t \rightarrow t^{-1}$ is of minimal possible dimension. In fact, we can write the Lie algebra $\ttt$ of $T(\C)$ as the sum of the two eigenspaces for 
$\theta$ as $\ttt = \ttt_1 + \ttt_2$ with the Lie algebra $\ttt_\R$ for $T(\R)$ with the property that
$\ttt_\R = \ttt_1 \cap \ttt_\R + \ttt_2 \cap \ttt_\R$  where 
$\ttt_1 \cap \ttt_\R$ is the Lie algebra of the maximal compact torus 
of $T(\R)$ and  $ \ttt_2 \cap \ttt_\R$ is the Lie algebra of the maximal split torus of $T(\R)$. This has the consequence 
 for us that for the torus $T$ above, $W(\C)^\theta = W(\R)$. To prove this, first we note that we can assume 
for purposes of Weyl groups that the group $G$ is adjoint, and then we can use Lie algebras to understand the action
of the Weyl group. Now, if an element $w \in W(\C)$ is $\theta$-invariant, it must preserve $\ttt_1$ and $\ttt_2$,
and since $\ttt_1 \cap \ttt_\R$ is the Lie algebra of the maximal compact torus 
of $T(\R)$ and  $ \ttt_2 \cap \ttt_\R$ is the Lie algebra of the maximal split torus of $T(\R)$, these get preserved too, hence
$\ttt_\R$ is preserved, proving that a $\theta$ invariant element of $W(\C)$ is in $W(\R)$. Similarly the converse.

\begin{remark}\label{15}
An involution on $G(\C)$ defines an element in the outer automorphism group of $G(\C)$, and hence defines a 
quasi-split group $G(\R)$ obtained by twisting the split group by that outer automorphism. It can be seen that the real 
structure on $G(\C)$ associated to an involution $\theta$ discussed above is an innerform of the
group $G^{\out}$. 
\end{remark}

The following result of J. Adams, cf. [Ad2], and Borovoi [Bor], 
compares the cohomology of $G(\C)$ with respect to these two
involutions, incorporating the above remarks on tori and Weyl groups. It will be used for Levi subgroups 
of $\wG$, so we make notational change here.

\begin{thm}Let $\wM$ be a connected reductive group over $\C$, 
and $\varphi$ an involutive algebraic 
automorphism of  $\wM(\C)$. Then there is a natural real structure on $\wM(\C)$, and an identification of the sets:
$$ H^1(\langle \varphi \rangle, \wM(\C)) 
\cong H^1(\Gal(\C/\R), \wM(\C)) \cong
H^1(\langle \varphi \rangle, T(\C))/W^{\varphi},$$
where $T$ is a $\varphi$ stable maximal torus of $\wM(\C)$ for which $T^\varphi$ has the largest dimension, and is the complexification of a fundamental torus for $\wM(\R)$; $W^\varphi$ is the $\varphi$ fixed 
points of the action of $\varphi$ on $W = N(T)/T$, and the action of $W^\varphi$ on $H^1(\langle \varphi \rangle, T(\C)),$
as in Borovoi's theorem,  arises from the exact sequence,
 $$1 \rightarrow T(\C) \rightarrow N(T)(\C) \rightarrow W \rightarrow 1,$$
of groups with an action of  $\langle \varphi \rangle 
= \Z/2$.
\end{thm}

Cohomologies of $T$ and $T^{\out}$ are related via the following lemma.

\begin{lemma} \label{lemma6}For a torus $T$ over $\R$, let $T^{\out}$ be the torus $T$ twisted by the automorphism $t\rightarrow t^{-1}$. Let ${}^LT = \wT \rtimes \langle \varphi \rangle $ be the $L$-group of $T$, where $\varphi$ denotes the action of 
$\Gal(\C/\R)$ on $\widehat{T}$. Then, we have a natural isomorphism (equivariant under 
automorphisms of $T$ which act on the two sides):
$$ H^1(\langle \varphi \rangle, \widehat{T}(\C)) \cong H^1(\Gal(\C/\R), T^{\out}(\C))^\vee.$$
\end{lemma}
\begin{proof}
Note the exponential sequence,
 $$0 \rightarrow 2\pi i \Z \rightarrow  \C \rightarrow \C^\times \rightarrow 1,$$
which is equivariant for $\Gal(\C/\R)$; since $\bar{i} = -i$, the action of $\Gal(\C/\R)$ 
on $2\pi i \Z$ is $a\rightarrow -a$, which is of crucial importance to us.

Tensoring the above exponential sequence  by $X_\star(T^{\out})$, and noting the above change of sign which 
takes  $X_\star(T^{\out})$ to  $X_\star(T)$ as Galois modules, we have the exact sequence of Galois modules:
 $$0 \rightarrow X_\star(T) \rightarrow X_\star(T^{\out}) \otimes \C \rightarrow T^{\out}(\C)\rightarrow 1.$$
From the long exact sequence of cohomology groups, it 
 follows that $$\hspace{3cm}H^1(\Gal(\C/\R), T^{\out}(\C)) \cong H^2(\langle \sigma \rangle, X_\star(T)). \hspace{3.6cm}(1)$$ 

Similarly, using the exponential sequence for the torus $\wT$ with the action $\varphi$ (and not for the Galois group!), 
we have
$$\hspace{2cm}H^1(\langle \varphi \rangle, \widehat{T}(\C)) \cong H^2(\langle \sigma \rangle, X_\star(\wT))\cong H^2(\langle \sigma \rangle, X^\star(T)). \hspace{3cm} (2)$$

Using the isomorphisms $(1)$ and $(2)$, to prove the lemma, we need to prove that 
$H^2(\langle \sigma \rangle, X_\star(T) )$ and  $H^2(\langle \sigma \rangle, X^\star(T))$ are 
canonically dual to each other. This is a consequence of the general Duality theorems for finite cyclic groups;
we refer to the book of K. Brown, {\it Cohomology of groups}, Chapter VI, exercise 3 of \S7, as a precise reference.
\end{proof}

The next lemma compares tori in $G$ and $G^{\out}$.

\begin{lemma} \label{lemma-torus}Let $G$ be a quasi-split reductive group over $\R$ with $G^{\out}$, the associated 
quasi-split reductive group over $\R$. 
Let $T_s$ and $T_f$ be respectively maximally split and maximally compact tori
in $G$. Then $T_f$ is obtained  from $T_s$ by twisting by the automorphism $\omega_0$, an element of the longest length 
in the Weyl group for the maximal torus $T_s$ of $G$. 
The map $T\longrightarrow T^{\out}$ gives a bijective correspondence between stable conjugacy classes of tori in $G(\R)$ 
and stable conjugacy classes of tori in $G^{\out}(\R)$.  Further, under this correspondence between 
stable conjugacy classes of tori in $G(\R)$ and $G^{\out}(\R)$, $W_G(T)(\R)=W_{G^{\out}}(T^{\out})(\R)$.
\end{lemma} 

\begin{proof}The first part of the lemma is proved along the way in  Lemma \ref{lemma3}.

For comparing tori in $G$ and $G^{\out}$, we assume that $G$ is split. (One of the two groups $G$ or $G^{\out}$ 
is always split except for
forms of $\SO_{2n}$, $n$ even, but for which $G = G^{\out}$, and the conclusion is along the same lines as we give below.)

Since $G^{\out}$ is obtained by twisting $G$ by an automorphism $C_0$ which restricted to the maximal split torus $T_s$ is $-w_0$,
the corresponding torus in $G^{\out}$ is $T_o= (T_s^{w_0})^{\out}$.

Since we are assuming that $G$ is split, stable conjugacy classes of tori, 
which are given by $H^1(\Gal(\C/\R), W)$, is the same as 
conjugacy classes of involutions in $W$. On the other hand,  
stable conjugacy classes of tori in $G^{\out}$  
are given by $H^1(\Gal(\C/\R), W^{\out})$, where $W^{\out}$ 
is the Weyl group of $G^{\out}$, which is $W$ on which the action of $\Gal(\C/\R)$ is by conjugation by 
the involution $w_0$ in $W$.

It is clear that 
\begin{eqnarray*} 
H^1(\Gal(\C/\R), W)   \
& \longrightarrow & H^1(\Gal(\C/\R), W^{\out}) \\
w & \longrightarrow & w_0w  
\end{eqnarray*}
gives a bijective correspondence of cohomology groups (actually not groups but pointed sets only), 
sending the torus $T_s^w$ to $T_o^{w_0 w}$ where $T_o= (T_s^{w_0})^{\out}$,
thus the correspondence of cohomology sends the torus $T_s^w$ to $(T_s^w)^{\out}$, as asserted in the lemma. 

To prove the statement on Weyl groups, note that for a torus $T_s^w$ in $G(\R)$, represented by 
$w \in H^1(\Gal(\C/\R), W) $, $W_G(T_s^w)(\R) = W^w$, the centralizer of $w$ in $W$. Similarly, 
for a torus $(T_s^w)^{\out}= T_o^{w_0w}$ in $G^{\out}(\R)$, represented by 
$w_0w \in H^1(\Gal(\C/\R), W^{\out}) $, 
$$W_{G^{\out}}((T_s^{w})^{\out})(\R) = \{\lambda \in W^{\out}(\C) | \lambda^{w_0w} = \lambda\} = \{\lambda \in W | \lambda^{w_0ww_0} = \lambda\},$$ proving that  
$W_G(T)(\R)$ is conjugate to $W_{G^{\out}}(T^{\out})(\R)$ in $W(\C)$.
\end{proof}

\noindent{\bf Example:} For the group $G= \GL_{2n}(\R)$, the maximal split torus is $\R^{\times 2n}$, whereas the 
maximal compact torus is $\C^{\times n}$. For $\GL_{2n}(\R)$, the element $w_0$ can be taken to be the involution
$w_0=(1,2n)(2,2n-1)\cdots(n,n+1)$. Clearly, $T_s^{w_0}=T_f$. The group $G^{\out}$ is $\U(n,n)$ for which the maximal split torus is $\C^{\times n}$, and maximal compact torus is $\Si^{2n}$, in conformity with the Lemma.

\vspace{3mm}

\noindent{\bf Example:} For the group $G= \GL_{2n}(\R)$, $G^{\out} = \U(n,n)$. It is well known that 
a maximal torus $T$ in $G$ is of the form $A^\times$ for $A$ a commutative separable algebra over $\R$ 
of dimension $2n$. It is also known that any maximal torus in $U(n,n)$ is of the form
$T_A= \{x \in A \otimes_\R \C | x\bar{x}=1 \}$. It is easy to see that $T_A=T^{\out}$.

\vspace{3mm}

We now recall the following presumably well-known lemma.

\begin{lemma} \label{lemma7} Let $G^{\out}$ be a connected reductive quasi-split 
group over $\R$. Then a parameter $\varphi: W_\R\rightarrow {}^LG^{\out}(\C)$ for $G^{\out}$ 
canonically determines  a maximal torus $T_{\varphi}$ and a Levi subgroup $M_\varphi$ over $\R$ 
with $T_{\varphi} \subset M_\varphi \subset G^{\out}(\R)$, $T_\varphi$ elliptic in $M_\varphi$,
 with $\LT_\varphi = \wT_\varphi\rtimes \langle \varphi(j) \rangle 
\subset \LM_\varphi = \wM_\varphi\rtimes \langle \varphi(j) \rangle$ (but not  necessarily 
contained in  $\LG^{\out}$!), and with $\wT_\varphi$ having largest dimensional 
torus in $\wM_\varphi$  on which  $\langle \varphi(j) \rangle$ acts trivially. 
Further, the pair $(T_\varphi, M_\varphi)$ depends only on 
$\varphi|_{\C^\times}$.
\end{lemma}
\begin{proof}Write $W_\R = \C^\times \cdot j$ with $j^2=-1 \in \C^\times$. Let $\wM_\varphi$ be the 
centralizer of $\varphi(\C^\times)$ in $\widehat{G}$. Thus, $\wM_\varphi$ is a connected reductive group over 
$\C$ containing a maximal torus of $\widehat{G}$.
 
Clearly $\varphi(j)$, an element of ${}^LG^{\out}(\C)$ leaves $\wM_\varphi$ invariant, and acts on $\wM_\varphi$ as an involution.
Thus $\wM_\varphi$ is a Levi subgroup of ${}^LG^{\out}(\C)$ corresponding to a Levi subgroup $M_\varphi$ of $G^{\out}(\R)$, cf. [Bo], Lemma 3.5.
  Fix a pair $(B,S)$ inside $\wM_\varphi$
consisting of a Borel subgroup $B$ of $\wM_\varphi$ and a maximal torus $S$ of $\wM_\varphi$ which are left invariant under $\varphi(j)$.

This gives rise to a 
pair $(S, \varphi(j))$ 
consisting of a maximal torus $S$ in $\wM_\varphi$ together with an involution $\varphi(j)$ on $S$. It is known that the pair
$(S, \varphi(j))$  is unique up to conjugacy by $\wM_\varphi^{\varphi(j)}$. 
We assume, after conjugating $\varphi$ if necessary, that $S$ is the `standard'
maximal torus  of $\widehat{G}$, i.e., $S =\widehat{T}$, the dual torus 
corresponding to the maximally 
split maximal torus $T$ in $G^{\out}$.

 The automorphism $\varphi(j)$ of $S=\widehat{T}$ gives rise to a 
canonical involution on the torus $T$. The torus $T_{\varphi}$ is the torus over $\R$ obtained by twisting $T$ by $\varphi(j)$, i.e.,
$T_\varphi(\R) = \{t \in T(\C) | \varphi(j)(t)= \bar{t}\}.$ Thus the involutive action of $\varphi(j)$ on $\wT$ makes $
\wT\rtimes \langle \varphi(j) \rangle $ into the $L$-group of $T_\varphi$. It can be seen that $T_{\varphi}$ is contained in $M_\varphi(\R)$.

If $\varphi': W_\R \rightarrow {}^LG^{\out}$ is another parameter for $G^{\out}(\R)$ with $\varphi|_{\C^\times} = \varphi'|_{\C^\times}$, 
then clearly, $\varphi(j) \varphi'(j)^{-1}$ commutes with $\varphi(z), z \in \C^\times$, i.e., $\varphi(j) \varphi'(j)^{-1}$
belongs to $\wM_\varphi$. Thus the inner-conjugation action of $\varphi(j)$ and of $\varphi'(j)$ on $\wM_\varphi$ are the same as 
elements of $\Aut(\wM_\varphi)/{\rm inn}(\wM_\varphi)$, and this is all that went into defining $T_\varphi$ and $M_\varphi$. 
\end{proof}

\begin{remark}The parameter $\varphi$ in Lemma \ref{lemma7} gives rise to an $L$-packet of representations of $G(\R)$ which are obtained as Langlands quotients 
of certain `standard' representations induced from discrete series representations of a Levi subgroup of a parabolic in $G$. The torus
$T_\varphi$ of the lemma sits inside the Levi subgroup, and is compact modulo the center of the Levi subgroup.
\end{remark}

With all these preliminaries out of our way, we now construct a pure innerform of $G$ given a Langlands parameter
for $G^{\out}(\R)$.
 
Recall that by Lemma \ref{fibers}, the set of possible lifts of the parameter of a representation of $G(\C)$ (if non-empty) 
is parametrized by $H^1(\langle \varphi(j) \rangle, \wM_\varphi(\C))$ 
where $\wM_\varphi$ is as in Lemma \ref{lemma7}, 
i.e., $\wM_\varphi=Z_{\widehat{G}}(\varphi(\C^\times))$. By theorem 3 due to Adams and Borovoi, and Lemma \ref{lemma6},
\begin{eqnarray*}
 H^1(\langle \varphi(j) \rangle, \wM_\varphi(\C)) & \cong & H^1(\langle \varphi(j) \rangle, \wT_\varphi(\C))/W^\varphi\\
& \cong &   H^1(\Gal(\C/\R), T^{\out}_{\varphi}(\C))^\vee/W^\varphi,
\end{eqnarray*}
where, by Lemma \ref{lemma-torus}, 
we are assuming $T^{\out}_\varphi$ to be a torus in $G$, since $T_\varphi$ is contained in $G^{\out}$.
Now we have the natural isomorphism,
$$H^1(\Gal(\C/\R), 
T^{\out}_{\varphi}(\C))^\vee/W^\varphi \cong H^1(\Gal(\C/\R), T^{\out}_{\varphi}(\C))/W^\varphi,$$
which is a reflection of the fact that for an involution $\varphi$ of $X_\star(T)$, 
$H^2(\langle \varphi \rangle, X_\star(T))$ is isomorphic to its dual as $W^\varphi$-modules.
 
Thus given a Langlands parameter $\varphi: W_\R\rightarrow \LG^{\out}$, identified to a cohomology class in
$H^1(\langle \varphi(j) \rangle, \wM_\varphi(\C))$, following the various isomorphisms above, we 
have constructed an element in $H^1(\Gal(\C/\R), T^{\out}_{\varphi}(\C))/W^\varphi$ for $T^{\out}_\varphi$ a maximal torus in $G$, hence an element of
$H^1(\Gal(\C/\R), G(\C))$, i.e a pure innerform of $G$.

\begin{remark}
This section has been long and tedious. In this remark, I 
try to give another point of view to constructing an innerform (but not pure innerform) 
of $G(\R)$ given a Langlands parameter $\varphi$ for
$G^{\out}(\R)$, i.e., $\varphi: W_\R\rightarrow \LG^{\out}$. 
We assume for this construction that for the element $-1 \in \C^\times \subset W_\R$, $\varphi(-1) \in Z(\wG)$. Therefore $\varphi(j)^2 = \varphi(-1) $ belongs to the the center of $\wG$, in fact to the $W_\R$-invariant part of the center of $\wG(\C)$. Thus conjugation by $\varphi(j)$ gives an involution 
on $\wG$. It is easy to see that the set of involutions up to conjugation on $G(\C)$ and $\wG(\C)$ are 
canonically isomorphic as sets.  Thus the involution afforded on $\wG(\C)$ by conjugation by $\varphi(j)$ can also be 
treated as (a conjugacy class of) involution on $G(\C)$. But as discussed in the beginning of the section, an involution 
on $G(\C)$ gives rise to a real structure on $G(\C)$, which by Remark \ref{15} is exactly an innerform of $G^{\out}$.  
\end{remark}

\vspace{3mm}

\noindent{\bf Example:} For the group $G^{\out}= \GL_{n}(\R)$, 
let a Langlands parameter for $\GL_{n}(\C)$ consist 
of the trivial mapping from $W_\C = \C^\times$ to $\GL_n(\C)$. The possible lifts of this parameter to  
one of $G^{\out}= \GL_{n}(\R)$ is an $n$-tuple of characters of $\R^\times$ which are either trivial or $\omega_{\C/\R}$,
say $p$ many trivial characters, and $q$ many   $\omega_{\C/\R}$ for $p+q =n$. Then $\varphi(j)$ is the $n \times n$ matrix:
$$\varphi(j) = \left ( \begin{array}{ccccc} 
  1 & & & &  \\
& 1 & & & \\
& & * & & \\
& & & -1& \\
& & & & -1 
\end{array}
\right ) \in \GL_n(\C),$$
with $p$ many 1 on the diagonal, and $q$ many $-1$ on the diagonal. The inner conjugation by $\varphi(j)$ 
on $\GL_n(\C)$  is associated to the real form of $\GL_n(\C)$ which is $\U(p,q)$, or to 
elaborate, the centralizer of $\varphi(j)$ in $\GL_n(\C)$ which is $\GL_p(\C) \times \GL_q(\C)$ is the complexification
of the maximal compact subgroup $\U(p) \times \U(q)$ of $\U(p,q)$.

\section{The conjecture}
We continue with  $E/F$ a  quadratic extension of local fields, and $G$ a quasi-split reductive algebraic 
group over $F$. Let $G_\alpha$ be the collection of pure innerforms of $G$ parametrized by those elements 
of the cohomology $H^1(\Gal(\bar{F}/F), G)$ which become trivial on restriction to $\Gal(\bar{F}/E)$, and as a result,
 $G_\alpha(E)= G(E)$.

We  fix a
Borel subgroup $B$ of $G$ with a Levi decomposition $B=TN$, and fix a   Whittaker datum:
$$\psi: N(E)/N(F)\rightarrow \C^\times.$$

Observe that since the groups $G$ and $G^{\out}$ become isomorphic over $E$, 
$$\Hom(W_E, {}^{L}G^{\out}) = \Hom(W_E, {}^L{G}
).$$
Let $Z(\sigma_\pi)$ be the centralizer of  the parameter 
$\sigma_\pi: W_E\rightarrow {}^L{G} = {}^L{G}^{\out},$  and $Z(\widetilde{\sigma_\pi})$ 
be the centralizer of  the parameter 
 $\widetilde{\sigma_\pi}: 
W_F\rightarrow {}^L{G} = {}^L{G}^{\out}$ which extends $\sigma_\pi$. Clearly, 
$Z(\widetilde{\sigma_\pi}) \subset Z({\sigma_\pi})$. 
Given the extension $\widetilde{\sigma_\pi}$ of $\sigma_\pi$ to $W_F$, there is a natural action of ${\rm Gal}(E/F)$ on 
$Z({\sigma_\pi})$ with $Z(\sigma_\pi)^{{\rm Gal}(E/F)}= Z(\widetilde{\sigma_\pi})$.  
The following lemma compares the image of 
$\pi_0(Z(\sigma_\pi)^{\Gal(E/F)})$ inside 
$\pi_0(Z(\sigma_\pi))$ with the group 
$\pi_0(Z(\sigma_\pi))^{\Gal(E/F)}$, or rather it says that they are almost not related.

\begin{lemma}Let $X$ be a topological space together with the action of a finite group $F$ on it. Denote by $\pi_0(X)$ 
the set of connected components of $X$. Then $F$ operates on $\pi_0(X)$, and there is a natural map from 
$\pi_0(X^F)$ to $\pi_0(X)$ which is neither injective nor surjective, but which lands inside $\pi_0(X)^F$.
\end{lemma}  
\begin{proof} All that there is to the lemma is the assertion that
 if $F$ does not fix a particular component, then it cannot have a fixed point inside that component. 
On the other hand, even if it fixes a particular component, it may not have a fixed point in that component. 
\end{proof}

\begin{remark}If $Z({\sigma_\pi})$ is finite and abelian, then the action of $\Gal(E/F)$
on $Z(\sigma_\pi)$   for any extension  $\widetilde{\sigma_\pi}$ 
of  ${\sigma_\pi}$ is the same, and further, 
we have the tautological assertion that 
$$\pi_0(Z(\widetilde{\sigma_\pi})) = \pi_0(Z(\sigma_\pi)^{\Gal(E/F)}) = \pi_0(Z(\sigma_\pi))^{\Gal(E/F)}.$$ 
This applies in particular
for discrete series representations of $\SL_2(E)$.
\end{remark}

Now we note the maps between affine algebraic varieties---the parameter spaces of Langlands parameters---which appear 
in the following commutative diagram:

 $$
\begin{CD}
X_1 = \Hom(W'_F, {}^LG^{\out}(\C)  )
@>\Phi'>>  X_2 = \Hom(W'_E, {}^LG^{\out}(\C))  \\ 
 @ V{\pi_1}VV  @VV{\pi_2}V  \\
\Sigma_F(G^{\out})= X_1/\!/\widehat{G}(\C)
@> \Phi >>  \Sigma_E(G)=X_2/\!/\widehat{G}(\C).
\end{CD}
$$

The map $\Phi$ induces a map from stabilizer in $\widehat{G}^{\out}(\C)
$ of a point in $X_1$ 
to the stabilizer in $\widehat{G}^{\out}(\C)$ of the 
corresponding point in $X_2$, hence also a map on  the group of connected components of the stabilizers. 
Note that the group $ H^1(\Gal(E/F), Z( \widehat{G}^{\out} )^{W_E})$ 
operates on $\Sigma_F(G^{\out})$ by central twists,
and the mapping $\Phi$ is equivariant under it.

Recall that to a finite map between algebraic varieties $f: X\rightarrow Y$, one can define the degree 
of $f$ at a point $x\in X$ by $\deg(f)(x) = \dim [{\mathcal O}_{X,x}/f^{\star}({\mathfrak m}_{Y,f(y)})]$, where we 
do not explain the standard notation that we have used here.

\begin{conj} \label{conj3} Let $G$ be a quasi-split group over a local field $F$, $E/F$ a separable quadratic extension, 
and 
$\pi$ an irreducible admissible representation of $G(E)$. Then if $\pi$ is 
$\omega_{G_\alpha}$-distinguished 
by $G_\alpha(F)$ for a pure innerform $G_\alpha$ of $G$ over $F$ defined by an element $\alpha$ 
of $H^1(F,G)$ which becomes trivial 
when restricted to $H^1(E,G)$, 
we must have,

\begin{enumerate}

\item $\left \{ \pi^\vee \right \}   = \left \{ \pi^\sigma \right \},$ an equality of $L$-packets.

\item The $L$-packet $\{\pi\}$ on $G(E)$ arises from basechange of an $L$-packet on $G^{\out}$, i.e., the
Langlands parameter  $\sigma$ of $\pi$ in  $ \Hom(W_E, {}^L{G}) = \Hom(W_E, {}^{L}G^{\out}) $
arises as the restriction of a parameter $\tilde{\sigma}$ in $\Hom(W_F, {}^{L}G^{\out})$. 

\end{enumerate}
If the condition in $(2)$ is satisfied,  and if  
$\pi$ has a   Whittaker model for a character
$$\psi: N(E)/N(F)\rightarrow \C^\times,$$
then $\pi$ is $\omega_G$-distinguished.
  
Let $F(\sigma) = \{ \tilde{\sigma} \in  \Hom(W_F, {}^{L}G^{\out}) | \tilde{\sigma}|_{W_E} = \sigma\}$.

We take up the non-Archimedean case first.
In this case,
Each orbit $O(\tilde{\sigma})$ of $ H^1(\Gal(E/F), Z( \widehat{G}^{\out} )^{W_E})$ action on $F(\sigma)$ 
is associated to a subgroup $ A_G(\tilde{\sigma}, E)\subset H^1(\Gal(E/F), G(E))$, defining a set of 
pure innerforms $G_\alpha$ of $G$ over $F$.  Each such orbit in $F(\sigma)$ contributes a certain space of invariant forms 
$I_\alpha(O(\tilde{\sigma}))$ in 
$ \Hom_{G_{\alpha}(F)}[\pi, \omega_{G_\alpha}] $ for exactly these pure innerforms $G_\alpha$,  
$ \alpha \in A_G(\tilde{\sigma}, E)\subset H^1(\Gal(E/F), G(E))$, which for 
$$d_0(\tilde{\sigma})= \left |{\rm coker} \{\pi_0(Z(\tilde{\sigma})) \rightarrow 
\pi_0(Z(\sigma))^{\Gal{(E/F)}} \} \right|, $$ 
has dimension given by
$$\dim I_\alpha(O(\tilde{\sigma})) = 
(\deg \Phi)(\tilde{\sigma})
/d_0(\tilde{\sigma}),$$  if the character of the component group $\mu_\pi$ associated to $\pi$ 
restricted to the image of 
$\pi_0(Z(\tilde{\sigma}))$ inside $\pi_0(Z(\sigma))$ contains the  trivial representation and $I_\alpha(O(\tilde{\sigma})) = 0$ otherwise.
The space $ \Hom_{G_{\alpha}(F)}[\pi, \omega_G] $ is built from 
 $I_\alpha(O(\tilde{\sigma}))$  which are linearly independent for different orbits $O(\tilde{\sigma})$ in  $F(\sigma)$.

In case $F=\R$, each element $\tilde{\sigma} \in F(\sigma)$  
associates a pure innerform $G_{\tilde{\sigma}}(\R)$ of $G(\R)$ 
by the recipe given in section 13 such that if the condition on the component group from the previous paragraph is 
satisfied by $\pi$, then $ \Hom_{G_{\tilde{\sigma}}(\R)}[\pi, \omega_G] $ is built from linearly independent spaces of 
dimension $(\deg \Phi)(\tilde{\sigma})$ as $\tilde{\sigma}$  runs over all elements in $ F(\sigma)$.
\end{conj}

\begin{remark} Suppose $\pi$ has a Whittaker model for a character of $N(E)/N(F)$, and is a discrete series representation of $G(E)$, equivalently in terms of parameters, if 
for the parameter $\sigma$ of $\pi$, $Z(\sigma)/Z({}^LG)$ is finite,
then  for each lift $\tilde{\sigma}$ of $\sigma$, $(\deg \Phi)(\tilde{\sigma}) = d_0(\tilde{\sigma}) =1$, 
thus in this case, 
$\sum_\alpha \dim\Hom_{G_\alpha}[\pi, \omega_{G_\alpha}]$ 
is equal to the number of lifts of $\sigma$ to $G^{\out}$.
If $G$ is quasi-split over $F$, then 
$ \dim\Hom_{G(F)}[\pi, \omega_{G}]$ is the number of orbits on the set of lifts $\tilde{\sigma} \in H^1(W_F, Z(\wG^{\out}))$ 
of $\sigma$ under twists by  $\chi \in H^1(W_F, Z(\wG^{\out}))$ with $\chi|_{W_E}=1$. Each orbit
under the character twists contributes one dimensional space of invariant forms 
to exactly as many pure innerforms of $G$ as cardinality of the orbit as described in Proposition \ref{prop4}.
Note that the subtleties of pure innerforms go away if $G$ is semisimple simply connected group over a non-Archimedean field,
when we would simply assert that $\dim\Hom_{G}[\pi, \omega_{G}]$ 
is equal to the number of lifts of $\sigma$ to $G^{\out}$ if $\pi$ is a discrete series representation. 
\end{remark}

\begin{remark}
 The simplest example to bring out the role of the character $\omega_G$ is the Steinberg representation 
$\St_2$ of $\PGL_2(E)$. It can be seen that $\St_2$ is not distinguished by $\PGL_2(F)$, and is distinguished 
exactly by the character $\omega_{E/F}$ of $F^\times/F^{\times 2}$ treated as a character of $\PGL_2(F)$. We will give 
more examples  in the case of real groups in section \ref{17}.
\end{remark}

\begin{remark} The mysterious role of the character $\omega_G$ in the conjecture appears to be related 
to the symmetry of the bilinear form $B: \pi \times \pi^\sigma \rightarrow \C$. (By Schur's lemma, 
$B(u,v) = \epsilon_\pi B(v,u)$ for $\epsilon_\pi = \pm 1$ which is the sign we are talking about.) Indeed
the character $\omega_G$ is constructed via an element of the center of $\widehat{G}$ which determines
symmetry property of selfdual representations of $\widehat{G}$.

\end{remark}
  
\begin{remark}Under some conditions on $X$ and $Y$, for instance by exercise III.10.9, page 276 of R. Hartshorne's book on 
Algebraic Geometry, if $X$ is Cohen-Macaulay, $Y$ is regular, and $f: X\rightarrow Y$ finite map of irreducible 
algebraic varieties, then 
$f$ is flat, and as a consequence, the sum of  degrees at fibers of $f$ is constant over $Y$. However, 
this is not assured in general, for example the map $p:\C^2\rightarrow \C^2/j$ where $j$ is the
involution $(z_1,z_2)\rightarrow (-z_1,-z_2)$ has fiber of multiplicity 3 at origin, but 2 at other points.
The author thanks Nitin Nitsure for all this. 
In our case, the spaces involved are non-singular by Proposition \ref{prop44}, but the maps that we consider are finite only onto their image which may be a singular space, so I 
am not sure that the sum of fiber degrees of the map in our case is constant from a connected component of $\Sigma_F(G^{\out})$ 
to its image 
in a connected component of $\Sigma_E(G)$ (of course a point in a connected component of $\Sigma_E(G)$ 
arises as the image of points in many different components of $\Sigma_F(G^{\out})$),  and if not constant, if 
we need to make modifications in this paper.
\end{remark}

\section{Tori}
Many features of the conjectured multiplicity formula can already be seen for 
tori which we discuss in some detail here for a torus $T$ over a local field
$F$. 

Observe that the norm mapping  $N_{E/F}: T(E)\rightarrow T(F)$ need not be surjective. 
As a result, for a character $\chi: T(E) \rightarrow \C^\times$ to be distinguished by $T(F)$,
it is necessary, but not sufficient, that $\chi^\sigma = \chi^{-1}$. Our conjecture says that
a character $\chi: T(E) \rightarrow \C^\times$ is distinguished by $T(F)$ if and only if it arises
as a basechange of a character on an another torus, denoted $T^{\out}$.

For a torus $T$ over $F$, as mentioned earlier, $T^{\out}$ is the torus which 
sits in the exact sequence of algebraic groups,
$$1 \rightarrow T \rightarrow R_{E/F}(T) \rightarrow T^{\out} \rightarrow 1.$$
Thus, at the level of $F$-rational points (using  Shapiro's lemma), we have,
$$1 \rightarrow T(F) \rightarrow T(E) \rightarrow T^{\out}(F) 
\rightarrow H^1( {\rm Gal}(\bar{F}/F), T) \rightarrow H^1( {\rm Gal}(\bar{F}/E), T)
\rightarrow \cdots$$
This long exact sequence can be re-written as:
$$1 \rightarrow T(F) \backslash T(E) \rightarrow T^{\out}(F) \rightarrow {\rm Ker}^1(E/F, T)
\rightarrow 0, $$
 where ${\rm Ker}^1(E/F, T) 
= {\rm Kernel}\{ H^1( {\rm Gal}(\bar{F}/F), T) \rightarrow H^1( {\rm Gal}(\bar{F}/E), T) \}.$

It follows that the characters of $T(E)$ which are trivial on $T(F)$ arise from restriction
of characters of $T^{\out}(F)$ which in this case corresponds to basechange of characters of $T^{\out}(F)$ to
characters of $T(E)$ (the {\it norm} map from $T(E)$ to $T^{\it out}(F)$ is $t\rightarrow \sigma(t)/t$); 
further, the character $\chi$ of $T(E)$ arises as basechange of as many representations
of $T^{\out}(F)$ as the order of ${\rm Ker}^1(E/F, T) $ which is the number of pure innerforms
of $T$ over $F$ whose restriction to $E$ gives a fixed pure innerform of $T$ over $E$, i.e., $R_{E/F}(T)$. This is as predicted
by the conjecture since in this case $T_\alpha(F) \subset T_\alpha(E)$ are all the same as $T(F) \subset T(E)$,
so each one contributes exactly once to $\Hom_{T_\alpha(F)}[\chi,\C]$.

We next check the condition on the characters of the component groups imposed in our conjecture for the case of tori.
We will see that the group of connected components of the centralizer in $\widehat{T}$ of any parameter for a torus $T$ over $F$
is isomorphic to  $H^1(F,T)^\vee = \Hom[H^1(F,T), \Q/\Z]$. The mapping $R_{E/F}(T) \rightarrow T^{\out}$ induces a 
mapping of the dual groups, and hence a mapping of the group of connected components which in this case will be the natural map
$H^1(F,T^{\out})^\vee \rightarrow H^1(E, T)^\vee$. We thus want to understand those characters of $H^1(E,T)^\vee$ 
which are trivial on
the image of $H^1(F,T^{\out})^\vee$.

{\bf Claim:} Identifying the characters of $H^1(E,T)^\vee$ to $H^1(E,T)$, 
the characters of $H^1(E,T)^\vee$ 
which are trivial on
the image of $H^1(F,T^{\out})^\vee$ are exactly those elements of $H^1(E,T)$ which are in the image of $H^1(F,T)$.

We do this in the following paragraphs. 

Recall that the Langlands correspondence for tori (due to Langlands) gives a bijective correspondence
between characters of a torus $T(F)$, and  admissible homomorphisms in $\Hom(W_F, {}^L{T}(\C))$. 
It can be seen that the group of connected components of the centralizer  
of (any) admissible homomorphism  in  $\Hom(W_F, {}^L{T}(\C))$ 
is $\pi_0(\widehat{T}(\C)^{W_F})
$. By the long exact sequence for cohomologies associated to the short 
exact sequence
of $W_F$-modules:
$$0 \rightarrow X^\star(T) \rightarrow X^\star(T)\otimes \C \rightarrow \widehat{T}(\C) \rightarrow 0,$$
it follows that  $ \pi_0(\widehat{T}(\C)^{W_F}) \cong H^1(W_F, X^\star(T))$. 

By the Tate duality, there is a perfect pairing
$$\begin{CD}H^1(F, X^\star(T))  \times H^1(F, T) @>>>  H^2(F, \Gm) = \Q/\Z, \end{CD}$$ 
and as a result, elements of $H^1(F,T)$ are identified
to characters of $H^1(F,X^\star(T))$. This gives the refined local Langlands correspondence in the sense of
Vogan for tori: The character group 
of the group of connected components of the centralizer  
of any element in  $\Hom(W_F, {}^L{T}(\C))$ is in bijective correspondence with elements $\alpha$ of 
$H^1(F,T)$ which parametrizes pure innerforms of $T$, thus a Langlands parameter
for $T$ together with a character of the component group of the parameter are in bijective correspondence
with characters of $T_\alpha(F)=T(F)$ as $\alpha$ runs over elements of $H^1(F,T)$.

For a homomorphism of tori $f: T_1\rightarrow T_2$, we have induced maps on Galois cohomology groups,
$$f_\star: H^1(F, T_1) \rightarrow H^1(F,T_2), {\rm ~and~~} f^\star: H^1(F, X^\star(T_2)) \rightarrow H^1(F,X^\star(T_1)),$$
which make up the following diagram of maps
 $$
\begin{CD}
H^1(F, X^\star(T_1)) @. \times @. H^1(F, T_1) @>>> H^2(F, \Gm) = \Q/\Z \\ 
 @A{f^*}AA @. @VV{f_\star}V @| \\
H^1(F, X^\star(T_2)) @. \times @. H^1(F, T_2) @>>>  H^2(F, \Gm) = \Q/\Z, 
\end{CD}
$$
making $f^\star$ the adjoint of $f_\star$:
$$\langle f_\star a, b \rangle = \langle a, f^\star b \rangle,$$
where $a\in H^1(F,T_1)$ and $b \in H^1(F, X^\star(T_2))$.

We will apply this adjoint relationship for $T_1 =R_{E/F} T$, a torus over $F$, and $T_2=T^{\out}$ 
for $E/F$ a quadratic extension, and 
$f: R_{E/F} T\rightarrow T^{\out}$ the natural map, giving us:
 $$
\begin{CD}
H^1(F, X^\star(R_{E/F}T))
 @. \times @. H^1(F, R_{E/F}T) @>>> H^2(F, \Gm) = \Q/\Z \\ 
 @A{f^*}AA @. @VV{f_\star}V @| \\
H^1(F, X^\star(T^{\out})) @. \times @. H^1(F, T^{\out}) @>>>  H^2(F, \Gm) = \Q/\Z.
\end{CD}
$$
making $f^\star$ the adjoint of $f_\star$:
$$\langle f_\star a, b \rangle = \langle a, f^\star b \rangle,$$
where $a\in H^1(F,R_{E/F}T)$ and $b \in H^1(F, X^\star(T^{\out}))$. By the non-degeneracy of the pairing $ \langle-, -\rangle$, it follows that the elements of $H^1(E,T) = H^1(F, R_{E/F}(T))$ 
considered as characters on $H^1(E, X^\star(T)) = H^1(F, X^\star(R_{E/F}T))$, which are zero on the image 
of $H^1(F, X^\star(T^{\out}))$ are exactly those whose image in $H^1(F,T^{\out})$ is zero. 

From the exactness
of the sequence,
$$\cdots 
\rightarrow H^1( F, T) \rightarrow  H^1( F, R_{E/F}T) 
\stackrel{f_\star}\rightarrow H^1( F, T^{\out}) \rightarrow  \cdots$$
it follows that ${\rm Ker}(f_\star): H^1( F, R_{E/F}T) \rightarrow H^1(F, T^{\out})$
is the image of the map:   $H^1( F, T) \stackrel{f_\star}\rightarrow  H^1( F, R_{E/F}T)$, proving our 
claim in the beginning that the restriction imposed on the character of component groups forces 
the pure innerforms of $R_{E/F}T$ to come from pure innerforms of $T$.

\section{$\SL(2)$}

In the  case of $(\SL_2(E),\SL_2(F))$, the multiplicity of the space of 
$\SL_2(F)$-invariant linear forms on a representation of $\SL_2(E)$  was studied in [AP03] in detail,
and it was found that  $\dim \Hom_{\SL_2(F)}[\pi, \C]$, as $\pi$ runs over an $L$-packet of 
representations of $\SL_2(E)$, is either $d_{\pi}$ or 0 for an integer $d_\pi$ which depends only on the $L$-packet of $\pi$, and that it is 
nonzero for a particular $\pi$ if and only if

\begin{enumerate}
\item $\{\pi^\sigma\} = \{ \pi \}$ in the sense of $L$-packets. (Recall that for $\GL_2(E)$, $\pi^\vee \cong \pi \otimes 
(\det \pi)^{-1}$, hence for $L$-packets of $\SL_2(E)$,
$\{\pi\} = \{\pi^\vee\}$.)

\item The representation $\pi$ has a Whittaker model for a character of $N(E)=E$ which is trivial on $F$.
\end{enumerate}

The following lemma was proved in [AP03] using explicit realization of an $\GL_2(F)$-invariant linear form
in the Kirillov model of a representation $\pi$ of $\GL_2(E)$  due to Jeff Hakim. We offer a `pure thought' argument here.
Unfortunately, we have not succeeded in proving an analogous result for higher rank groups, not even for $\SL_n(E)$.

\begin{lemma} \label{AP}
Let $\pi$ be an irreducible admissible infinite dimensional representation of $\SL_2(E)$. Then if $\pi$ is distinguished by
$\SL_2(F)$, then $\pi$ must have a Whittaker model for a character $\psi: E/F \rightarrow \C^\times$.
\end{lemma}

\begin{proof} Since $\pi$ is distinguished by $\SL_2(F)$, the largest quotient of $\pi$ on which $\SL_2(F)$ operates trivially
is nonzero. As a consequence, the largest quotient $\pi_{F}$ of $\pi$ on which $N(F) = F$ operates trivially is nonzero. 
Clearly $\pi_{F}$ is a smooth module for $N(E)/N(F) = E/F$. Thus there are two options:

\begin{enumerate}
\item  $N(E)/N(F)$ does not operate trivially on $\pi_F$, in which case it is easy to prove that
for some nontrivial character $\psi: N(E)/N(F) \rightarrow \C^\times$, $\pi_{\psi} \not = 0$; cf. Lemma 11.1 in [AP]

\item   $N(E)/N(F)$ operates trivially on $\pi_F$, in which case in particular
$N(E)$ will operate trivially on the linear form $\ell: \pi \rightarrow \C$
which is $\SL_2(F)$-invariant. Thus this linear form will be invariant under
$\SL_2(F)$ as well as $N(E)$, and therefore the group generated by $\SL_2(F)$ and $N(E)$. 
It is easy to see that the group generated by $\SL_2(F)$ and $N(E)$ is $\SL_2(E)$. 
Thus $\ell: \pi \rightarrow \C$ is invariant under $\SL_2(E)$, a contradiction.

\end{enumerate}

This completes the proof of the lemma. \end{proof} 
We use this lemma to prove the following crucial lemma which is at the basis of the condition on the character
of the component groups in Conjecture \ref{conj3}.

\begin{lemma} \label{28}
Let $\pi$ be an irreducible admissible infinite dimensional representation of $\SL_2(E)$ which is distinguished 
by $\SL_2(F)$ and has  
 a Whittaker model for a character $\psi: E/F \rightarrow \C^\times$. Use this to parametrize representations
in the $L$-packet $\{\pi\}$ containing $\pi$ by $E^\times/G(\pi)$ where $G(\pi)$ is the subgroup of
$E^\times$ containing $E^{\times 2}$ such that $\lambda \in G(\pi)$ if and only if $\pi^\lambda \cong \pi$.
Suppose that the $L$-packet $\{\pi\}$ of $\SL_2(E)$ arises as base change of an $L$-packet $\{\pi_0\}$ 
of $\SL_2(F)$, and the $L$-packet $\{\pi_0\}$ is parametrized similarly by $F^\times/G(\pi_0)$.
Then a representation $\pi^\mu \in \{\pi\}$ for $\mu \in E^\times/G(\pi)$ 
is distinguished by $\SL_2(F)$ if and only if $\mu$ belongs to the kernel of the map    
$\Nm: E^\times/G(\pi) \rightarrow F^\times/G(\pi_0)$,  
a condition equivalent to the assertion
that for the dual map 
$\Nm^\vee: \left[F^\times/G(\pi_0)\right]^\vee \rightarrow \left [ E^\times/G(\pi) \right ]^\vee$,  
which is the map on component groups appearing in Conjecture \ref{conj3},  the character of 
$\left [ E^\times/G(\pi) \right ]^\vee$ given by $\mu \in  E^\times/G(\pi) $, is trivial on the
image of $\Nm^\vee$.
\end{lemma}
\begin{proof}

By Lemma \ref{AP}, a representation $\pi^\mu \in \{\pi\}$ for $\mu \in E^\times/G(\pi)$ 
is distinguished by $\SL_2(F)$ if and only if $\mu$ belongs to $F^\times$, or more precisely, 
in $[F^\times G(\pi)]/G(\pi) \subset E^\times/G(\pi)$. 
Thus we need to prove that $[F^\times G(\pi)]/G(\pi) \subset E^\times/G(\pi)$ is exactly 
the kernel of the map    
$\Nm: E^\times/G(\pi) \rightarrow F^\times/G(\pi_0)$.

We will find it convenient to use an equivalent interpretation of $G(\pi)$ as the common 
kernel (in $E^\times$) of possible selftwists $\tilde{\pi} \otimes \chi \cong \tilde{\pi}$ for 
an irreducible admissible
representation $\tilde{\pi}$ of $\GL_2(E)$ containing $\pi$; similarly $G(\pi_0)$ using an irreducible admissible
representation $\tilde{\pi}_0$ of $\GL_2(F)$ containing $\pi_0$.

Clearly if a character $\chi_0$ of $F^\times$ is a selftwist of $\tilde{\pi}_0$, the character $\chi \circ \Nm$ of $E^\times$
is a selftwist of $\tilde{\pi}$, which simply goes to say that the norm mapping 
$\Nm: E^\times/G(\pi) \rightarrow F^\times/G(\pi_0)$ is well-defined. Clearly, it is trivial restricted to $F^\times G(\pi)$,
and the lemma asserts that it is the precise kernel, i.e., the induced mapping
$\Nm: E^\times/[F^\times G(\pi)] 
\rightarrow F^\times/G(\pi_0)$ 
is injective, or equivalently, the mapping
on the character groups 
$\Nm^\vee:   \left [ F^\times/G(\pi_0) \right ]^\vee 
\rightarrow \left [ E^\times/[F^\times G(\pi)] \right]^\vee$ is surjective. 

Since $F^\times E^{\times 2} $ contains $E^1$,  $E^\times/[F^\times G(\pi)]$ is isomorphic to 
$ \Nm E^\times/[F^{\times 2}\Nm G(\pi)]$, a group of exponent 2, which is a quotient of 
$\Nm E^\times/F^{\times 2}$. Thus characters of 
$E^\times/[F^\times G(\pi)] $ 
are 
selftwists of $\tilde{\pi}$ which arise from a quadratic character 
of $F^\times$ through the norm mapping. 

It suffices then to prove the following claim.

{\bf Claim:} If $\tilde{\pi}$ has a selftwist by $\chi_0 \circ \Nm$ for $\chi_0$ a character
of $F^\times$, then $\tilde{\pi_0}$ has selftwist either by $\chi_0$ or $\chi_0 \omega_{E/F}$.

The proof of this claim is easily made by considering the principal series and discrete series separately. 
\end{proof}

\begin{remark}
Even for $\SL_2$, there is no obvious relationship between component groups of a representation $\pi_0$ of $\SL_2(F)$,
and its basechange to $\SL_2(E)$, except that there is a map between them (which  amounts to saying
that given a selftwist $\chi:F^\times \rightarrow \C^\times$ for a representation of $\GL_2(F)$, $\chi \circ \Nm: E^\times 
\rightarrow \C^\times$ is a selftwist for the basechanged representation). Thus the above claim, which is the crux of 
the proof  of Lemma \ref{28}, brings some order to otherwise unrelated objects.
\end{remark}

We next turn our attention to the question about multiplicities $\dim \Hom_{\SL_2(F)}[\pi, \C]$ for
a representation $\pi$ of $\SL_2(E)$. The following lemma combined with theorem 1.4 of [AP03] 
proves that $\dim \Hom_{\SL_2(F)}[\pi, \C] = m(\pi), $ or 0, where $m(\pi)$ is the number of distinct ways of lifting
the Langlands parameter of $\pi$ from $W_E$ to $W_F$.

\begin{lemma} \label{lemma8}
Given a Langlands parameter $\phi: W_E\longrightarrow \PGL_2(\C)$, corresponding to a representation $\pi$ of $\GL_2(E)$,
its extensions to $W_F$ in the following diagram
$$
\xymatrix{
W_E \ar[r]  \ar@{_(->}[d]  & \PGL_2(\C) \\
W_F \ar@{-->}[ur] & {}  
}
$$
has order given by,
  $$m(\pi)= \frac{X_{{\pi}}}{Z_{{\pi}}/Y_{{\pi}}},$$ where
\begin{enumerate}
\item
$X_\pi=\left\{\chi \in \widehat{F^\times} | 
\begin{tabular}{c}
\mbox{ $\pi$ {is} $\chi$-distinguished}\\
\mbox{ with respect to $\GL_2(F)$}
\end{tabular} 
\right\}.$
\item
$Y_\pi=\{\mu \in \widehat{E^\times} \mid
\pi \otimes \mu \cong \pi; \mu|_{_{F^\times}}=1 \}$.
\item
$Z_\pi=\{\mu \in \widehat{E^\times} \mid
\pi \otimes \mu \cong \pi \}$.
\end{enumerate}

\end{lemma}
\begin{proof}: Consider the group,
\[A_\pi= \frac{\left\{\chi: E^\times \rightarrow \C^\times 
 \mid (\pi \otimes \chi)^\sigma \cong 
\pi \otimes \chi\right \}}{\left \{\chi: E^\times \rightarrow \C^\times  \mid \chi=\chi^\sigma\right \}}.\]
For $\chi \in A_\pi$, let $\pi_{\chi^{-1}\chi^\sigma}$ denote the class of pair of representations 
$\{\pi^\prime,\pi^\prime \otimes \omega_{_{E/F}}\}$ of
$\GL_2(F)$ such that $\pi \otimes \chi = {\rm BC} (\pi^\prime)$. 
Note that $Z_\pi/Y_\pi$ acts freely on
$\{\pi_{\chi^{-1}\chi^\sigma} \mid \chi \in A_\pi \}$. The possible extensions of the parameter 
 $\phi: W_E \rightarrow \PGL_2(\C) $ to $W_F$ is given by orbits under the above action:
\[\frac{\{\pi_{\chi^{-1}\chi^\sigma} \mid \chi \in A_\pi \}}{Z_\pi/Y_\pi}.\]
We note that $A_\pi$ is in bijection with $Y_\pi$ under the map $\chi \mapsto \chi^{-1}\chi^\sigma$, 
and that if $\pi$ is in the discrete series, then 

\begin{enumerate}
\item The set on the numerator above is in bijection with $A_\pi$, and 
\item $Y_\pi$ is in bijection with $X_\pi$. 
\end{enumerate}

This proves the lemma if $\pi$ is a 
discrete series representation. 
For principal series representations, both $(1)$ and $(2)$ are wrong in general,
thus the lemma is more subtle; we verify  it in a case-by-case check below.
\end{proof}

\vspace{4mm}

\noindent{\bf Examples :} We illustrate the multiplicity formula with examples
of principal series representations of $\SL_2(E)$ taken from [AP03]. In what follows, we introduce
the notation $\pi_1 \sim \pi_2$ for two irreducible representations of $\GL_2(E)$ (or $\GL_2(F)$) which are twists
of each other by a character; thus on restriction to $\SL_2(E)$ (or $\SL_2(F)$), they give rise to the same $L$-packet of 
representations.

Let $V$ be an irreducible admissible representation
of $\SL_2(E)$ that occurs in the restriction of a principal series
representation $\pi = {\rm Ps}(\chi_1,\chi_2)$ of $\GL_2(E)$. 
Suppose that $V$ is distinguished with respect to $\SL_2(F)$
(and therefore by Proposition 2.3 of [AP03], $\chi_1\chi_2^{-1}|_{_{F^\times}}=1$ or $\chi_1\chi_2^{-1}=
(\chi_1\chi_2^{-1})^\sigma$).
Then we have,
\begin{enumerate}
\item
$\dim_{{\Bbb C}}{\rm Hom}_{\SL_2(F)}(V,1)=1$,
if 
$\chi_1\chi_2^{-1}|_{F^\times}=1$, and $\chi_1^2 \not = \chi_2^2$. The $L$-packet containing $V$ has only 1 element.

In this case, since $\chi_1/\chi_2$ is trivial on $F^\times$, 
there is a character $\chi$ of $E^\times$ such that $\chi_1/\chi_2 = \chi/\chi^\sigma$, and so 
$$~~~~~~~~~~\pi = {\rm Ps}(\chi_1,\chi_2) \sim  {\rm Ps}(\chi_1/\chi_2,1) = {\rm Ps}(\chi/\chi^\sigma, 1) \sim 
{\rm Ps}(\chi,\chi^\sigma) = {\rm BC}(Ds(\chi)), $$  comes as the base change 
of a unique representation of $\GL_2(F)$ (up to twists) 
which is the discrete series representation $Ds(\chi)$ corresponding to the character $\chi$ of $E^\times$.
\item 
$\dim_{{\Bbb C}}{\rm Hom}_{\SL_2(F)}(V,1)=1$,
if 
$\chi_1\chi_2^{-1}|_{F^\times}=\omega_{_{E/F}}$, $\chi_1^2=\chi_2^2$, 
$\chi_1 \not = \chi_2$.  The $L$-packet containing $V$ has 2 elements, and both are distinguished by $\SL_2(F)$.

In this case, $$\frac{\chi_1}{\chi_2} = \frac{\chi_2}{\chi_1} = \sigma\left (\frac{\chi_1}{\chi_2} \right ),$$ 
so $\chi_1/\chi_2 
= \mu \circ \N$ for a character $\mu$ of $F^\times$ with $\mu^2 = \omega_{E/F}$. Hence 
the representation $\pi$ (up to twists) is the base change of a unique
principal series representation of $\GL_2(F)$ (up to twists).

We note that the basechange map from 
equivalence classes of parameters $W_F\rightarrow \PGL_2(\C)$ to 
equivalence classes of parameters $W_E\rightarrow \PGL_2(\C)$ is of degree 2 at the representation 
$\mu:W_F\rightarrow \left ( \begin{array}{cc} \mu & 0 \\ 0 & 1\end{array}\right ) \in \PGL_2(\C)$ 
going to $(\chi_1,\chi_2):W_E\rightarrow \left ( \begin{array}{cc} \chi_1 & 0 \\ 0 & \chi_2\end{array}\right ) \in \PGL_2(\C)$,
as follows from the discussion on parameter spaces for $\SL_2(F)$ in section \ref{parameter} (with the connected component of parameter
space passing through the point $\mu$ (resp $(\chi_1,\chi_2)$) being $\C^\times$ (resp. $\C$), and the basechange 
mapping being $z\rightarrow z+z^{-1}$ being considered at the point $1$ or $-1$ in $\C^\times$ which is clearly of degree 2 
at its points of ramification.
 
\item
$\dim_{{\Bbb C}}{\rm Hom}_{\SL_2(F)}(V,1)=2$,
if either
$\chi_1\chi_2^{-1} =(\chi_1\chi_2^{-1})^\sigma = \mu \circ \N$ and $\chi_1^2 \neq \chi_2^2$,
or $\chi_1=\chi_2$. The $L$-packet containing $V$ has only 1 element.

In this case, 
$$\pi = {\rm Ps}(\chi_1,\chi_2) \sim  {\rm Ps}(\chi_1/\chi_2,1) = {\rm Ps}(\mu\circ \N,1),$$ 
hence, $\pi = {\rm Ps}(\chi_1,\chi_2)$ arises as base change of two principal series representations 
of $\GL_2(F)$ which are ${\rm Ps}(\mu,1)$, and ${\rm Ps}(\mu, \omega_{E/F})$.

\item
$\dim_{{\Bbb C}}{\rm Hom}_{\SL_2(F)}(V,1)=3$,
if $\chi_1\chi_2^{-1}|_{F^\times}=1$, $\chi_1^2=\chi_2^2$, $\chi_1 \not = 
\chi_2$. The $L$-packet containing $V$ has 2 elements, but only one member is  distinguished by $\SL_2(F)$.

In this case, the representation $\pi$ (up to twists) is the base change of a unique
discrete series representation of $\GL_2(F)$, and two 
principal series representation of $\GL_2(F)$ (up to twists):
$${\rm Ps}(\chi_1,\chi_2) \sim  {\rm Ps}(\chi_1/\chi_2,1) = {\rm Ps}(\chi/\chi^\sigma, 1) \sim 
{\rm Ps}(\chi,\chi^\sigma) = {\rm BC}(Ds(\chi)), $$  as well as
$${\rm Ps}(\chi_1,\chi_2) \sim  {\rm Ps}(\chi_1/\chi_2,1) = {\rm Ps}(\mu \circ \N, 1) = {\rm BC(Ps}(\mu, 1))={\rm BC(Ps}(\mu, \omega_{E/F})), $$
where $\mu$ is a character of $F^\times$, $\mu \not = 1, \omega_{E/F}$, with $\mu^2 =1$. 
\end{enumerate}

The following table summarizes the information relevant for us which is contained in the above  four examples.
In examples $(3)$ and $(4)$, the representation $\pi$ is the basechange of respectively 2 and 3 representations of
$\SL_2(E)$, thus there are subcases, listed as ${\rm III}(a),{\rm III}(b), {\rm IV}(a),\cdots$. 
In this table, we have used  ${\rm cok}=\frac{|BC(\pi_1)|}{j|\pi_1|}$ where $j$ is the natural map from group 
of connected components associated to $\pi_1$      to the group of connected components of $\pi= BC(\pi_1)$ which is
nothing but base change of characters of $F^\times$ which are selftwists of $\pi_1$ to characters of $E^\times$ (which are selftwists for $BC(\pi_1)$).  
The last column of the table is what is relevant for us: it shows that by our conjecture in each case the multiplicity contributed to 
the space of $\SL_2(F)$-invariant forms on $\pi$ 
by a representation $\pi_1$ which basechanges to $\pi$, is  1, and therefore for the representation $\pi$ of $\SL_2(E)$, $m(\pi)$, the space of $\SL_2(F)$-invariant forms on $\pi$, equals the number of distinct ways $\pi$ is a basechange from
$\SL_2(E)$.

\vspace{4mm}

\begin{center}
\begin{tabular}{c|c|c|c|c|c|c}
\hline 
Case & $|\pi_1|$ = size of the 
& $|\pi| = |BC(\pi_1)|$  & cok
 & $d(\Phi) $ 
&  $d(\Phi) /{\rm cok}$
\\
& $L$-packet $\pi_1$  &= size of the $L$-packet 
& = $\frac{|BC(\pi_1)|}{j|\pi_1|}$   
& =degree of $\Phi$  &  
\\
&  for $\SL_2(F)$&  $\pi$ for $\SL_2(E)$ & & &  \\
& & & && \\
\hline
& & & && \\
I & 2 & 1 & 1 & 1 & 1 \\
& & & && \\

\hline 
& & & && \\

II & 1 & 2 & 2 & 2 & 1 \\ 
\hline 
& & &  &&\\

III(a) &  1 & 1 & 1 & 1 & 1 \\
\hline 
& & &  &&\\

III(b) &  1 & 1 & 1 & 1 & 1 \\
\hline

& & &  &&\\

IV(a) &  4 & 2 & 1 & 1 & 1 \\
\hline

& & &  &&\\

IV(b) &  2 & 2 & 1 & 1 & 1 \\

\hline

& & &  &&\\

IV(c) &  2 & 2 & 1 & 1 & 1 \\

& & &  &&\\
\hline
\end{tabular}
\end{center}
\vskip 10pt

\section{$\GL(n)$ and $\U(n)$}

{\bf The pair $(\GL_n(E),\GL_n(F))$.}   
In this case, $G^{\out}$, is the 
unitary group $\U_n$ defined by $E/F$. Our conjecture above says that representations of
$\GL_n(E)$ distinguished by $\GL_n(F)$ are precisely those which arise as 
base change of a representation of $\U_n(F)$.  By Proposition \ref{prop5}, the base change 
map taking the  Langlands parameter of a representation of $\U_n(F)$  to one of $\GL_n(E)$ is an
injective map; since multiplicity one for the pair $(\GL_n(E),\GL_n(F))$   
is well-known (an elementary 
result based on the method of Gelfand pairs, cf. [Fli91]), our multiplicity formula 
matches well with known results in this case. It is expected (conjecture of Jacquet, Rallis and Flicker) 
that representations of $\GL_n(E)$ which are distinguished by $\GL_n(F)$ are precisely those which
arise as base change from $\U_n(F)$, and this is known for discrete series representations of $\GL_n(E)$ 
by A. Kable.

{\bf The pair $(\GL_n(E),\U_n(F))$.} It is known that if a 
representation $\pi$ of $\GL_n(E)$ is distinguished by $\U_n(F)$, then it must arise as a base change
of a representation of $\GL_n(F)$.
It has been conjectured by Jacquet in [Jac01] that if $n=2m+1$ is odd, 
and if a representation $\pi$ of $\GL_n(E)$ arises as a base change
of a representation of $\GL_n(F)$, 
then $\dim \Hom_{\U_n(F)}[\pi, \C]$ 
is equal to half the number of representations of $\GL_n(F)$ which base change to the representation
$\pi$ of $\GL_n(E)$. Our conjectures fit well with this since in this case our conjecture will be dealing with two pure
innerforms of the unitary group $\U_n$, say $\U_{\{m,m+1\}}$ and $\U_{\{m+1,m\}}$, 
but these two unitary 
groups are actually identical inside $\GL_n(E)$, therefore for any representation $\pi$ of $\GL_n(E)$ 
$\dim \Hom_{\U_{\{m,m+1\}}(F)}[\pi, \C] = \dim \Hom_{\U_{\{m+1,m \}}(F)}[\pi, \C]$. Observe that if a parameter 
$\sigma'$ for $\GL_{2m+1}(F)$ 
basechanges to the parameter $\sigma$ for $\GL_{2m+1}(E)$, then so does the parameter $\sigma' \otimes\omega_{E/F}$ 
for $\GL_{2m+1}(F)$. Further, by looking at the determinants, it is clear that the parameters $\sigma'$ and $\sigma'\otimes \omega_{E/F}$ for $\GL_{2m+1}(F)$ are distinct. Our conjectures imply that if $\sigma'$ contributes to a linear form 
on $\pi$ invariant under $\U_{\{m,m+1\}}$, then  $\sigma' \otimes \omega_{E/F}$ will 
contribute to a linear form 
on $\pi$ invariant under $\U_{\{m+1,m\}}$.

For $n$ even, our conjectures propose that
for $\U_n(F)$ the quasi-split unitary group over $F$ defined by a 
Hermitian form of dimension $n$ over $E$,    $\dim \Hom_{\U_n(F)}[\pi, \C]$ 
 is at least  the number of  equivalence classes of 
representations of $\GL_n(F)$ under the equivalence $V \sim V \otimes \omega_{E/F}$
which base change to $\pi$ (note that  $V$ and $ V \otimes \omega_{E/F}$ have the same base change
to $\GL_n(E)$); and for $\U_n(F)$ the non-quasi-split unitary group over $F$ 
defined by a Hermitian form of dimension $n$ over $E$,    $\dim \Hom_{\U_n(F)}[\pi, \C]$ 
 is at least  the number of  equivalence classes of 
representations of $\GL_n(F)$ under the equivalence $V \sim V \otimes \omega_{E/F}$ 
but $V \not \cong V \otimes \omega_{E/F}$, 
which base change to $\pi$.

It is known that there are exactly two Hermitian spaces in dimension 1 which are distinguished by their 
Hermitian norm  which is an element of $F^\times/\Nm E^\times = \pm 1$. 
We denote 
$v^+$ to be a vector in a Hermitian space with norm 1 in $F^\times/\Nm E^\times$,  $v^-$ to be a vector with norm $-1$ in 
$F^\times/\Nm E^\times = \pm 1$, and we use the notation $v^{\pm}$ to denote anyone (but exactly one)
 of these two vectors.
 We use this notation to define $\{e_1^{\pm },  e_2^{\pm }, \cdots,   e_n^{\pm }\}$ 
to be any one of the $2^n$ 
 Hermitian space over $E$ of dimension $n$ with orthogonal basis $e_i^{\pm }$. Over a non-Archimedean field $E$, there are
exactly two Hermitian spaces of a given dimension, distinguished by their discriminant, thus exactly $2^{n-1}$ 
of these spaces give rise to each of the two isomorphism classes of Hermitian spaces over $E$ of dimension $n$. 
Define $V^+$ to be an $n$-dimensional vector space over $E$ containing all the vectors $e_i^{\pm }$ 
with product of the signs $= 1$, 
and $V^-$ to be an $n$-dimensional vector space over $E$ containing all the vectors $e_i^{\pm }$ 
with product of the signs $= -1$.

Define $V_i$ to be an $i$-dimensional vector subspace of a fixed vector space $V$ of dimension $n$ 
---which may be either $V^+$ or $V^-$--- to be generated by the $i$-vectors 
$\{e_1^{\pm },  e_2^{\pm }, \cdots,   e_i^{\pm }\}$ choosing one of the signs in $e_j^{\pm }$ for each $j$, with
$$V_1 \subset V_2 \subset V_3 \subset \cdots \subset V_n = V,$$
and different vectors assumed to be orthogonal.

For a given Hermitian structure on $V$ to be either $V^+$ or $V^-$, 
these are distinct orbits of $\U(V^{\pm })$ on the flag variety of $\GL(V)$, 
i.e., give rise to distinct elements in   $\U(V^{\pm })\backslash \GL(V)/B$, with stabilizer which is 
$\U(1)^n$ fixing $e_i^{\pm }$, thus these are open orbits of $\U(V^{\pm })$ on the flag variety of $\GL(V)$, and 
it can be seen to be the only open orbits of $\U(V^{\pm })$ on the flag variety of $\GL(V)$.

Suppose now that we are given a unitary principal series representation $\pi$ on $\GL(V)=\GL_n(E)$ induced by
$n$ unitary characters $\chi_1,\chi_2,\cdots, \chi_n$ of $E^\times$ which each $\chi_i = \mu_i \circ \Nm$
obtained from characters $\mu_i$ of $F^\times$. Such principal series representations are
known to be irreducible representations of $\GL_n(E)$, and  the  Langlands parameter $\sigma_\pi$ of 
$\pi$ is 
\begin{eqnarray*}
\sigma_\pi & = & \chi_1 \oplus \chi_2 \oplus \cdots \oplus \chi_n, \\
& = & \mu_1 \circ \Nm \oplus \mu_2 \circ \Nm \oplus \cdots \oplus \mu_n \circ \Nm.
\end{eqnarray*}
In fact since the norm mapping from $E^\times$ to $F^\times$ has cokernel of order 2, if the characters 
$\chi_i$ are distinct, there are exactly $2^n$ ways of lifting the parameter $\sigma_\pi$ from a parameter
of a representation of $\GL_n(F)$. 

It is easy to see that the restriction of the principal series representation $\pi$ of $\GL(V)$ to each of the open
orbits contributes one dimensional space to the space of $\U(V^{\pm })$-invariant linear form on functions 
supported on that open orbit. In a `generic' situation, these linear forms extend to $\pi$, giving rise to 
$2^n$ linear forms on $\GL(V)$ half of which give $\U(V^+)$-invariant linear forms, and half of which
give $\U(V^-)$-invariant linear forms. This is exactly what our conjecture proposes. The same analysis that we 
have carried out works in the Archimedean case too except that by gluing the lines $\langle e_i^{\pm } \rangle $, 
we do not create
only 2 Hermitian spaces, but Hermitian spaces of all possible signatures $(p,q)$ with $p+q =n$, each occurring 
with multiplicity $\dbinom{p+q}{p}$, which is what the dimension of the space of linear forms $\Hom_{\U(p,q)}[\pi,\C]$
would be, at least in `generic' cases.

\section{Real groups}
Let $G$ be a reductive algebraic group over $\R$. The question of interest for this paper is to classify 
irreducible representations
$\pi$ of $G(\C)$ for which $\Hom_{G(\R)}[ \pi, \C]$ is nonzero, a question of classical interest, and considered by many people, see for example 
the thesis of F. Bien
\cite {Bie}, and the Bourbaki talk of Delorme in \cite {Del} on the work of A. Bouaziz et P. Harinck. In the context of symmetric spaces of real groups ---so certainly in our case, there is the {\it automatic continuity theorem} due to [B-D] according to which an $(\mathfrak h, H \cap K)$-invariant
 linear form on $(\mathfrak g, K)$-modules (for $H$ a symmetric subgroup of $G$ with Lie algebras $\mathfrak h, \mathfrak g$, respectively)
 extends continuously to its smooth Frechet globalization of moderate growth, 
thus we do not need to take topology into account even for real groups.

In some ways, dealing with the distinction question for representations of $G(\C)$ is a bit easier from our point of view since
there are no $L$-packets for $G(\C)$, however, base change map from $G^{\out}$ to $G(\C)$ can still have fibers of multiplicity more than 1.

Here is a basic case of our conjecture. 

\begin{conj} \label{conj4}For an irreducible representations
$\pi$ of $G(\C)$ for which $\Hom_{G(\R)}[ \pi, \omega_G]$ is nonzero, 
we must have $\pi^\sigma = \pi^\vee$. If $G$ is quasi-split over $\R$,
then  $\Hom_{G(\R)}[ \pi, \omega_G]$ is nonzero if and only if the parameter for $\pi$ comes as basechange of a parameter
for $G^{\out}$.
\end{conj}

\begin{remark}
The notation $\pi^\sigma$ for a representation $\pi$ of $G(\C)$ 
implies that it depends on the real form of $G$ over $\R$, but in fact 
if $G_1$ and $G_2$ are two forms of $G$ over $\R$ which are inner twists of each other, then for the corresponding involutions (complex conjugations) $\sigma_1$ and $\sigma_2$ on $G(\C)$,
we have $\pi^{\sigma_1} \cong \pi^{\sigma_2}$. 
This is because to say that $G_1$ and $G_2$ are inner twists of each other means that 
$\sigma_1(g) = k \sigma_2(g) k^{-1}$ for some $k \in G(\C)$.  Clearly then, $\pi^{\sigma_1} \cong \pi^{\sigma_2}$. The following Lemma goes further in this direction, whose proof we omit.
\end{remark}

\begin{lemma}
Let $G$ be a connected reductive group over $\R$, with ${}^LG$ its $L$-group which endows $\widehat{G}$ with an action of ${\rm Gal}(\C/\R)$, call it $A$.
Given an irreducible admissible representation $\pi$ of $G(\C)$  with
Langlands parameter 
$\phi(\pi): \C^\times \rightarrow {}^LG$ of $\pi$, the Langlands parameter of $\pi^\sigma$ is obtained by applying $A$ to the composition of  $\phi$ with  the complex conjugation on $\C^\times$,
i.e., $\phi(\pi^\sigma)(z)= A(\phi(\pi)(\bar{z}))$.  
\end{lemma}

\vspace{2mm}

\noindent{\bf Definition (Discrete parameter):} For a local field $F$, a  Langlands parameter $\phi: W'_F \rightarrow {}^LG,$ 
is said to be discrete if $Z_{\widehat{G}}(\phi)$ is a finite group.

\vspace{2mm}

The following simple and well-known lemma plays an important role in the study of discrete parameters.

\begin{lemma} \label{lemma9}Let $\theta$ be an automorphism of order 2 on a connected reductive group $G$ over $\C$ such that
its fixed points $G^\theta = \{g \in G| \theta(g) = g\}$ is a finite group. 
Then $G$ is a torus, and 
$\theta$ is the involution $\theta(g) = g^{-1}$.
\end{lemma}

\begin{proof}It suffices to prove an analogous statement for Lie algebras. For a torus, the assertion contained in the Lemma is trivial, thus we are reduced
to proving that for a  semisimple group $G$, $G^\theta$ cannot be a finite group unless $G = \langle e \rangle $. If $G^\theta $ is finite, then
$\mathfrak g^\theta = 0$, i.e., $\theta$ operates by $-1$ on $\mathfrak g$, but since $[\theta (X),\theta(Y)] = \theta[X,Y]$, this leads 
to a contradiction to the semisimplicity of $G$, i.e., $[\mathfrak g,\mathfrak g] = \mathfrak g,$ unless $\mathfrak g =0$.
\end{proof}

\vspace{2mm}

\begin{lemma} \label{lemma10}If $\phi: W_\R \rightarrow {}^LG$ is a  discrete parameter, then the centralizer of $\phi(\C^\times)$ in $\widehat{G}$ is a maximal torus $T$ on which 
$W_\R/\C^\times = \Z/2$ operates via the involution $t \rightarrow t^{-1}$. 
\end{lemma}
\begin{proof} Clearly, the centralizer of $\phi(\C^\times)$ in $\widehat{G}$ is a connected reductive group on which the inner conjugation 
action of $\phi(W_\R)$ descends to give an action of 
$W_\R/\C^\times = \Z/2$ by an involution. 
Now we are done by Lemma \ref{lemma9}.
\end{proof}

The following lemma will play an important role in the formulation of our next conjecture. It is a 
 consequence of Lemma \ref{fibers}, but we have preferred to write out a more detailed proof.

\begin{lemma} \label{lemma 11}Given a Langlands parameter $\phi: \C^\times \rightarrow {}^LG$, 
which extends to a discrete Langlands parameter 
$\phi': W_\R \rightarrow {}^LG,$ then  (up to equivalence) there is a  unique extension of $\phi$ to $W_\R$.
\end{lemma}
\begin{proof} Write $W_{\R}= \C^\times \cdot \langle j \rangle$ 
with $j^2 =-1, jzj^{-1} = \bar{z}$ for $z \in \C^\times$.  By Lemma \ref{lemma10}, if $\phi'$ is a discrete parameter, then   
 $Z_{\widehat{G}}(\phi)$ is a maximal torus $T$ (containing $\phi(\C^\times)$) on which $W_\R/\C^\times = \Z/2$ operates by the involution $t \rightarrow t^{-1}$. Clearly the image of $W_\R$ in ${}^LG$ under 
any $\phi'$ extending $\phi$ must normalize $T$. 

Using conjugation in $\widehat{G}$, we can assume that $T$ is {\it the} maximal torus in $\widehat{G}$ used to construct ${}^LG$; in particular $T \rtimes W_\R$ and $N(T) \rtimes W_\R$ 
sit naturally 
inside ${}^LG$ where $N(T)$ is the normalizer of $T$ in $\widehat{G}$.

Any extension of $\phi: \C^\times \rightarrow {}^LG$, to $\phi': W_\R \rightarrow {}^LG,$ is obtained by sending 
$j$ to an element of $N(T) \rtimes W_\R \subset {}^LG$.  If there are two extensions $\phi_1'$ and $\phi_2'$ of $\phi$ to $W_{\R}$, then
the inner-conjugation action of $\phi_1'(j)$ and $\phi_2'(j)$ on $\phi(\C^\times)$ being the same, $\phi_1'(j)\cdot \phi_2'(j)^{-1}$ 
must commute with $\phi(\C^\times)$, i.e., $T\cdot \phi'_1(j) = T\cdot \phi'_2(j)$. It follows that if $\phi_1'(j)$ acts as $t \rightarrow t^{-1}$ on $T$, so does 
$\phi_2'(j)$; thus if  $\phi_1'$ 
is a discrete 
Langlands parameter, so is $\phi_2'$.

An extension of $\phi: \C^\times \rightarrow {}^LG$, to $\phi': W_\R \rightarrow {}^LG,$ is obtained by sending 
$j$ to $tw_0$ where $w_0 $ belongs to $W \rtimes W_\R$ where $W$ is the Weyl group of $T$ in $\widehat{G}$. Note that,
\begin{eqnarray*}
(tw_0)^2 & = & tw_0 tw_0 \\
& = & t w_0 t w_0^{-1}w_0^2\\
& = & t\cdot t^{-1}\cdot w_0^2.
\end{eqnarray*}

It follows (by the generators and relations for $W_\R$) that $\phi'(j) = t\cdot w_0$ gives an extension of $\phi$ to $W_\R$ for one choice of $t \in T$
if and only if it does so for any other choice of $t$ (such as $t=1$).  

We check that if there are two extensions $\phi_1'$ and $\phi_2'$ of $\phi$ to $W_{\R}$, given by
  \begin{eqnarray*}
\phi_1'(j) & = & t_1 \cdot w_0 \\
\phi_2'(j) & = & t_2 \cdot w_0 ,
\end{eqnarray*}
then $ \phi_1'(j) $
and $\phi_2'(j)$ are conjugate by $T$, and hence $ \phi_1' $
and $\phi_2'$ are conjugate by $\widehat{G}$. This is because,
  \begin{eqnarray*}
t \phi_1'(j) t^{-1} &  = & t (t_1w_0)t^{-1} \\
&  = & t t_1w_0t^{-1}w_0^{-1} w_0^2 \\
&  = & t^2t_1w_0^2.
\end{eqnarray*}
Since the squaring map $t \rightarrow t^2$ is surjective on $T$, we have proved the uniqueness of the extension of $\phi$ to $W_\R$. \end{proof}

It seems most reasonable to expect that in $L^2(G(\R)\backslash G(\C))$, 
the discrete spectrum consists exactly of 
those representations of $G(\C)$ which have their Langlands parameters 
basechange of discrete parameters for $G^{\out}$. Therefore, Lemma \ref{lemma 11} together with Conjecture \ref{conj3}  suggests the following conjecture.
(Much of this conjecture is known 
through the works of Flensted-Jensen, Matsuki, Oshima, Bien and others.)
  
\begin{conj} Let $G$ be a real reductive group, and $\pi$ an irreducible unitary representation of $G(\C)$. Then there is at most one pure innerform $G'$ of $G$ over 
$\R$ which must be the quasi-split innerform of $G$ such that $\pi$ 
appears in the discrete spectrum of $L^2(G'(\R)\backslash G(\C))$. If $\pi$ appears in 
the discrete spectrum of $L^2(G'(\R)\backslash G(\C))$, it appears with multiplicity one, and does not appear in 
the continuous part of $L^2(G''(\R)\backslash G(\C))$ for any innerform $G''$ of $G$.  

The discrete spectrum in $L^2(G(\R)\backslash G(\C))$ 
is the base change of the discrete spectrum of $L^2(G^{\out})$, in particular, the discrete spectrum in $L^2(G(\R)\backslash G(\C))$
is tempered, and is nonzero  if and only if  $L^2(G^{\out})$ has discrete spectrum, 
and unless $G^{\out}$ is an innerform of $G$, $L^2(G^{\out})$ has nonzero discrete spectrum if and only if $L^2(G)$ does not have discrete spectrum. 

\end{conj}

\begin{remark}
In our proof of uniqueness of a discrete lift, 
we have made use of  the fact that the centralizer of $\phi(\C^\times)$ in $\wG(\C)$ is a connected group,
which is of course a reflection of the fact that the $L$-packets for $G(\C)$ consist of single elements. 
\end{remark} 
\section{Contribution from the closed Bruhat Cell}\label{17}

Let $G$ be a real reductive group. Since any irreducible representation of $G(\C)$ arises as a subquotient of a principal series
representation of $G(\C)$ induced from a character of a Borel subgroup, to understand the restriction of irreducible 
representations of $G(\C)$ to $G(\R)$, we must understand the orbits of $G(\R)$ on the flag variety $B(\C)\backslash G(\C)$ 
of $G(\C)$. 
It is known, cf. Theorem 3.3 [Wo], that $G(\R)$ has a unique closed orbit which is $P(\R)\backslash G(\R)$ where $P(\R)$ is the minimal parabolic in $G(\R)$,
e.g. a Borel subgroup of $G(\R)$ if $G(\R)$ is quasi-split. Since restriction of functions from $B(\C)\backslash G(\C)$ to $B(\R)\backslash G(\R)$ gives a $G(\R)$ equivariant map to a principal series representation of $G(\R)$, this
allows construction of $\omega_G$-invariant linear forms on the principal series representations of $G(\C)$ which arise as 
basechange of discrete series representations of $G^{\out}(\R)$.  Among the most important aspects of this analysis is that it 
brings out the role of the quadratic character $\omega_G$ rather clearly.

\begin{lemma} \label{lemma 12}Let $G$ be a quasi-split group over $\R$ containing a Borel subgroup $B$ over $\R$, and $T$ a maximal 
torus in $B$.
Then the modulus function for $B(\C) \subset G(\C)$ restricted to $T(\R)$ is the 
square of the modulus function of $B(\R)$.
\end{lemma}
\begin{proof} Let $B=T\cdot N$ be a Levi decomposition of $B$. If $T_s$ (resp. $T_d$) is the maximal split (compact) 
subtorus of $T$ over $\R$, then $T(\R)$ contains $T_s(\R) \cdot T_d(\R)$ as a subgroup of finite index. Since the 
modulus character is a $+$ve character, trivial on compact groups, it suffices to prove the assertion in the Lemma
about modulus character restricted to $T_s(\R)$.   
Since we can decompose $N(\R)$ 
into eigenspaces for $T_s(\R)$, a split torus, the complexification of these eigenspaces gives the 
eigenspace decomposition for $T_s(\C)$, proof of the lemma follows. 
\end{proof}
\begin{lemma} 
Let $G$ be a quasi-split group over $\R$ containing a Borel subgroup $B$ over $\R$, and $T$ a maximal 
torus in $B$ defined over $\R$. For a principal series representation $\pi_\chi$ of $G(\C)$ induced from a character $\chi$ of $T(\C)$
whose restriction to  $T(\R)$ is the restriction of the character $\omega_G$ of $G(\R)$ to $T(\R)$,  
$\dim \Hom_{G(\R)}[\pi_\chi, \omega_G] \not = 0.$
\end{lemma}
\begin{proof}Clearly $B(\R)\backslash G(\R)$ is a $G(\R)$-invariant closed subset of  $B(\C)\backslash G(\C)$. By  Lemma 
\ref{lemma 12}, it follows that $\Ind_{B(\R)}^{G(\R)} (\delta_\R\cdot \omega_G) = (\Ind_{B(\R)}^{G(\R)} \delta_\R) \otimes \omega_G  
 $ is a quotient of $\pi$. But $\Ind_{B(\R)}^{G(\R)} \delta_\R$ has 
the trivial representation of $G(\R)$ as a quotient, proving the lemma.
\end{proof}

\begin{proposition}
Let $G$ be a quasi-split group over $\R$ containing a Borel subgroup $B$ over $\R$, and $T$ a maximal 
torus in $B$ defined over $\R$. For a principal series representation $\pi_\chi$ of $G(\C)$ 
which arises as basechange of a discrete series representation of $G^{\out}$ (in particular, $G^{\out}(\R)$ is an innerform
of a compact group, and hence $G$ is split over $\R$) 
$\dim \Hom_{G(\R)}[\pi_\chi, \omega_G] \not = 0.$
\end{proposition}
\begin{proof}
The proof is now a consequence of the previous lemma combined with Proposition \ref{prop3}
\end{proof}

\section{Contributions  from the open Bruhat Cells}

For a real reductive group $G(\R)$, and a principal series representation $\pi$ of $G(\C)$, we analyze 
the contribution to $\dim \Hom_{G(\R)}[\pi_\chi, \omega_G] $ coming from 
the open orbits of $G(\R)$ on the flag variety $B(\C)\backslash G(\C)$ 
of $G(\C)$ which have especially pleasant structure which we analyze in the following Lemma. The author thanks 
 D. Akhiezer for pointing out  to Wolf's paper, \cite {Wo}, Corollary 4.8 for this result. We have 
decided to give an independent self contained proof.

\begin{lemma} \label{lemma 15}
Let $G$ be a real reductive group with $G(\C)$ its complexification, and $B(\C)$ a Borel subgroup of $G(\C)$. Let $K$ be a maximal compact subgroup of $G(\R)$, and $T$ a maximal torus of $K$. It is known that the centralizer of $T$ in $G(\R)$ is a maximal torus $T_f(\R)$ in $G(\R)$, the so-called fundamental torus of $G(\R)$, with $T_f(\C)$ 
a maximal torus in $G(\C)$. 
Thus, the Weyl group $W_K$ associated to the maximal torus $T$ in $K$ is a subgroup of  
$W_{G(\C)}({T_f(\R)}) = N_{G(\C)}({T_f(\R)}) /Z_{G(\C)}({T_f(\R)})$, the Weyl group of $T_f(\R)$ in $G(\C)$. The 
open orbits for the action of $G(\R)$ on $G(\C)/B(\C)$ 
are all of the form $G(\R)/T_f(\R)$, and are
 in bijective correspondence with $W_{G(\C)}({T_f(\R)})/W_K$.
\end{lemma}

\begin{proof}
Let $B(\C)$ be a Borel subgroup of $G(\C)$ containing $T_f(\C)$.  
We claim that
the $G(\R)$ orbits of $n\cdot B(\C)$ are open in $G(\C)/B(\C)$ 
for $n$ belonging to $N_{G(\C)}(T_f(\R))$, and the orbits corresponding to $n_1$ and $n_2$ are distinct if and only if they define distinct cosets of $W_{G(\C)}({T_f(\R)})/W_K$.

The stabilizer in $G(\R)$  of the point $n\cdot B(\C)$ in $G(\C)/B(\C)$ 
is $G(\R) \cap n B(\C)n^{-1}$. 
The group $G(\R) \cap n B(\C)n^{-1}$ 
is clearly 

\begin{enumerate}
\item a solvable group,

\item defined over $\R$, and

\item contains $T_f(\R)$.

\end{enumerate}

So $G(\R) \cap n B(\C)n^{-1}$ is generated by $T_f(\R)$ and the root spaces corresponding to certain of its roots.  
Since the centralizer of $T$ in $G(\R)$ is $T_f(\R)$, we can actually do the root space decomposition with respect to $T$.
However, $T$ being a compact group, if $\alpha$ is a root of $T$, $\bar{\alpha}= -\alpha$. But 
$G(\R) \cap n B(\C)n^{-1}$ being defined over $\R$, if $\alpha$ is root space inside it, so is $\bar{\alpha}= -\alpha$, 
a contradiction to solvability of $G(\R) \cap n B(\C)n^{-1}$. Thus 
$G(\R) \cap n B(\C)n^{-1}$ has no roots, i.e., $G(\R) \cap n B(\C)n^{-1} = T_f(\R)$. By dimension count, it follows
that the orbits of   $n\cdot B(\C)$ in $G(\C)/B(\C)$ are open, and as $n$ runs over $W_{G(\C)}({T_f(\R)})/W_K$, 
we get distinct open orbits.

Conversely, if the $G(\R)$ orbit passing through the point $g\cdot B(\C)$ in $G(\C)/B(\C)$ is open, then 
$G(\R) \cap g B(\C)g^{-1}$ must have the same dimension as that of $T_f(\R)$. 
Complexifying $G(\R) \subset G(\C)$, 
gives $G(\C) \subset G(\C) \times G(\C)$, and using Bruhat decomposition for $G(\C) \times G(\C)$, we find that 
if the dimension of $G(\R) \cap g B(\C)g^{-1}$ is that of $T_f(\R)$, 
$G(\R) \cap g B(\C)g^{-1}$ 
must itself be a torus in $G(\R)$ which we will prove below to be $T_f(\R)$ 
up to $G(\R)$-conjugacy.   But if $G(\R) \cap g B(\C)g^{-1} = T_f(\R)$, then $T_f(\R) \subset gB(\C)g^{-1}$, hence also
$T_f(\C) \subset gB(\C)g^{-1}$. But it is known that the only Borel subgroups of $G(\C)$ which contain 
a particular maximal torus are the Weyl group conjugates of a fixed Borel containing the torus. 
This easily completes the Lemma.
 \end{proof}
\begin{lemma}Let $G$ be a reductive group over $\R$, $B(\C)$ a Borel subgroup of $G(\C)$ with $G(\R) \cap B(\C)$ a maximal torus of $G(\R)$. Then $G(\R) \cap B(\C)$ must be a fundamental torus of $G(\R)$.
\end{lemma} 
\begin{proof} Let $T$ be the maximal torus over $\R$ with $T(\R)  = G(\R) \cap B(\C)$. Since $T$ is a torus over $\R$, 
the root space decomposition of $G(\C)$ with respect to $T(\R)$ is invariant under $\alpha \rightarrow \bar{\alpha}$.  
From the condition that $T(\R)= G(\R) \cap B(\C)$, there is a system of positive roots $\Sigma^+$ in this root space 
(corresponding to the roots belonging to $B(\C)$) which goes to $\Sigma^-$  under $\alpha \rightarrow \bar{\alpha}$. This implies that all the roots of $T(\R)$ when restricted to
$T_c(\R)$, the maximal compact torus of $T(\R)$, remain nontrivial.    This is easily seen to be the characteristic property
of the fundamental torus of $G(\R)$.
\end{proof}

It is possible to be more specific about the coset space $W_{G(\C)}({T_f(\R)})/W_K$ 
which parametrizes open orbits for the action of $G(\R)$ on $G(\C)/B(\C)$. 
First of all, if $T_f(\R)$ is a compact torus, i.e., $G$ satisfies ${\rm rk}(G) = {\rm rk}(K)$ 
condition, then  $W_{G(\C)}({T_f(\R)}) =  W_{G(\C)}({T_f(\C)})$ 
since automorphisms of $T_f(\C)$ must leave its maximal compact subgroup, i.e. $T_f(\R)$,  invariant. Thus, when 
${\rm rk}(G) = {\rm rk}(K)$, $W_{G(\C)}({T_f(\R)})/W_K = W_{G(\C)}/W_K$.

Next, we note that if an element $n \in G(\C)$ normalizes $T_f(\C)$ taking $T_f(\R)$ to itself, then by the Zariski density
of $T_f(\R)$ in $T_f(\C)$, it follows that the action of $n$ on $T_f(\C)$ is real, and therefore $n^{-1}\bar{n}$ commutes
with $T_f(\C)$, so must belong to $T_f(\C)$. This defines an element of $H^1({\rm Gal}(\C/\R),T_f(\C))$. If  
$H^1({\rm Gal}(\C/\R),T_f(\C)) = 0$, 
it then follows that $n$ can be chosen to be in $G(\R)$. But normalizer 
of $T_f(\R)$ inside $G(\R)$ can be chosen inside $K$ (up to an element of $T_f(\R)$). Therefore, we have the following conclusion.

\begin{lemma} Let $G(\R)$ be a real reductive group with $T_f(\R)$ a fundamental torus. Then if 
$H^1({\rm Gal}(\C/\R),T_f(\C)) = 0$, there is a unique open orbit for the action of $G(\R)$ on $G(\C)/B(\C)$.
\end{lemma}

\noindent{\bf Example :} It is easy to see that for a fundamental torus $T_f$ in $\GL_n(\R)$, 
we have $H^1({\rm Gal}(\C/\R),T_f(\C)) = 0$, hence there is a unique open orbit for the action of $\GL_n(\R)$ on $\GL_n(\C)/B(\C)$.
Same conclusion  for $\SL_n(\R)$ but only for $n$ odd.

\vspace{2mm}

\begin{proposition} Let $G^{\out}$ be a split semisimple group over $\R$ with $T^{\out}$ a maximal split torus in $G^{\out}(\R)$. 
The group $G$ in this case is an innerform of a compact form of $G^{\out}$ containing a maximal compact torus $T$.  
Fix a principal series representation $\pi_\chi$ of $G(\C)$ obtained from a character $\chi$ of $T(\C)$ which is the 
basechange of a character of $T^{\out}(\R)$ via the exact sequence of tori:
$$1 \rightarrow T(\R) \rightarrow T(\C) \rightarrow T^{\out}(\R) \rightarrow 1.$$
Then for a pure innerform $G_\alpha$ of $G$ with maximal compact subgroup $K_\alpha$, there are exactly $W_G/W_{K_\alpha}$ 
many open orbits of $G_\alpha(\R)$ acting on $G(\C)/B(\C)$, each one contributing exactly one linear form on 
part of $\pi_\alpha$ consisting of functions supported on that orbit, which we assume extend to invariant forms 
on $\pi_\alpha$. Then as $G_\alpha$ varies among pure innerforms, 
we get 
$$ \sum_\alpha\Hom_{G_\alpha}[\pi_\chi, \omega_{G_\alpha}] = \sum_\alpha \left |W_G/W_{K_\alpha}\right | = 2^d,$$ 
where $d$ is the rank of $G(\C)$. Further, $2^d$ is the number 
of ways of obtaining the character $\chi$ as a base from characters of $T^{\out}(R)$, and is also the
same as the number of ways the representation  $\pi_\chi$ of $G(\C)$ is obtained as a basechange from $G^{\out}(\R)$.  
\end{proposition}
\begin{proof} We fix $G$ to be a compact group, so that all the groups $G_\alpha$ under consideration are 
 pure innerforms of it. Since the stabilizer of open orbits for the action of $G_\alpha(\R)$ on $G(\C)/B(\C)$ 
are all of the form 
$T(\R)$, the principal series representation $\pi_\chi$  of $G(\C)$ will have nonzero 
$\omega_{G_\alpha}$-invariant form on 
functions supported on any open orbit of $G_\alpha(\R)$ 
as long as the inducing  character $\chi$ of $T(\C)$ and $\omega_{G_\alpha}$ are the same on $T(\R)$. But $T(\R)$ 
being a connected Lie group, the finite order character $\omega_{G_\alpha}$ must be trivial on $T(\R)$. 
Because of the exact sequence,
$$1 \rightarrow T(\R) \rightarrow T(\C) \rightarrow T^{\out}(\R) \rightarrow 1,$$
characters of $T(\C)$ coming as basechange of characters of $T^{\out}(\R)$ are trivial on $T(\R)$.

Thus by Lemma \ref{lemma 15}, each $G_\alpha(\R)$ this way gives rise to exactly $W_G/W_{K_\alpha}$ many invariant linear forms.

Now,
$$ \sum_\alpha \left |W_G/W_{K_\alpha}\right | = 2^d,$$ 
follows because pure innerforms of $G$ are given by $W$-orbits of 
elements of order $\leq 2$ in $T(\R) = \Si^d$ which forms the group $(\Z/2)^d$ with the natural action of $W$. 
If $\alpha$ is an element of  $(\Z/2)^d$, then for the corresponding pure innerform $G_{\alpha}$, the group $K_\alpha$ 
is the centralizer of this element in $G$. It is well-known that the Weyl group of $K_\alpha$ (a subgroup of the Weyl group
of $G$, because they share the same maximal torus) is the stabilizer of $\alpha \in (\Z/2)^d$ in $W$.

Since the mapping  $T(\C) \rightarrow T^{\out}(\R)$ has the form $\C^{\times d} \rightarrow \R^{\times d}$ in some 
co-ordinates, it is clear that there are exactly $2^d$ many characters of $T^{\out}(\R)$ which 
basechange to a given character of $T(\C)$.

The final assertion in the proposition is clear too since two  characters of $T(\C)$ give rise to the 
same Langlands parameter if and only if they are conjugate under the Weyl group. \end{proof}

\section{Unitary groups: the Archimedean case}

Let $\mu: \C^\times \rightarrow \GL_n(\C)$ be the Langlands parameter associated to a representation $\pi$ of $\GL_n(\C)$, 
given by 

$$\mu = \left ( \begin{array}{ccccc} 
  \chi_1 & & & &  \\
& \chi_2 & & & \\
& & * & & \\
& & & \chi_{n-1}& \\
& & & & \chi_n 
\end{array}
\right ) \in \GL_n(\C).$$

Assume that $\pi^\sigma \cong \pi$, and thus the set, call it $X$, of characters $\chi_i$ appearing in $\mu$ are invariant
under $\sigma$. Write, $X=\{(\chi_1,\chi_2),\cdots, ,(\chi_{2k-1},\chi_{2k}), \chi_{2k+1},\cdots, \chi_{n=2k + \ell}\}$
such that 

\begin{eqnarray*}
\chi_{2i+1}^\sigma & = & \chi_{2i+2} 
{\rm ~~~for~~~~} 0\leq i \leq k-1, {\rm with~~~~} \chi_{2i+1}^\sigma \not =  \chi_{2i+1}, \\
\chi_{2k+j}^\sigma & = & \chi_{2k+j} 1 \leq j \leq \ell.
\end{eqnarray*}

\begin{conj} Assume that $\pi$ is an irreducible principal series representation of $\GL_n(\C)$ 
(for example if it is tempered), then
$$\dim \Hom_{U(k+r,k+s)}(\pi, \C)
 = \left( \begin{array}{c} \ell \\ r \end{array} \right ), 
r \geq 0, s \geq 0, r+s = \ell,$$ and $= 0$, otherwise. In particular, $\sum _{p+q = n} \dim \Hom_{U(p,q)}(\pi, \C) = 2^\ell.$
\end{conj}

\begin{remark}
Observe that $\mu$ is the basechange of exactly $2^\ell$ representations:

$$\sum _{i=1}^{i=k} {\rm ind}_{\C^\times}^{W_\R}(\chi_{2i-1}) + \sum \nu_{2k+r}\circ \Nm \cdot \omega_{\C/\R}^{\{0,1\}},$$
where $\nu_{2k+i}$ are characters of $\R^\times$ such that the Galois invariant characters $\chi_{2k+i}$ are obtained
by composing with the norm map $\Nm: \C^\times \rightarrow \R^\times$.

\end{remark}

\begin{remark}
The linear forms appearing in the conjecture are zero on functions supported on the open orbits if $\ell < n$.
\end{remark}

\vspace{5mm}

\noindent{\bf Acknowledgement:} The paper has been long in the making. The author has spoken on 
earlier versions of the paper at various occasions, such as at Banff in the summer of 2011, Sanya in December 2012 among other places. 
In the meanwhile several papers, on the `Contragredient' have appeared: due to Adams \cite{adams}, Adams-Vogan \cite{adams-vo}, Kaletha \cite {kal}.
The author thanks J. Adams, U.K. Anandavardhanan, Wee Teck Gan, Erez Lapid, Birgit Speh and Sandeep Varma for 
helpful comments, and Yiannis Sakellaridis for his encouragement.

\end{document}